\documentclass[10pt]{amsart}

\usepackage[T1]{fontenc}

\usepackage[cal=boondoxo]{mathalfa}
\DeclareSymbolFontAlphabet{\mathcalorig} {symbols}

\usepackage{amssymb,amsmath,amsthm,mathrsfs}

\usepackage{amsfonts}
\usepackage{latexsym}
\usepackage{float}
\usepackage{eulervm}
\usepackage{amscd}
\usepackage{tikz-cd}
\usepackage{longtable}
\usepackage{afterpage}
\usepackage{array}
\usepackage{verbatim}
\usepackage{graphicx} 
\usepackage{color}
\usepackage{fancyhdr} 
\usepackage[labelformat=simple]{caption}
\usepackage{setspace}
\usepackage{bm}
\usepackage{enumitem}
\usepackage{booktabs}
\usepackage{microtype}
\usepackage{lineno}

\usepackage{appendix}
\usepackage{titletoc}

\newenvironment{psmatrix}
{\left[\begin{smallmatrix}}
{\end{smallmatrix}\right]}

\newtheorem{theorem}{Theorem}[section]

\newtheorem*{theorem*}{Theorem}
\newtheorem{proposition}[theorem]{Proposition}
\newtheorem*{proposition*}{Proposition}
\newtheorem{lemma}[theorem]{Lemma}
\newtheorem*{lemma*}{Lemma}
\newtheorem{corollary}[theorem]{Corollary}

\newtheorem*{question*}{Question}
\theoremstyle{definition}
\newtheorem{definition}[theorem]{Definition}

\newtheorem{remark}[theorem]{Remark}

\newtheorem{example}[theorem]{Example}
\newtheorem{examples}[theorem]{Examples}
\newtheorem*{acknowledgements}{Acknowledgements}

\newtheorem{rem}[theorem]{Remark}
\newtheorem{notation}[theorem]{Notation}
\newtheorem{claim}[theorem]{Claim}
\newtheorem{alg}[theorem]{Algorithm}

 \DeclareMathOperator{\Tr}{Tr}

\newcommand{\0}{{(0)}}

\newcommand{\RR}{\mathbb{R}}

\newcommand{\PGL}{\operatorname{PGL}}
\newcommand{\PSL}{\operatorname{PSL}}

\newcommand{\SL}{\operatorname{SL}}

\newcommand{\meas}{\operatorname{\meas}}

\newcommand{\nc}{\newcommand}

\nc{\la}{\langle} \nc{\ra}{\rangle}
 \nc{\CA}{\cal A}
 \nc{\CBB}{\cal B}

\nc{\CDD}{\cal D}
\nc{\CE}{\cal E}
\nc{\CF}{\cal F} \nc{\CG}{\cal
G} \nc{\CH}{\cal H} \nc{\CI}{\cal I} \nc{\CJ}{\cal J}
\nc{\CK}{\cal K} \nc{\CL}{\cal L} \nc{\CM}{\cal M} \nc{\CN}{\cal
N} \nc{\CO}{\cal O} \nc{\CP}{\cal P} \nc{\CQ}{\cal Q}
\nc{\CR}{\cal R} \nc{\CS}{\cal S} \nc{\CT}{\cal T} \nc{\CU}{\cal
U} \nc{\CV}{\cal V} \nc{\CW}{\cal W} \nc{\CZ}{\cal Z}

\makeatletter
\g@addto@macro\bfseries{\boldmath}
\makeatother

\nc{\fa}{  \mathfrak a} \nc{\fg}{  \mathfrak g} \nc{\fk}{  \mathfrak k}
\nc{\fh}{  \mathfrak h} \nc{\fm}{  \mathfrak m} \nc{\fn}{  \mathfrak n}
\nc{\fA}{  \mathfrak A} \nc{\fC}{  \mathfrak C} \nc{\fI}{  \mathfrak I}
\nc{\fL}{  \mathfrak L} \nc{\fS}{  \mathfrak S}
\nc{\fz}{  \mathfrak z}

\nc{\nen}{\newenvironment} \nc{\ol}{\overline}
\nc{\ul}{\underline} \nc{\lra}{\longrightarrow}
\nc{\lla}{\longleftarrow} \nc{\Lra}{\Longrightarrow}
\nc{\Lla}{\Longleftarrow} \nc{\Llra}{\Longleftrightarrow}
\nc{\hra}{\hookrightarrow} \nc{\iso}{\overset{\sim}{\lra}}

\def\N{\mathbb N}

\def\Z{\mathbb Z}

\newcommand{\Mod}{mod}

\newcommand{\trace}{\Tr}
\DeclareMathOperator{\cs}{css}

\definecolor{refkey}{rgb}{1,1,1}
\definecolor{labelkey}{rgb}{1,0,0}

\makeatletter
\@addtoreset{equation}{section}
\makeatother
\numberwithin{equation}{subsection}
 
 \nc{\bq}{\mathbb Q}
 \nc{\br}{\mathbb R}
 \nc{\bz}{\mathbb Z}
 \nc{\bc}{\mathbb C}
 \nc{\bn}{\mathbb N}
 \nc{\bX}{\mathbb{X}}
 \nc{\cL}{ \mathcal{L}}
 \nc{\G}{\Gamma}
 \nc{\sm}{\setminus}
 \nc{\sub}{\subset}
 \nc{\lm}{\lambda}
  \nc{\Lm}{\Lambda}
 \nc{\al}{\alpha}
 \nc{\bt}{\beta}
 \nc{\om}{\omega}
 \nc{\dl}{\delta}
 \nc{\g}{\gamma}
 \nc{\Dl}{\Delta}
 \nc{\Om}{\Omega}
 \nc{\s}{\sigma}
 \nc{\ro}{\rho}
 \nc{\te}{\theta}
 \nc{\SLZ}{\operatorname{SL}_2(\bz)}
 \nc{\SLZp}{\operatorname{SL}_2(\bz[\frac{1}{p}])}
 \nc{\SLZl}{\operatorname{SL}_2(\bz[\frac{1}{\ell}])}
 \nc{\SLR}{\operatorname{SL}_2(\br)}
 \nc{\GLR}{\operatorname{GL}_2(\br)}
 \nc{\PGLR}{\operatorname{PGL}_2(\br)}
 \nc{\PSLR}{\operatorname{PSL}_2(\br)}
 \nc{\PSLZ}{\operatorname{PSL}_2(\bz)}
 
 \nc{\SLC}{\operatorname{SL}_2(\bc)}
 \nc{\uH}{\mathbb H}
 \nc{\fD}{ \mathcal{D}}
 \nc{\fE}{ \mathcal{E}}
 \nc{\fO}{ \mathcal{O}}
 \nc{\haf}{\frac{1}{2}}
 \nc{\qtr}{\frac{1}{4}}
 \nc{\shaf}{{\scriptstyle\frac{1}{2}}}
 \nc{\hlm}{{\scriptstyle\frac{\lambda}{2}}}

 \nc{\8}{\infty}
 \nc{\7}{{-\infty}}
 \nc{\inv}{^{-1}}
 \nc{\eps}{\varepsilon}
 \nc{\aG}{\mathbf{G}}
 \nc{\spn}{\operatorname{Span}}
 \nc{\Cm}{\operatorname{CM}}

\nc{\beq}{\begin{equation}}
\nc{\eeq}{\end{equation}}
\def\ve {\varepsilon}

\nc{\new}{{\color{red}{\label{ATTENTION ******}}}}

\let\LaTeXStandardTableOfContents\tableofcontents
\renewcommand{\tableofcontents}{
\begingroup
\renewcommand{\bfseries}{\relax}
\LaTeXStandardTableOfContents
\endgroup
}

\begin{document}

\renewcommand{\thefootnote}{ } 
\renewcommand{\footnoterule}{{\hrule}\vspace{3.5pt}} 

\title{Commutators in $SL_2$ and Markoff surfaces I}

\author{Amit Ghosh}
\address{Oklahoma State University\\ Stillwater, OK 74078, USA}
\curraddr{}
\email{ghosh@okstate.edu}
\thanks{}

\author{Chen Meiri}
\address{Technion - Israel Institute of Technology ( Haifa Israel)}
\curraddr{}
\email{chenm@tx.technon.ac.il}
\thanks{}

\author{Peter Sarnak}
\address{Institute for Advanced Study and  Princeton University \\Princeton, NJ 08540, USA}
\curraddr{}
\email{sarnak@math.ias.edu}

\dedicatory{To the memory of Vaughan Jones with admiration.}

\begin{abstract} 
We show that the commutator equation over $\SLZ$ satisfies
a profinite local to global principle, while it can fail with infinitely many 
exceptions for $\SL_2(\Z[\frac{1}{p}])$. The source of the failure is a reciprocity obstruction 
to the Hasse Principle for cubic Markoff surfaces.
\end{abstract}

\maketitle

\setcounter{secnumdepth}{5}
\setcounter{tocdepth}{3}
\setcounter{section}{0}
\setcounter{subsection}{4}
\setcounter{subsubsection}{5}

\tableofcontents


\setcounter{footnote}{0}
\setcounter{page}{1}



\renewcommand{\thefootnote}{\arabic{footnote}\quad } 



\section{Introduction}\label{intro}


An element $Z$ in a group $\mathbb{G}$ is a commutator if the equation 
\begin{equation}\label{commutator}
\mathcal{C}_Z:\quad\quad XYX^{-1}Y^{-1}=Z ,
\end{equation}
has a solution with $X$, $Y\in \mathbb{G}$.

There is an extensive literature on understanding the set of commutators for various $\mathbb{G}$'s and the results are decisive for finite simple groups, and for semi-simple matrix groups over local fields and their rings of integers (see \cite{shalev2012}, Conjecture 1.3 and  \cite{agks} for a discussion). In these cases, most if not all elements are commutators. On the other hand if $\mathbb{G}$ is an $S$-arithmetic group, little is known and a description of its commutators appears to be a difficult problem, even for the simplest case that we investigate here, namely $\mathbb{G}=\Gamma_D=\SL_2(D)$ where $D$ is a ring of $S$-integers.

A necessary condition for $Z\in \Gamma_D$ to be a commutator is that it be so in every finite quotient of $\Gamma_D$. If this condition is also sufficient then we say that $\Gamma_D$ satisfies a profinite local to global principle, or a profinite Hasse principle (for commutators). These $\SL_2(D)$'s appear to be good candidates for such a Hasse principle, but one of our main results below shows that there may be subtle diophantine obstructions that intervene.

A fundamental difference between the case that $D$ contains infinitely many units and when it does not, is that in the former case essentially all the finite quotients of $\SL_2(D)$ are congruence quotients (\cite{menn}, \cite{serre70}) and correspondingly the local (finite) obstructions are all simply congruence obstructions. In particular this is true for the Ihara groups $\SL_2([\frac{1}{\ell}])$, with $\ell$  a prime, which will be our central focus. The main results of this paper are that

\ 
\begin{enumerate}[label=(\Roman*).]
\item $\SL_2(\mathbb{Z})$ satisfies the profinite Hasse principle for commutators (and it is crucial that one allows all finite quotients here, and not only congruence quotients).

\

\item For $\ell\equiv \pm 1$ (mod 5) a prime, the commutator Hasse principle fails for $\SL_2([\frac{1}{\ell}])$, and in each case there are infinitely many conjugacy classes which are commutators in every finite (= congruence) quotient, but are not commutators.
\end{enumerate}

\vspace{10pt}

(I) is proved using group theoretic tools. Results of this type are known for the equation $Z=X^m$ (that is $Z$ is a $m$-th power), by Lubotsky for the free group, and Thompson for $\SLZ$ (see \cite{Thomp}).
For a free group $F$, Khelif \cite{Khe} proved that if an element of  $F$ is a commutator in every finite quotient of $F$, then it is a commutator in $F$. Our proof of (I) (see Theorem \ref{thm:main}) follows Khelif's arguments, and we outline the main steps of this proof here. Since $-I_2$ does not belong to the commutator subgroup $\SL_2(\Z)'$ of $\SL_2(\Z)$ and $[\SL_2(\Z):\SL_2(\Z)']=12$, it is enough to prove an analogous theorem for $\PSL_2(\Z)$ instead of $\SL_2(\Z)$. The advantage in working with $\PSL_2(\Z)$ is that it is a free product of cyclic groups of orders 2 and 3. 

The first step of the proof implies that if $d \in \PSL_2(\Z)$ is a commutator in the profinite completion of $\PSL_2(\Z)$, then there exists a subgroup  $G \subseteq \PSL_2(\Z)$ which is generated by two elements and such that $d$ is the commutator of two topological generators of the profinite completion of $G$. In particular, $G$ is not abelian and, by the  Kurosh subgroup theorem, $G = \langle a\rangle * \langle b\rangle$, where  the orders of $a$ and $b$  are either 2 or 3 or $\infty$. The proof of this step is purely combinatorial, and a similar statement holds in an arbitrary free product of finitely many cyclic groups.

In the second step we assume that $d$ is a commutator of two topological generators of the profinite completion of a free product $G_{m,n}$ of two cyclic groups of orders $m,n \in \{2,3,\infty\}$. In each case, we choose a specific embedding $\rho$ of $G_{m,n}$ in $\PSL_2(\Z)$ and show that there are only finitely many possible values $t$ for $|\trace(\rho(d))|$. Every finitely generated subgroup of $\PSL_2(\Z)$ contains only finitely many conjugacy classes of elements of trace $t$. For each such conjugacy class we show  that it either  consists of commutators or that its elements are not commutators in some finite quotient of $G_{m,n}$. The  proof of the second step relies on the congruence structure of $\PSL_2(\Z)$ and the structure of the $\Z[G_{m,n}/G_{m,n}']$-module $G_{m,n}'/G_{m,n}''$ for $m,n \in \{2,3,\infty\}$. We do not know how to generalize it to arbitrary $m,n$. 

\ 

The proof of (II) is number theoretic. Equation \eqref{commutator} for $X$ and $Y$ in $\SL_2(D)$ can be viewed as seeking the $D$-points on the affine variety in $\mathbb{A}^{\!8}$ cut out by the four equations of degree four in the eight variables which are the entries in $X$ and $Y$, together with the two quadratic equations $\det X= \det Y=1$. In the case that $\SL_2(D)$ satisfies the congruence subgroup property, the profinite obstruction to solving \eqref{commutator} is the same as the familiar congruence obstruction to solving \eqref{commutator} over $D$. This puts us in the realm of the usual Hasse principle for integral points on an affine variety, albeit a complicated system when viewed directly.

To study \eqref{commutator} we introduce two intermediate varieties defined over $D$\,: for $t\in D$
\begin{equation}\label{tr-eq}
\mathcal{T}_t : \quad\quad \text{Tr}(XYX^{-1}Y^{-1})=t\, ,
\end{equation}
and the Markoff surfaces for $k\in D$
\begin{equation}\label{Mar-eq}
\mathcal{M}_k : \quad\quad x^2 +y^2 +z^2 -xyz=k\,.
\end{equation}

The relation between these three varieties (for a full discussion, see Section \ref{sec3}) is that a solution $X$, $Y$ to \eqref{commutator} yields a solution $X$, $Y$ to \eqref{tr-eq} with $t=\Tr(Z)$, while a solution to \eqref{tr-eq} yields one to \eqref{Mar-eq} with $x=\Tr(X)$, $y=\Tr(Y)$ and $z=\Tr(XY)$ with $k=t+2$ (which follows from Fricke's identities \cite{FR}). In general, an integral solution of \eqref{Mar-eq} does not {\em lift} to \eqref{tr-eq}, and similarly one of \eqref{tr-eq} need not lift to a solution of \eqref{commutator}. However, we show (see Section \ref{sec3}) that there is an effective procedure to decide if there are $D$-lifts in each of these two steps. This shows that the key diophantne equation that is at the heart of \eqref{commutator} is the cubic affine surface $\mathcal{M}_k$. For $k\neq 4$ these surfaces are nonsingular log K3's (see \cite{CT19}, \cite{whang}) and their diophantine analysis over $D$ is delicate. If $D=\Z$, one can exploit the ``Markoff'' non-linear group of polynomial morphisims of $\mathbb{A}^{\!3}$ which preserve the surfaces $\mathcal{M}_k$, to give a diophantine theory of these surfaces (see \cite{GhSa}). In particular, failures of the Hasse principle for $\mathcal{M}_k$ that have their roots in quadratic reciprocity were discovered in \cite{GhSa}, and have been interpreted as integral Brauer-Manin obstructions in \cite{CT19} and \cite{LM}. In Section \ref{sec5}, we give a streamlined treatment of these reciprocity obstructions to Hasse principles for $\mathcal{M}_k$ and show that they extend to some of the rings $D$ such as the ones occurring in (II) above.

On the other hand, we show that for certain $D$'s, $\mathcal{M}_k$ and $\mathcal{T}_t$ may have no Hasse failures. In fact they can have no local congruence obstructions and \eqref{tr-eq} and \eqref{Mar-eq} are universal for these $D$'s, in that they have solutions for every choice of $t$ and $k$ (in $D$). This demonstrates the delicate nature and complexity of the diophantine properties of $\mathcal{M}_k$ over these rings $D$.

In Section \ref{sec5}, the ``Brauer-Manin'' obstructions to the Hasse principle for the $\mathcal{M}_k$'s are promoted to the Hasse failures of \eqref{commutator} that were mentioned in (II). The precise conjugacy classes which are locally commutators but not globally so, are given in Theorem \ref{HFE1}. The $D$'s in Section \ref{sec4} for which $\mathcal{M}_k$ and $\mathcal{T}_t$ are universal are candidates for $\SL_2(D)$'s for which every $Z$ is a commutator, if indeed such exist.

In paper II of this series, we develop other aspects of the diophantine analysis of $\mathcal{M}_k$ over the base rings $D$. Specifically topological density and descent by the nonlinear morphism group, which when $D$ has infinitely many units acts with infinitely many orbits on $\mathcal{M}_k$ (\cite{Ba05}, \cite{Si90}), unlike the case $\Z$, where there are finitely many orbits. One of the goals is to find a procedure to decide whether a given $Z$ in $\SL_2(D)$ is a commutator. Apparently in going from $\Z$ to such $D$'s, the integer solutions become more random and plentiful even though \eqref{Mar-eq}, \eqref{tr-eq} and \eqref{commutator}  have only sparse solutions.

\begin{acknowledgements}\ 

 It is a pleasure to thank I. Agol, N. Avni, W. Duke, A. Lubotzky and D. Puder for their valuable comments and insights concerning this paper.

AG thanks the Institute for Advanced Study and Princeton University for making possible visits during part of the years 2018-21, including an extended visit to the IAS in Spring 2020, when much of this work took place. He also gratefully acknowledges support from the Mathematics Fund of the IAS, a Simons Foundation grant No. 634846,  and support from his home university for a sabbatical in 2020.

CM thanks the ISF for grant No. 1226/19. CM and PS  thank the BSF for the grant No. 2014099.

PS is also supported by NSF/DMS Grant No. 1302952.
\end{acknowledgements}

 \vspace{20pt}


\section{Local to global principle for the commutator word in $\SL_2(\Z)$}\label{sec2}

\begin{notation}\label{nota:main}\ 
\begin{enumerate}\ 
	\item For every $n \in \N \cup \{\infty\}$, let $C_n$ be a cyclic group of order $n$.  Let
\begin{equation}
	U := \left[\begin{array}{cc}
		0 & -1 \\
		1 & \ 0
	\end{array}\right]
	\text{ and }\ 
	V := \left[\begin{array}{cc}
		0 & -1 \\
		1 & \ 1
	\end{array}\right]\ .
\end{equation}
Then $o(U)=4$, $o(V)=6$, $U^2=V^3$ and $$\SL_2(\Z)=\langle U\rangle*_{\langle U^2\rangle} \langle V\rangle \cong  C_4 *_{C_2} C_6.$$
\item We denote the image of $U$ and $V$ in $\PSL_2(\Z)=\SL_2(\Z)/\{\pm I_2\}$ by $u$ and $v$ respectively.  Then $\PSL_2(\Z)=\langle u \rangle*\langle v \rangle \cong  C_2 *C_3$. 
\item The derived subgroup of a group $G$ is denoted by $G'$ and $G''=(G')'$. 
\item The profinite completion of a group $G$ is denoted by $\widehat{G}$.  
\item An element $g$ in a group $G$ is called a commutator if there exists $x,y \in G$ such that $g=[x,y]:=xyx^{-1}y^{-1}$. If $g \in G$ is a commutator then $g \in G'$ but there might be elements of $G'$ which are not commutators.
\end{enumerate}
\end{notation}

\begin{remark}\ 
	We assume that the reader is familiar with the basic properties of profinite groups and profinite completions (see \cite{RZ} and \cite{Wil} for details). 
\end{remark}

Khelif \cite{Khe} proved the following theorem:

\begin{theorem}[Khelif]\label{thm:khelif}
Let $F$ be a  free group and let $g \in F$. If $g$ is a commutator in $\widehat{F}$ then $g$ is a commutator in $F$.	
\end{theorem}

 Here we prove a similar result, namely

\begin{theorem}\label{thm:main} Let $\Gamma$ be either $\SL_2(\Z)$ or $\PSL_2(\Z)$ and let  $d \in \Gamma$. If $d$ is a commutator in $\widehat{ \Gamma }$ then $d$ is a commutator in $\Gamma$.
\end{theorem} 

We recall that a subset $S$ of a topological group  $G$ is said to topologically 
generate $G$ if the discrete subgroup generated by $S$ is dense in $G$.  The proof of Theorem \ref{thm:main}  follows Khelif's argument closely and   is based on the following two propositions.

\begin{proposition}\label{prop:reduction}
	Let $d \in \PSL_2(\Z)$  be a commutator in $\widehat{ \PSL_2(\Z)}$. Then there exist $a,b \in \PSL_2(\Z)$ such that $d \in G:=\langle a,b\rangle$ and $d$ is the commutator of a pair of topological generators of $\widehat{G}$.	
\end{proposition}

\begin{proposition}\label{prop:commutator} Let $m,n \in \{2,3,\infty\}$ and let $G_{m,n}=\langle a \rangle * \langle b\rangle$ where $o(a)=m$ and $o(b)=n$. If $d \in G_{m,n}$ is a commutator of a pair of topological generators of $\widehat{G}_{m,n}$ then $d$ is a commutator in $G_{m,n}$.
\end{proposition}

\begin{remark}\label{rem} \ 
\begin{enumerate}
	\item For $(m,n)=(\infty,\infty)$, Proposition \ref{prop:commutator} follows from Theorem \ref{thm:khelif}. 
	\item The proof of Proposition \ref{prop:commutator} is easy in the case $(m,n)=(2,2)$. Indeed,  $G_{2,2}$ is the infinite dihedral group $D_\infty = C_\infty \mathbb{o} C_2$ where the conjugation action of the generator $g$ of $C_2$  on the generator $h$ of $C_\infty$ is given by $ghg^{-1}=h^{-1}$. Thus, every element in $H:=\langle h^2 \rangle$ is a commutator in $D_\infty$. On the other hand, $H$ is a normal subgroup in $D_\infty$ and $D_\infty/H \simeq C_2 \times C_2$ is abelian so $H$ is the commutator subgroup  of $D_\infty$. It follows that  an element of $C_\infty \mathbb{o} C_2$ is a commutator if and only if it belongs to $\langle h^2\rangle$. In particular, an element of   $G_{2,2} \simeq  C_\infty \mathbb{o} C_2$ is a commutator if and only if its image in  $C_2 \mathbb{o} C_2 \simeq C_2 \times C_2$ is a commutator. 
	\item We do not know how to prove  Proposition \ref{prop:commutator} nor Theorem \ref{thm:main} for a general free product of two cyclic groups. 
\end{enumerate}
\end{remark}

\begin{proof}[Proof of Theorem \ref{thm:main}]\ 

  Assume first that $\Gamma=\PSL_2(\Z)$. Since $\PSL_2(\Z)$ is isomorphic to $C_2 * C_3$, it follows from the Kurosh subgroup theorem (see Chap.\,3.3 of \cite{LySh}) that every non-trivial subgroup of $\PSL_2(\Z)$ is a free product of cyclic groups and that every one of these cyclic groups is isomorphic to either $C_2$, $C_3$ or  $C_\infty$. It follows from the Grushko-Neumann theorem (see Chap.\,3.3 of \cite{LySh}) that the product of $k$ non-trivial cyclic groups cannot be generated by less than $k$ elements. Thus, every non-abelian subgroup of $\PSL_2(\Z)$ is of the form $\langle a \rangle * \langle b\rangle$ where $o(a)=m$, $o(b)=n$ and  $m,n \in \{2,3,\infty\}$. Theorem \ref{thm:main} then follows from Propositions \ref{prop:reduction} and \ref{prop:commutator}.

 Now assume that $\Gamma=\SL_2(\Z)$. Let $D\in \SL_2(\Z)$ be the commutator of two elements of $\widehat{\SL_2(\Z)}$. The first paragraph implies that there exist  $A,B \in \SL_2(\Z)$ such that  $[A,B]=\pm D$. Let $g$ be a generator of the cyclic group $C_{12}$. The map $U \mapsto g^3$ and $V \mapsto g^2$ can be uniquely extended to a homomorphism from    $\SL_2(\Z)=\langle U\rangle*_{\langle U^2\rangle} \langle V\rangle$ to $C_{12}$. The image of $U^2$ under this homomorphism is $g^6\ne 1$. Since $C_{12}$ is abelian, $U^2=-I_2$ is not contained in $\SL_2(\Z)'$. Thus, it is not possible that both $D$ and  $-D$ are  commutators in the finite quotient $\SL_2(\Z)/\SL_2(\Z)'$. Since $\PSL_2(\Z)/\PSL_2(\Z)'$ is finite, the assumption implies that $D$ is a commutator in $\SL_2(\Z)/\SL_2(\Z)'$ so $-D$ is not a commutator in $\SL_2(\Z)/\SL_2(\Z)'$. Thus it is not possible that $[A,B]= -D$.\end{proof}

\subsection{\ Proof of Proposition \ref{prop:reduction}}\ 

We first summarize the definitions and results we need about groups acting on trees, see \cite{Ser} for  more details.

\begin{enumerate}	
\item The sets of vertices and edges of a graph $X$ are denoted by $V(X)$ and $E(X)$.
\item Graphs are assumed to be oriented. Thus:
	\begin{enumerate}
		\item Every edge $e$ has an origin  $o(e)$ and a terminus $t(e)$. 
		\item There exists an involution $\bar{}:E(X)\rightarrow E(X)$ such that for every $e \in E(X)$, $o(e)=t(\bar{e})$ and $\bar{e} \ne e$. 
		\item There is an orientation $E^+(X)\subseteq E(X)$, so for every $e \in E(X)$, $$|E^+(X)\cap \{e,\bar{e}\}|=1.$$ 
	\end{enumerate}
	\item A tree is a graph such that between any two vertices there exists a unique directed path without backtracking. 
	\item An action of a group $G$ on a tree $X$ is called a good action if for every non-identity $ g \in G$ and every $e \in E^+(X)$, $ge \in E^+(X)$ and  $ge\ne e$. 
	\item If a group $G$ has a good action on a tree $X$, then the quotient graph $G \backslash X$ is oriented. 
\item If $G$ acts on a tree $X$ and $v \in V(X)$ then the stabilizer of $v$ is defined to be $G_v:=\{g \in G \mid gv=v\}$.
\item If $G=H_1* \cdots *H_k$ then there exists a good action of $G$ on a tree $X$ with the following two properties: 
	\begin{enumerate}
	\item For every $1 \le i \le k$, there exists a vertex $v$ such that $H_i=G_v$.
	\item For every $v \in V(E)$, there exists $1 \le i\le k$ such that $G_v$ is conjugate to $H_i$.
	\end{enumerate}
\item  In the special case where $G=H_1* \cdots *H_k$ and each $H_i$ is cyclic, there exists a good action of $G$ on a tree $X$ with the following two properties: 
	\begin{enumerate}
	\item For every $1 \le i \le k$, if $H_i$ is finite then there exists a vertex $v$ such that $H_i=G_v$.
	\item For every $v \in V(E)$, $G_v$ is finite and  either $G_v$ is trivial or there exists $1 \le i\le k$ such that $G_v$ is conjugate to $H_i$.
	\end{enumerate}
\end{enumerate}

\vspace{10pt}
The following Lemma is a special case of the structure theorem of groups acting on trees (Chap.\,5 of \cite{Ser}).

\begin{lemma}\label{lemma:ser}
	 Let $G$ be a group which has a good action on  a tree $X$.
\begin{enumerate}[label=\upshape(\arabic*)]
\item Let $T_G$ be a maximal tree of $X_G:=G\backslash X$ and let $T$ be a lift of $T_G$ to $X$.
\item For every $e \in E^+(X_G)\setminus E(T_G)$, let $\tilde{e}\in E^+(X)$ be a lift of $e$  such that $o(\tilde{e})\in T$ and let $g_e \in G$ be such that $g_et(\tilde{e})\in T$.
\end{enumerate}
Then every $g_e$ is of infinite order and $G$ is the free product $$(*_{v \in V(T)}G_v) *(*_{e \in E^+(X_G)\setminus E(T_G) }\langle g_e\rangle).$$
\end{lemma}

The following corollary is a well known consequence of the Kurosh subgroup theorem.

\begin{corollary}\label{cor:Kurosh}
	Let $K$ be free factor of a group $G$ and let $L$ be a subgroup of $G$. Then $K \cap L$ is a free factor of $L$. 
\end{corollary}
\begin{proof}
	There exists a good action of $G$ on a  tree $X$ such that  $K$ is the stabilizer of some vertex $v$ of $X$.  Let $T_L$ be a maximal tree of $X_L:=L\backslash X$ and let $T$ be a lift of $T_L$ to $X$ such that $v$ is a vertex of $T$. Lemma \ref{lemma:ser} implies that $L\cap K$ is a free factor of $L$.  
\end{proof}

\begin{corollary}\label{bbb}
	Let $G$ be a group which has a good action on a tree $X$ and let $H$ be a subgroup of $G$. 
	\begin{enumerate}[label=\upshape(\arabic*)]
		\item  Let $S_H$ and  $T_G$ be  maximal trees of $X_H:=H\backslash X$  and $X_G:=G\backslash X$ and let $S$  and $T$ be  lifts of $S_H$ and $T_G$ to $X$. 
		\item	 For every $e \in E^+(X_H)\setminus E(S_H)$,   let $\tilde{e}$ be a lift of $e$  to $X$ such that $o(\tilde{e})\in V(S)$ and let $h_e \in H$ be such that $h_et( \tilde{e}) \in V(S)$. 
		\item Let $D \subseteq E^+(X_H) \setminus E(S_H)$ and $U \subseteq V(S)$.		
	\end{enumerate}
	Assume that
	\begin{enumerate}[label=\upshape(\alph*)]
		\item\label{assa} $V(S \cap T)$ contains $U \cup \{o(\tilde{e}),h_et(\tilde{e}) \mid e \in D\}$.
		\item\label{assb} For every $v \in U \cup \{o(\tilde{e}) \mid e \in D\}$, $G_v=H_v$.
	\end{enumerate}
Then $L:=(*_{v \in U} H_v)* (*_{e \in D} \langle h_e\rangle)$	is a free factor of $G$. 
\end{corollary}
\begin{proof}
Let $\pi_H:X \rightarrow X_H$ and $\pi_G:X \rightarrow X_G$ be the quotient maps. 
We  claim that for every distinct $e_1,e_2 \in \{\tilde{e} \mid e \in D\} \cup E(S \cap T)$,   $\pi_G(e_1)\ne \pi_G(e_2)$. Assume otherwise. Assumption \ref{assa} implies that $o(e_1),o(e_2)\in V(S \cap T)$. Since the map $\pi_G$ is injective on $V(T)$, $o(e_1)=o(e_2)$. Thus, there exists $g \in G_{o(e_1)}$ such that $ge_1=e_2$. Assumption \ref{assb} implies that  $G_{o(e_1)}=H_{o(e_1)}$ so $\pi_H(e_1)=\pi_H(e_2)$. This is a contradiction since $\pi_H$ is injective on $E(S) \cup \{\tilde{e} \mid e \in E^+(X_H)\setminus E(S_H)\}$. 

Let $e \in D$. We want to show that $\hat{e}:=\pi_G(\tilde{e}) $ belongs to  $E^+(X_G) \setminus E(T_G)$. Since $\pi_G(t(\tilde{e}))=\pi_G(h_et(\tilde{e}))$, assumption \ref{assa} implies that the subtree $\pi_G(S \cap T)$ of $T_G$ contains a path from $o(\hat{e})=\pi_G(o(\tilde{e}))$ to $t(\hat{e})=\pi_G(t(\tilde{e}))$. Since $\{\tilde{e} \mid e \in D\}$ and $E(S \cap T)$ are disjoint,  the first paragraph implies that $\hat{e} \notin E(\pi_G(S \cap T))$. Since
	$\pi_G(S \cap T)$ is a subtree of $T_G$ and $\pi_G(S \cap T)$  contains the origin and terminus of $\hat{e}$ but does not contain $\hat{e}$, $\hat{e}\notin E(T_G)$ as desired. 

If $e \in D$ then:
	\begin{itemize}
		\item By the first two paragraphs, $\pi_G(\tilde{e}) \in E^+(X_G)\setminus E(T_G)$.
		\item Assumption \ref{assa} implies that  $o(\tilde{e}) \in T$. 
		\item $h_e \in G$ satisfies $h_et(e)\in T$.
	\end{itemize}
Since $U \subseteq V(S \cap T) \subseteq V(T)$, Lemma \ref{lemma:ser} implies that 
$$(*_{v \in U} G_v)* (*_{e \in D} \langle h_e\rangle)$$
	is a free factor of $G$. The result follows since for every $v \in U$, $G_v=H_v$.      
\end{proof}

\begin{corollary}\label{cor:factor}
		Let $G$ be a free product of cyclic groups and let $H_n$ be a decreasing sequence of finite index subgroups of $G$. If $K$ is a finitely generated free factor of $H:=\cap_{n \ge 1}H_n$ then for every large enough $n$, $K$ is a free factor of $H_n$. 
\end{corollary}
\begin{proof} 
Fix a good action of $G$ on a tree $X$ such that the stabilizer of every vertex of $X$ is a finite cyclic group. Let $S_H$ be a maximal tree of $X_H:=H\backslash X$. and let $S$ be a lift of $S_H$ to $X$. For every $e \in E^+(X_H)\setminus E(S_H)$, let $\tilde{e}$ be a lift of $e$ to $X$ such that $o(\tilde{e})\in V(S)$ and choose $h_e \in H$ be such that $h_et(\tilde{e})\in V(S)$. There exists a finite set of vertices $U \subseteq V(S)$ and a finite set of edges $D \subseteq E^+(X_H)\setminus E(S_H)$ such that $K$ is contained in $L:=*_{u \in U}H_u *_{e \in D}\langle h_e \rangle$. Corollary \ref{cor:Kurosh} implies that $K=K \cap L$ is a free factor of $L$ so it is enough to prove that there exists $n$ such that $L$ is a free factor of $H_n$. 

Let $R \subseteq S$ be the minimal subtree of $X$ which contains $U \cup \{o(\tilde{e}),h_et(\tilde{e}) \mid e \in D\}$ and note that $R$ is finite. 
For every $n$, let $\pi_n:X \rightarrow X_{H_n}:=H_n \backslash X$ be the quotient map. Since the stabilizer of every vertex of $X$ is finite, for every $v_1,v_2 \in V(X)$, $\{g \in G \mid gv_1=v_2\}$ is finite. 
Thus, since $\pi_H$ is injective on $V(S)$, for every large enough $n$:
\begin{enumerate}[label=(\alph*)]
	\item\label{item:tree1} $\pi_n$ is injective on $V(R)$ so there exists a maximal tree $T_{H_n}$ of $X_{H_n}$ and a lift $T_n$ of $T_{H_n}$ to $X$ such that $R \subseteq T_n \cap S$.
	\item\label{item:tree2}  For every $v \in V(R) $, $(H_n)_v=H_v$.
\end{enumerate}

The result follows from Corollary \ref{bbb} applied to $H_n$ instead of $G$ and $L$ instead of $H$.
\end{proof}
 
\begin{lemma}\label{lem:Ab}
Let $K$ be the free product of $m$ non-trivial cyclic groups. Then $\widehat{K}$ is not topologically generated by less than $m$ elements.
\end{lemma}
\begin{proof}
	 Abert and Hegedus \cite{AH} proved  that the profinite completion of a free product of $m$ non-trivial finite groups cannot be topologically   generated by less than $m$ elements. The result follows since a free product of $m$ non-trivial cyclic groups projects onto a free product of $m$ non-trivial finite cyclic groups.
\end{proof}

\ 

\begin{proof}[Proof of Proposition \ref{prop:reduction} ]\ 

We can assume that $d \ne 1$. Let $x,y \in \widehat{\PSL_2(\Z)}$ be such that $d=[x,y]$.  

For every $L \le \widehat{\PSL_2(\Z)}$, let $\overline{L}$ be the closure of $L$ in $\widehat{\PSL_2(\Z)}$. For every $n\ge 1$, let $M_n$ be the intersection of all subgroups of $\PSL_2(\Z)$ of index at most $n$. The following properties hold:
	\begin{enumerate}
		\item $(M_n)_{n\ge 1}$ is a decreasing sequence of finite index normal subgroups of $\PSL_2(\Z)$.
		\item For every $L \le \widehat{\PSL_2(\Z)}$, $\overline{L} = \cap_{n \ge 1} L\overline{M_n}$.  
	\end{enumerate}
	 
	For every $n$, denote $H_n:=\overline{ \langle x,y, M_n\rangle }\cap \PSL_2(\Z)$. Then:
	
	\begin{enumerate}[resume]
		\item For every $n \ge 1$, $\overline{H_n}=\overline{\langle x,y,M_n\rangle} $ and $\cap_{n \ge 1}\overline{H_n}=\overline{\langle x,y\rangle}$.
		\item  $H:=\cap_{n \ge 1}H_n$ contains $d$. 
	\end{enumerate} 
	
	Since $H$ is a subgroup of $\PSL_2(\Z)\simeq C_2 * C_3$, Kurosh's subgroup theorem implies that $H$ is a free product of (possibly infinitely many) cyclic groups.  Choose a finitely generated free factor $K$ of $H$ which contains $d$.  Corollary \ref{cor:factor} implies that there exists $n$ such that $K$ is a free factor of $H_n$ so there exists a retraction $\rho:H_n\rightarrow K$. We claim that $\rho$ can be extended to a a retraction $\bar{\rho}:\overline{H_n}\rightarrow \overline{K}$. By the universal property of profinite completions,   $\rho$ can be extended to a continuous homomorphism  $\hat{\rho}:\widehat{H_n}\rightarrow \widehat{K}$. Since $H_n$ has a finite index in $\PSL_2(\Z)$ and $K$ is a free factor of $H_n$, $\widehat{H_n}=\overline{H_n}$ and $\widehat{K}=\overline{K}$ as topological groups.  Thus, there exists an extension of $\rho$ to a continuous homomorphism  $\bar{\rho}:\overline{H_n}\rightarrow \overline{K}$. Since $\rho$ is continuous and $\rho\restriction_K$ is the identity map, $\bar{\rho}\restriction_{\overline{K}}$ is the identity map as desired. 
	
 Since $\overline{K}\subseteq \overline{\langle x,y\rangle}$,  $$\overline{K}=\bar{\rho}(\overline{K})\subseteq
	\bar{\rho}(\overline{\langle x,y\rangle})\subseteq \overline{\langle\bar{\rho}(x),\bar{\rho}(y)\rangle}\subseteq \overline{K}.$$ 

	Since  $\widehat{K}=\overline{K}$ as topological groups,  $\widehat{K}$ is topologically generated by $\bar{\rho}(x),\bar{\rho}(y)$ and $d=\rho(d)=\bar{\rho}(d)=\bar{\rho}([x,y])= [\bar{\rho}(x),\bar{\rho}(y)]$. Kurosh's subgroup theorem implies that $K$ is a free product of  cyclic groups.  Lemma \ref{lem:Ab} implies that  $K$ is  a free product of at most two cyclic groups. Since $d \ne 1$, $K$ is not abelian so $K$ is  the free product of two cyclic groups.	 \end{proof}

\subsection{\ The ring $\Z[G_{m,n}/G_{m,n}']$ and the module $G_{m,n}'/G_{m,n}''$.}

\begin{notation}\label{nota:G_{m,n}} Let $2 \le m \lneq \infty$ and $2 \le n \le \infty$.  Recall that   $G_{m,n}=\langle a \rangle * \langle b \rangle$ where  $\langle a \rangle$ and $\langle b \rangle$ are cyclic groups of orders $m$ and $n$ respectively. For every $i,j$, let $c_{i,j}:=a^ib^j[a,b]b^{-j}a^{-i}$ and let $\bar{c}_{i,j}$ be the image of $c_{i,j}$ in  $G_{m,n}'/G_{m,n}''$. 	Note that for every $i,j$, the image of  $ac_{i,j}a^{-1}$ in $G_{m,n}'/G_{m,n}''$ is equal to $\bar{c}_{i+1,j}$ and the image of  $bc_{i,j}b^{-1}$ in $G_{m,n}'/G_{m,n}''$ is equal to $\bar{c}_{i,j+1}$
\end{notation}
\begin{remark}
	In all the cases that we are interested in, the order $m$ of $a$ is finite. However, in some definitions given below we allow $m$ to be infinite if there is no harm in doing so. 
\end{remark}

The goal of this section is to prove:

\begin{proposition}\label{lemma:conj} Under  Notation \ref{nota:G_{m,n}}, assume that  $$(m,n)\in \{(2,3),(2,\infty),(3,3),(3,\infty)\}.$$  Let $d \in G_{m,n}$ and assume that $d$ is the commutator of a pair of profinite generators of $\widehat{G}_{m,n}$. Then $d$ is conjugate to 
	$[a,b]^{\pm 1}$ modulo $G_{m,n}''$.
\end{proposition}

Note  that for every group $G$, $G'/G''$ is a $\Z[G/G']$-module under the conjugation action and that the $\Z[G/G']$-submodules of $G'/G''$ are the normal subgroup of $G/G''$ which are contained in $G'/G''$. In order to prove Lemma \ref{lemma:conj} we will study the $\Z[G_{m,n}/G_{m,n}']$-module $G_{m,n}'/G_{m,n}''$. 

The following lemma is well known, we only sketch a proof for the convenience of the reader.

\begin{lemma}\label{lemma:free_commutator}
Under  Notation \ref{nota:G_{m,n}}, $G'_{m,n}$ is a free group of rank $(m-1)(n-1)$.
\end{lemma}
\begin{proof}
	The Kurosh subgroup theorem implies that every subgroup of $G_{m,n}$ is a free product of cyclic groups and that the only finite cyclic groups which can appear as factors in this free product are conjugates of subgroups of $\langle a \rangle$ or of $\langle b \rangle$. Since $G'_{m,n}$ is normal and $G'_{m,n} \cap \langle a \rangle=G'_{m,n} \cap \langle b \rangle=1$, $G'_{m,n}$ is a free product of infinite cyclic groups, i.e. $G'_{m,n}$ is a free group.
	
	 The Euler-Poincar\'{e} characteristic of $G_{m,n}$ is $\frac{1}{m}+\frac{1}{n}-1$ (see \cite{Ser} for the definition and properties of the Euler-Poincar\'{e} characteristic). Since  $G_{m,n}'$ is a free subgroup of index $mn$, the Euler-Poincar\'{e} characteristic of $G_{m,n}'$ is $mn(\frac{1}{m}+\frac{1}{n}-1)=m+n-mn$. Thus, the rank of $G_{m,n}'$ is  $1-(m+n-mn)=(m-1)(n-1)$. 
	 \end{proof}

\begin{lemma}\label{lemma:basis}
Under  Notation \ref{nota:G_{m,n}}, assume that $n < \infty$. 
Then $G_{m,n}'/G_{m,n}''$ is a free abelian group with basis $C:=\{\bar{c}_{i,j} \mid 0 \le i \le m-2,0 \le j \le n-2\}$ and for every $i,j$, 
\begin{equation}\label{eq:1}
	\bar{c}_{i,n-1}=(\bar{c}_{i,0})^{-1}\cdots  (\bar{c}_{i,n-2})^{-1} \text{ and }\bar{c}_{m-1,j}=(\bar{c}_{0,j})^{-1}\cdots  (\bar{c}_{m-2,j})^{-1}.
\end{equation}
\end{lemma}
\begin{proof}Since $G'_{m,n}$ is a free group of rank $(m-1)(n-1)$, $G'_{m,n}/G_{m,n}''$ is a free abelian group of rank $(m-1)(n-1)$. Since $C$ contains $(m-1)(n-1)$ elements, if $C$ generates $G_{m,n}'/G_{m,n}''$ then it is a free basis. 

Let $g,h,k$ be elements of a group $G$. The identity $[g,hk]=[g,h]h[g,k]h^{-1}$ and induction implies that  for every $l \ge 1$, 
$$[g,h^l]=[g,h](h[g,h]h^{-1})\cdots (h^{l-1}[g,h]h^{1-l}).$$
Since $b^n=1$, by taking $g=a$, $h=b$ and $l=n$ we the get the identity 
\begin{equation}\label{eq:relation}
	\bar{c}_{0,n-1}=(\bar{c}_{0,0})^{-1}\cdots  (\bar{c}_{0,n-2})^{-1}.
\end{equation}
For every $i,j$, $a^i\bar{c}_{0,j}a^{-i}=c_{i,j}$. It follows form Equation \eqref{eq:relation} that for every $i$,
$$
\bar{c}_{i,n-1}=(\bar{c}_{i,0})^{-1}\cdots  (\bar{c}_{i,n-2})^{-1}.
$$
Since $[x,y]=[y,x]^{-1}$, the proof that for every $j$, 
$$
\bar{c}_{m-1,j}=(\bar{c}_{0,j})^{-1}\cdots  (\bar{c}_{m-2,j})^{-1}.
$$
is similar. 

Denote $c:=c_{0,0}$. For every $g \in G_{m,n}$, there there are $ 0 \le i \le m-1$ and $0 \le j \le n-1$ such that $g\in a^ib^jG'_{m,n}$ so $gcg^{-1} \in c_{i,j}G''_{m,n}$. Since the conjugacy class of $c$ generates $G'_{m,n}$, the set
$$\{\bar{c}_{i,j} \mid 0 \le i \le m-1,0 \le j \le n-1\}$$ generates $G_{m,n}'/G_{m,n}''$. Equation  \eqref{eq:1} implies that $C$ generates $G_{m,n}'/G_{m,n}''$.	
\end{proof}

A similar argument shows:
\begin{lemma}\label{lemma:oof5}
Under  Notation \ref{nota:G_{m,n}}, assume that $n=\infty$. 
Then $G_{m,\infty}'/G_{m,\infty}''$ is a free abelian group with basis $C:=\{\bar{c}_{i,j} \mid 0 \le i \le m-2, j \in \Z\}$ and for every $j$, 
\begin{equation}\label{eq:2}
\bar{c}_{m,j}=(\bar{c}_{0,j})^{-1}\ldots  (\bar{c}_{m-1,j})^{-1}.
\end{equation}
\end{lemma}

\begin{notation}\label{nota:ring} Let $2 \le m < \infty$ and $2 \le n \le \infty$.
\begin{enumerate}  
\item  $\phi_m(X):=X^m-1$.
\item $\psi_m(X):=\frac{\phi_m(X)}{X-1}=X^{m-1}+\ldots+X+1$.
\item  $\phi_\infty(X)=\psi_\infty(X)=0$ ,
\item  $R_{m,n}:=\Z[X,X^{-1},Y,Y^{-1}]/(\phi_m(X),\phi_n(Y))$. 
\item $S_{m,n}:=\Z[X,X^{-1},Y,Y,^{-1}]/(\psi_m(X),\psi_n(Y))$. 
\item $x^{\pm 1}$ and $y^{\pm 1}$ are the images of $X^{\pm 1}$ and $Y^{\pm 1}$  in $R_{m,n}$ and in $S_{m,n}$. 
\end{enumerate}
\end{notation}

\begin{remark}\label{remark:module}
$S_{m,n}$ is a quotient of $R_{m,n}$ so it is both a finitely generated ring and a cyclic  $R_{m,n}$-module.  The action of an element $r \in R_{m,n}$ on an element $s \in S_{m,n}$ is the multiplication of $s$ with the image of $r$ in $S$. In particular, a subset of $S_{m,n}$ is  an $R_{m,n}$-submodule if and only if it is an ideal in $S_{m,n}$. 
\end{remark}

\begin{lemma}\label{cor:isom} Under Notations \ref{nota:G_{m,n}} and \ref{nota:ring}:

\begin{enumerate}[label=\upshape(\arabic*)]
		\item 	There exists a unique ring isomorphism $\rho:\Z[G_{m,n}/G'_{m,n}] \rightarrow R_{m,n} $ such that $\rho(a)=x$ and $\rho(b)=y$.
		\item Assume that $n<\infty$. There exists a unique group isomorphism $\pi$ from $G_{m,n}'/G''_{m,n}$ to the additive group $ S_{m,n} $  such that for every $0 \le i \le m-2$ and every $0 \le j \le n-2$, $\pi(\bar{c}_{i,j})=x^iy^j$.
		\item Assume that $n=\infty$. There exists a unique group isomorphism $\pi$ from $G_{m,n}'/G''_{m,n}$ to the additive group $ S_{m,n} $ such that for every $0 \le i \le m-2$ and every $j \in \Z$, $\pi(\bar{c}_{i,j})=x^iy^j$.
		\item For every $g \in G_{m,n}/G'_{m,n} $ and $h \in G_{m,n}'/G''_{m,n}$, $$\pi(ghg^{-1})=\rho(g)\cdot \pi(h).$$ Thus, $\pi$ is a $\rho$-equivariant  isomorphism of the $\Z[G_{m,n}/G_{m,n}']$-module $G_{m,n}'/G_{m,n}''$  and the $R_{m,n}$-module $S_{m,n}$.
	\end{enumerate}
\end{lemma}
\begin{proof}
	The first claim is clear. The second and third are immediate consequences  of Lemmas \ref{lemma:basis} and \ref{lemma:oof5}, respectively. We check the forth claim. Let $C$ be as in Lemma \ref{lemma:basis} if $n<\infty$ and as in Lemma  \ref{lemma:oof5} otherwise. 
	
	Since $\{a,a^{-1},b,b^{-1}\}$ generate $\Z[G_{m,n}/G'_{m,n}] $ as a ring and $C$ generates $G_{m,n}'/G''_{m,n}$ as a $\Z$-module, it is enough to check that for every $\bar{c}_{i,j}\in C$, $$\pi(a\bar{c}_{i,j}a^{-1})=x^{i+1}y^j \text{ and }  \pi(b\bar{c}_{i,j}b^{-1})=x^iy^{j+1}.$$
	Indeed, for every $\bar{c}_{i,j}\in C$ the following holds: 
	\begin{itemize}
	\item If $0 \le i \le m-3$ then $a\bar{c}_{i,j}a^{-1}=\bar{c}_{i+1,j}$ and $\pi(\bar{c}_{i+1,j})=x^{i+1}y^j$.
	\item If $i=m-2$ then $a\bar{c}_{m-2,j}a^{-1}=\bar{c}_{m-1,j}=(\bar{c}_{0,j})^{-1}\cdots  (\bar{c}_{m-2,j})^{-1}$ and $$\pi\left((\bar{c}_{0,j})^{-1}\ldots  (\bar{c}_{m-2,j})^{-1}\right)=-y^j-\cdots - x^{m-2}y^j=x^{m-1}y^j.$$
	\item If $0 \le j \le n-3<\infty$ then $b\bar{c}_{i,j}b^{-1}=\bar{c}_{i,j+1}$ and $\pi(\bar{c}_{i,j+1})=x^{i}y^{j+1}$.
	\item If $j = n-2<\infty$ then $b\bar{c}_{i,n-2}b^{-1}=\bar{c}_{i,n-1}=(\bar{c}_{i,0})^{-1}\cdots  (\bar{c}_{i,n-2})^{-1}$ and $$\pi\left((\bar{c}_{i,0})^{-1}\cdots  (\bar{c}_{i,n-2})^{-1}\right)=-x^i-\ldots - x^{i}y^{n-2}=x^iy^{n-1}.$$
	\item If $j \in \Z$ and $n=\infty$ then  $b\bar{c}_{i,j}b^{-1}=\bar{c}_{i,j+1}$ and $\pi(\bar{c}_{i,j+1})=x^{i}y^{j+1}$.
\end{itemize}

\end{proof}

\begin{lemma}\label{lemma:d_in_G'}
 	Let $\Gamma$ be a finitely generated group and let $x,y\in \widehat{\Gamma}$. If $d:=[x,y]\in \Gamma$ then $d \in \Gamma'$.
 \end{lemma}
\begin{proof}
	Assume otherwise. The structure theorem of finitely generated abelian groups implies that every such group is residually finite. Since $\Gamma/\Gamma'$ is a finitely generated abelian group and $d \not\in\Gamma'$, there exists a finite abelian  quotient $\Delta$ of $\Gamma$ such that the image of $d$ in $\Delta$ is not trivial. Let $\Lambda$ be  a normal finite index  subgroup of $\Gamma$ such that $\Delta=\Gamma/\Lambda$. Let $\overline{\Lambda}$ be the closure of $\Lambda$ in $\widehat{\Gamma}$. Then $\widehat{\Gamma}/\overline{\Lambda} \simeq \Gamma/\Lambda$  is an abelian group so $d=[x,y]\in\overline{\Lambda}$. On the other hand, since  $\overline{\Lambda} \cap \Gamma=\Lambda $, $d \notin \overline{\Lambda}$, a contradiction.  	
\end{proof}

\begin{lemma}\label{lemma:inv} Under Notations \ref{nota:G_{m,n}} and \ref{nota:ring},
	let $d \in G_{m,n}$ be a  commutator of a pair of profinite generators of $\widehat{G}_{m,n}$. Denote the normal subgroup  of $G_{m,n}$ generated by $d$ by $D$. Then  $DG''_{m,n}=G_{m,n}'$.
	
	In particular, $\pi(d)$ is an invertible element of the ring $S_{m,n}$ where $\pi$ is as in Corollary \ref{cor:isom}.
\end{lemma}
\begin{proof} Denote $c:=c_{0,0}$. Assume that $DG''_{m,n}$ is properly contained in $G_{m,n}'$. Since $c$ normally generates  $G_{m,n}'$, $c \notin DG''_{m,n}$. We will show that there exists a finite index normal subgroup $L$ of $G_{m,n}$ such that $d \in L$ but $c \notin L$.  Then $G_{m,n}/L $ is not abelian. Let $\overline{L}$ be the closure of $L$ in $\widehat{G}_{m,n}$, then $G_{m,n}/L \simeq \widehat{G}_{m,n}/\overline{L}$ so $\widehat{G}_{m,n}/\overline{L}$ is not abelian. On the other hand, $d \in \overline{L}$ and the conjugacy class of $d$ in  $\widehat{G}_{m,n}$ topologically generates $(\widehat{G}_{m,n})'$. Thus $ \widehat{G}_{m,n}/\overline{L}$ must be abelian, a contradiction.

Let $\pi$ be as in Corollary \ref{cor:isom}. Denote $I:=\pi(D)$. Then $I$ is a proper $R_{m,n}$-submodule of $S_{m,n}$ and thus a proper  ideal of the ring $S_{m,n}$ (see Remark \ref{remark:module}). Since every finitely generated ring is residually finite (cf. \cite{OR}) and $\pi(c)\notin I$, there exists an  ideal $J \supseteq I$  such that  $S_{m,n}/J$ is finite and $\pi(c)\notin J$. Remark \ref{remark:module} implies that $J$ is an $R_{m,n}$-submodule of $S_{m,n}$. Thus, there exists a normal subgroup $K$ of $G_{m,n}$ which contains $G_{m,n}''$ such that  $K/G''_{m,n}=\pi^{-1}(J)$. Then $d\in K$ and $c \notin K$.

 If $n<\infty$ then $$[G_{m,n}:K]=[G_{m,n}:G_{m,n}'][G_{m,n}':K]=mn[G_{m,n}':K]<\infty$$  so we can take $L:=K$.
 
Assume that $n=\infty$.  If $H \le G$ are groups and $M$ is a normal subgroup of $G$ then $[G:H] \le [G:MH]|M|$. Since $O(a)=m<\infty$, $$[G_{m,n}/K:\langle b\rangle K/K]\le [G_{m,n}/K,\langle b\rangle G'_{m,n}/K][G'_{m,n}:K]= m[G'_{m,n}:K]<\infty.$$ If a finitely generated group $\Gamma$ has a finite index residually finite subgrup, then $\Gamma$ itself is residually finite. Since $\langle b\rangle K/K$ is an infinite cyclic group, $G_{m,n}/K$ is residually finite. 
 Hence, there exists a normal subgroup $L $ of $G_{m,n}$  such that $L \supseteq K$, $G_{m,n}/L$ is finite and $c \notin L$. Clearly,  $d \in K \subseteq L$. 

The ``in particular'' part follows since $\pi(D)$ is the ideal of $S_{m,n}$ generated by $\pi(d)$. 

 \end{proof}

\begin{corollary}\label{cor:(ab)^3} Under Notations \ref{nota:G_{m,n}} and \ref{nota:ring}, assume that $(m,n)=(3,3)$. Let $r$ be a non-zero integer and denote $h_1:=(ab)^{3r}$, $h_2:=(ba)^{3r}$ and $h_3:=(b^2a^2)^{3r}$.
Then non of $h_1$, $h_2$ and $h_3$ is a commutator of profinite generators of 
$\widehat{G}_{3,3}$. 
\end{corollary}
\begin{proof}
Since $ba=a^{-1}(ab)a$ and $b^2a^2=(ab)^{-1}$, it is enough to prove the claim with respect to $h_1=(ab)^3$. 
 Let $\pi$ be as in Lemma \ref{cor:isom}. It is enough to prove  that for every $1 \le i \le 3$, $\pi(h_1^{3r})$ is not invertible in $S_{3,3}$. Since
 
	$$(ab)^3=[a,b]a^2[a,b]a^{-2}(a^2b^2)[a,b](a^2b^2)^{-1},$$$\pi((ab)^{3r})=r(1+x^2+x^2y^2)$. The claim follows since  $1+x^2+x^2y^2$ is in the kernel of the homomorphism $\eta:S_{3,3}\rightarrow \Z/3\Z$ defined by $\eta(x)=\eta(y)=1$. 
\end{proof}

\begin{lemma}\label{lemma:S_{3,3}}
	Let $(m,n)\in \{(2,3),(2,\infty),(3,3),(3,\infty)\}$. The invertible  elements in $S_{m,n}$ are the elements of the form $\pm x^iy^j$. 
\end{lemma}
\begin{proof}
Denote $T=\Z[X,X^{-1}]/\psi_m(X)$. Then $S_{m,n}\simeq T[Y,Y^{-1}]/(\psi_n(Y))$. 
	If $m=2$ then $\psi_m(X)=X+1$ so $T \simeq \Z$ while if $m=3$ then $\psi_m(X)=X^2+X+1$ so $T$ is isomorphic to the ring of Eisenstein integers $\Z[\xi]$ where  $\xi=\frac{-1+\sqrt{-3}}{2}$ is a primitive third root of unity. 	
	
	If $(m,n)=(2,3)$ then $S_{m,n}\simeq \Z[Y,Y^{-1}]/(\psi_n(Y))\simeq \Z[\xi]$ where the isomorphism from $S_{m,n}$ to $Z[\xi]$ sends $x$ to $-1$ and $y$ to $\xi$. The claim follows  since the invertible elements in the Eisenstein integers are the elements of the from $\pm \xi^j$.
	
		If $(m,n)=(2,\infty)$ then $S_{m,n}\simeq \Z[Y,Y^{-1}]/(\psi_n(Y))\simeq \Z[Y,Y^{-1}]$ where the isomorphism from $S_{m,n}$ to $Z[Y,Y^{-1}]$ sends $x$ to $-1$ and $y$ to $Y$. The claim follows since the invertible elements $\Z[Y,Y^{-1}]$ are the elements of the from $\pm Y^j$.
		
		If $(m,n)=(3,\infty)$ then $S_{m,n}\simeq \Z[\xi][Y,Y^{-1}]/(\psi_n(Y))\simeq \Z[\xi][Y,Y^{-1}]$ where the isomorphism from $S_{m,n}$ to $\Z[\xi]$ sends $x$ to $\xi$ and $y$ to $Y$. The claim follows since the invertible elements $\Z[\xi][Y,Y^{-1}]$ are the elements of the from $\pm \xi^iY^j$.
		
		We are left with the case $(m,n)=(3,3)$. In this case $T \simeq \Z[\xi]$ and
	$$
	S_{3,3} =\Z[X,X^{-1},Y,Y^{-1}]/(\psi_3(X),\psi_3(Y)) \simeq \Z[\xi][Y,Y^{-1}]/(\psi_3(Y)).
	$$
	Since $x^3=1$ and $y^3=1$, every element of the from $\pm x^iy^j$ is invertible in $S_{3,3}$ and there are  18 such elements. Thus, it is enough to prove that there are at most 18 invertible elements in $Z[\xi][Y,Y^{-1}]/(\psi_3(Y))$. The ring $\Z[\xi]$ is a unique factorization domain, $ Y-\xi,Y-\xi^{-1}\in \Z[\xi][Y]$ are coprime and 
	$\psi_3(Y)=Y^2+Y+1=(Y-\xi)(Y-\xi^{-1})$, hence,
	 the map $Y \mapsto (\xi,\xi^{-1})$ induces an injective homomorphism of rings  
	 $$\phi:\Z[\xi][Y,Y^{-1}]/(\psi_3(Y))\rightarrow \Z[\xi] \times \Z[\xi].$$
	 The restriction of $\phi$ to the set of invertible elements  is an injective homomorphism of  the groups of invertible elements    
	 $$\phi^*:\left(\Z[\xi][Y,Y^{-1}]/(\psi_3(Y)\right)^*\rightarrow \Z[\xi]^*\times \Z[\xi]^*.$$
	  Since $\Z[\xi]^*\times \Z[\xi]^*$ contains 36 elements and the image of $\phi^*$ is a subgroup of $ \Z[\xi]^*\times \Z[\xi]^*$, it is enough to show that $\phi^*$ is not onto. We claim that $(1,-\xi^{-1})$ in not in the image of $\phi$ and thus not in the image of $\phi^*$. Every element of $\Z[\xi][Y,Y^{-1}]/(\psi_3(Y))$ is of the form $a+bY+(\psi_3(Y))$ where $a,b \in \Z[\xi]$ and $$\phi^*\left(a+bY+(\psi_3(Y))\right)=(a+b\xi,a+b\xi^{-1}).$$  
	  Thus, if $(1,-\xi^{-1})$ was in the image of $\phi$, then there would have been $a,b \in \Z[\xi]$ such that 
$$
a+b\xi=1\ ,\ a+b\xi^{-1}=-\xi^{-1}.
$$	
This is impossible since then 
$$a=\frac{1+\xi}{1-\xi^2}=\frac{1}{1-\xi}\notin \Z[\xi].$$
\end{proof}

\begin{proof}[Proof of Proposition \ref{lemma:conj}]\ 

Lemma $\ref{cor:isom}$ implies that there is an isomorphism  $\pi$ from  $G'_{m,n}/G''_{m,n}$ to the additive group of $ S_{m,n}$ and that the images under $\pi$ of the conjugates of $[a,b]^{\pm 1}$ are the elements of the from $\pm x^iy^j$. Lemma $\ref{lemma:S_{3,3}}$  implies that the invertible elements in the ring $S_{m,n}$ are elements of the from $\pm x^iy^j$. Finally, Lemmas \ref{lemma:d_in_G'} and  \ref{lemma:inv}  imply that if $d \in G_{m,n}$ is a commutator of topological generators of $\widehat{G}_{m,n}$ then $d \in G'_{m,n}$ and $\pi(d)$ in invertible in $S_{m,n}$. 
\end{proof}

\subsection{\ Trace computations }\ 

In this section we define a specific embedding of $G_{m,n}$ in $\PSL_2(\Z)$ and show that under this embedding there are only finitely many possible values for the trace of an element $d \in G_{m,n}$ which is a commutator of profinite generators of $\widehat{G}_{m,n}$.

\begin{notation}\label{nota:generators} 
Recall that
$$
	U:=\left[\begin{array}{cc}
		0 & -1 \\
		1 & \ \ 0
	\end{array}\right]
	\text{ and }
	V:=\left[\begin{array}{cc}
		0 & -1 \\
		1 & \ \ 1
	\end{array}\right].
$$
For every $(m,n) \in \{(2,3),(2,\infty),(3,3),(3,\infty)\}$ denote:

\ 

\begin{center}
\begin{tabular}{ | c | c| c |c| } 
  \hline

  $(n,m)$ & $A$ & $B$ & $C:=[A,B]$\\ 

  \hline
  &&&\\
  $(2,3)$  & $U=\left[\begin{array}{cc}
		0 & -1 \\
		1 & \ \ 0
	\end{array}\right]$ & $V=\left[\begin{array}{cc}
		0 & -1 \\
		1 & \ \ 1
	\end{array}\right]$ &  $\left[\begin{array}{cc} 
		 2& 1 \\
		1 & 1
	\end{array}\right]$ \\ 
	&&&\\
  \hline
  &&&\\
  $(2,\infty)$ & $U=\left[\begin{array}{cc}
		0 & -1 \\
		1 & \ \ 0
	\end{array}\right]$ & $VUV=\left[\begin{array}{cc} 
		\ \ 0 & \ \ 1 \\
		-1 & -2
	\end{array}\right]$ & $\left[\begin{array}{cc} 
		 5& 2 \\
		2 & 1
	\end{array}\right]$ \\
	&&&\\
  \hline
  &&&\\
  $(3,3)$ & $V=\left[\begin{array}{cc}
		0 & -1 \\
		1 & \ \ 1
	\end{array}\right]$ & $UVU^{3}=\left[\begin{array}{cc} 
		 1& -1\\
		1 & \ \ 0
	\end{array}\right]$ & $\left[\begin{array}{cc} 
		\ \ 1 & -2\\
		-2 & \ \ 5
	\end{array}\right]$\\
	&&&\\
  \hline
  &&&\\
  $(3,\infty)$ & $V=\left[\begin{array}{cc}
		0 & -1 \\
		1 & \ \ 1
	\end{array}\right]$ & $(UV)^3U=\left[\begin{array}{cc} 
		 -3 & 1 \\
		-1 & 0
	\end{array}\right]$ & $\left[\begin{array}{cc} 
		 \ \ 1 & -4\\
		-4 & \ 17
	\end{array}\right]$\\
	&&&\\
  \hline
\end{tabular}
\end{center}
\vspace{10pt}

Finally, for every	$(m,n)\in \{(2,3),(2,\infty),(3,3),(3,\infty)\}$, let $a,b,c$ be the images of $A,B,C$ in $\PSL_2(\Z)$ and denote $G_{m,n}:=\langle a,b\rangle$. Then for every $(m,n) \in \{(2,3),(2,\infty),(3,3),(3,\infty)\}$, $o(a)=m$, $o(b)=n$ and Lemma \ref{lemma:is_free_product} below implies that $G_{m,n}$ is  the free product of $\langle a\rangle$ and $\langle b\rangle$ so the notation defined here is consistent with Notation \ref{nota:G_{m,n}}. 
\end{notation}

\begin{lemma}\label{lemma:good projection}
	 Under Notation \ref{nota:generators},  for every odd prime $p$, $G_{m,n}$ projects onto $\PSL_2(\Z/p\Z)$.
\end{lemma}
\begin{proof}
If $(m,n)=(2,3)$, then $G_{m,n}=\PSL_2(\Z)$ and the claim is clear. Assume $(m,n)\ne (2,3)$. Gauss elimination implies that for every prime $p$, $\SL_2(\Z/p\Z)$ is generated by 
$$
\left[\begin{array}{cc}  1& 1 \\	0 & 1	\end{array}\right]  \text{ and }
\left[\begin{array}{cc}  1& 0 \\	1 & 1	\end{array}\right].
$$
If follows that $p$ is odd, then for every $r,s \in \Z$, the group  $\SL_2(\Z/p\Z)$ is generated by 
$$
\left[\begin{array}{cc}  1& \pm 2^r \\	0 & 1	\end{array}\right]  \text{ and }
\left[\begin{array}{cc}  1& 0 \\	\pm 2^s & 1	\end{array}\right].
$$
The result follows  the following computations: 
\begin{enumerate}
	\item If $(m,n)=(2,\infty)$ then 
$$
AB=\left[\begin{array}{cc}  1& 2 \\	0 & 1	\end{array}\right] \text{ and }
A^2BA^{-1}= \left[\begin{array}{cc} \ \  1 & 0 \\ -2 & 1	\end{array}\right].
$$
	\item If $(m,n)=(3,3)$ then 
$$
AB=\left[\begin{array}{cc}  -1 &\ \  0 \\	\ \ 2 & -1	\end{array}\right] \text{ and }
BA= \left[\begin{array}{cc}  -1 & -2 \\ \ \  0 & -1	\end{array}\right].
$$
	\item If $(m,n)=(3,\infty)$ then 
$$
BA=\left[\begin{array}{cc}  1 & 4 \\	0 & 1	\end{array}\right] \text{ and }
AB= \left[\begin{array}{cc}  \ \ 1 & 0 \\ -4 & 1	\end{array}\right].
$$
\end{enumerate}	
\end{proof}

The following lemma is well known and is proved only for the convenience of the reader: 
\begin{lemma}
	Let $p$ be an odd prime. Then the exact sequence $$1 \rightarrow \{\pm 1\} \rightarrow \SL_2(\Z/p\Z)\rightarrow \PSL_2(\Z/p\Z)\rightarrow 1$$ does not split
\end{lemma}
\begin{proof}
Let $\Gamma$ be a subgroup of $\SL_2(\Z)$ which projects onto $\PSL_2(\Z)$. Then there exists  $\epsilon=\pm 1$ such that 
$$
\left[\begin{array}{cc}  \epsilon & 1 \\	0 & \epsilon	\end{array}\right] \in \Gamma \text{ so }\left[\begin{array}{cc}  1 & 2 \\	0 & 1	\end{array}\right] = \left[\begin{array}{cc}  \epsilon & 1 \\	0 & \epsilon	\end{array}\right]^2 \in \Gamma.
$$
Similarly, $
\left[\begin{array}{cc}  1 & 0 \\	2 & 1	\end{array}\right] \in \Gamma
$ so $\Gamma=\SL_2(\Z)$.
\end{proof}

\begin{definition}
The projection map $\rho:\SL_2(\Z)\rightarrow \PSL_2(\Z)$ induces an epimorphosm $\hat{\rho}:\widehat{\SL_2(\Z)}\rightarrow \widehat{\PSL_2(\Z)}$ and $\ker \rho=\ker \hat{\rho}=\{\pm I_2\}$. Since   $-I_2 \notin \widehat{\SL_2(\Z)}'$, the restriction of $\hat{\rho}$ to $\widehat{\SL_2(\Z)}'$ is an isomorphism from $\widehat{\SL_2(\Z)}'$ onto $\widehat{\PSL_2(\Z)}'$. We refer to the inverse of this isomorphism as the commutator subgroup section and  denote it by $\cs:\widehat{\PSL_2(\Z)}'\rightarrow \widehat{\SL_2(\Z)}'$. Note that  $\cs(\PSL_2(\Z)')=(\SL_2(\Z))'$. 

Since  $\ker \hat{\rho}$ is central in $\widehat{\SL_2(\Z)}$, for every $x,y \in \widehat{\PSL_2(\Z)}$ and every lifts $X,Y \in \SL_2(\Z)$  of $x$ and $y$ to $\widehat{\SL_2(\Z)}$, $\cs([x,y])=[X,Y]$. 
\end{definition}

\begin{lemma}\label{lemma:trace} Under Notation \ref{nota:generators}, if $d\in G_{m,n}$ is the commutator of a pair of profinite generators of $\widehat{G}_{m,n}$ then there exists  $k \ge 0$ such that  $\trace (\cs(d))= 2 \pm 2^k$.  
\end{lemma}
\begin{proof}
 Since a homomorphism between groups induces a homomorphism between their profinite completions, we have the following commutative squares:
	
	\[
\begin{tikzcd}
\langle A,B \rangle \arrow{r}{\alpha}\arrow[two heads]{d}{\psi} & \arrow[two heads]{d}{\rho} \SL_2(\Z) \\
G_{m,n} \arrow{r}{\beta} & \PSL_2(\Z)
\end{tikzcd} 
\text{ and }
\begin{tikzcd}
\widehat{\langle A,B \rangle} \arrow{r}{\hat{\alpha}}\arrow[two heads]{d}{\hat{\psi}} & \arrow[two heads]{d}{\hat{\rho}} \widehat{\SL_2(\Z)} \\
\widehat{G_{m,n}} \arrow{r}{\hat{\beta}} & \widehat{\PSL_2(\Z)}
\end{tikzcd} 
\]
where $\rho$, $\psi$, $\hat{\rho}$ and $\hat{\psi}$ are epimorphisms, $\alpha$ and $\beta$ are inclusions and $\ker \rho =\ker\hat{\rho}=\{\pm 1\}$.

Assume that $x,y$ are topological  generators of $\widehat{G_{m,n}} $ such that $[x,y]=d$. Let $X,Y \in \widehat{\SL_2(\Z)}$ be lifts of $\hat{\beta}(x)$ and $\hat{\beta}(y)$ to $\widehat{\SL_2(\Z)}$.  Then $$[X,Y]=\cs([\hat{\beta}(y),\hat{\beta}(y)])=\cs(\hat{\beta}([x,y]))=\cs(\beta(d))=\cs(d)\in \SL_n(\Z)'.$$ Thus, we have to  show that $\trace[X,Y]= 2 \pm 2^k$. 

For every odd prime $p$, let $\pi_p:\widehat{\SL_2(\Z)}
\twoheadrightarrow \SL_2(\Z/p\Z)$ and $\bar{\pi}_p:\widehat{\PSL_2(\Z)}
\twoheadrightarrow \PSL_2(\Z/p\Z)$ be the modulo-$p$ homomorphisms. Lemma 
\ref{lemma:good projection} implies that 
$$\PSL_2(\Z/p\Z)=(\bar{\pi}_p\circ{\beta}) (G_{m,n})=(\bar{\pi}_p \circ \hat{\beta})(\langle x,y\rangle)=(\bar{\pi}_p\circ\hat{\rho})(\langle X,Y\rangle).$$
Since the sequence $$1 \rightarrow \{\pm 1\} \rightarrow \SL_2(\Z/p\Z)\rightarrow \PSL_2(\Z/p\Z)\rightarrow 1$$ does not split, $\pi_p(\langle X,Y\rangle)=\SL_2(\Z/p\Z)$. The result follows from the fact that for an odd prime $p$,  the trace of the  commutator of two generators of $\SL_2(\Z/p\Z)$ cannot be equal to 2 (see, for example, Lemma 3 of \cite{Khe}). 
	\end{proof}

\begin{claim}\label{claim1}
	Let $k \in \Z$, $X:=\left[\begin{array}{cc} 
		 1& k \\
		0 & 1
	\end{array}\right]$ and $Y:=\left[\begin{array}{cc} 
		 0& -1\\
		1 & \ \ 1
	\end{array}\right]$. Then $$
	[X,YXY^{-1}]=
	\left[\begin{array}{cc} 
		 1-k^2+k^4& k^3\\
		k^3 &  1+k^2
	\end{array}\right]\equiv
	\left[\begin{array}{cc} 
		 1-k^2& 0\\
		0 &  1+k^2
	\end{array}\right]  (mod\ k^3).$$
\end{claim}

\begin{lemma}\label{lemma:nc}
	Let $H$ be a group which is generated by two elements $h_1,h_2$ and $o(h_1)=m<\infty$. Then:
	\begin{enumerate}[label=\upshape(\arabic*)]
		\item\label{aaa} $H' \subseteq N:=\langle h_1^i(h_1h_2)h_1^{-i} \mid 0 \le i \le m-1 \rangle$.
		\item If $h_1^2$ is central in $H$, then $H''$ is contained in the subgroup generated by $H$-conjugates of $[h_1^2h_2h_1^{-1},h_1h_2]$.
		\item\label{bbb2} If $h_1^3$ is central in $H$, $H''$ is contained in the subgroup generated by $H$-conjugates of $[h_2h_1,h_1(h_2h_1)h_1^{-1}]$.
	\end{enumerate}
\end{lemma}
\begin{proof}
	We start by proving  the first claim.  Note that for every $k$, $h_1^kNh_1^{-k}=N$. 	 Therefore, in order to show that $N$ is normal in $H$, it is enough to show that $h_2Nh_2^{-1}=N$ and $h_2^{-1}Nh_2=N$. Since $h_1h_2$ and $h_1^{m}h_2h_1^{m-1}=h_2h_1^{m-1}$ belong to $N$,
	$$
	h_2Nh_2^{-1}=(h_2h_1^{m-1})N(h_2h_1^{m-1})^{-1}=N \text{ and }
	h_2^{-1}Nh_2=(h_1h_2)^{-1}N(h_1h_2)=N. 
	$$
	Since $h_1N=h_2^{-1}N$, the quotient $H/N$ is cyclic and in particular abelian so $H' \subseteq N$. 
	
	We now prove the second claim. Since $h_1^2$ is central in $H$, $N=\langle h_1h_2,h_1^2h_2h_1^{-1}\rangle $. Denote $g:=[h_1^2h_2h_1^{-1},h_1h_2]$. Then, 
	$$
	H'' \subseteq N' = \langle hgh^{-1} \mid h \in N\rangle \subseteq  \langle hgh^{-1} \mid h \in H\rangle.
	$$
	
	Finally, we prove the third claim. For every $k$, denote $g_k:=h^k_1(h_2h_1)h_1^{-k}$.
	Since $N$ contains $H'$, it is enough to prove that $N'$ is  
	contained in the subgroup generated by $H$-conjugates of $[g_0,g_1]$.
	Since $h_1^3$ is central in $H$, if $k=k'\ (\Mod\ 3)$ then $g_k=g_{k'}$.
	Since $h_1^{-1}(h_1h_2)h_1=h_2h_1$, $$N=\langle g_k \mid 0 \le i \le m-1 \rangle=\langle g_0,g_1,g_2 \rangle.$$
 Thus $N'$ is  the subgroup of $N$ generated by the $N$-conjugates of the elements $[g_0,g_1]$, $[g_1,g_2]$ and $[g_2,g_0]$. The result follows since $[g_1,g_2]=h_1[g_0,g_1]h_1^{-1}$ and  $[g_2,g_0]=[g_2,g_3]=h_1^2[g_0,g_1]h_1^{-2}$.
\end{proof}

\begin{lemma}\label{lemma:G''} Under Notation \ref{nota:generators}, 
\begin{enumerate}[label=\upshape(\arabic*)]
	\item If $(m,n)=(2,3)$ then the elements of $\langle A,B \rangle'' $ are equal to $I_2$ modulo 2. 
	\item If $(m,n)=(2,\infty)$ then the elements of $\langle A,B \rangle'' $ are equal to $I_2$ or to $5I_2$ modulo 8.  
	\item If $(m,n)=(3,3)$  then the elements of $\langle A,B \rangle'' $ are equal to $I_2$ or to $5I_2$ modulo 8.  
	\item  If $(m,n)=(2,\infty)$ then the elements of $\langle A,B \rangle'' $ are equal to $I_2$ or to $17I_2$ modulo 32.  
\end{enumerate}
\end{lemma}
\begin{proof}\ 
	\begin{enumerate}[label=\upshape(\arabic*)]
		\item $\SL_2(\Z/2\Z)\simeq S_3$ and $S_3''=1$.
		\item  Since $o(A)=4$ and $A^2=-I_2$, Lemma \ref{lemma:nc} implies that 
		$ \langle A,B \rangle''$ is contained in the subgroup generated by $SL_2(\Z)$-conjugates of 
		  $$[A^2BA^{-1},AB]=\left[\begin{array}{cc} 
		\ \  5 & -8 \\
		-8 &\  13
	\end{array}\right].$$
	The result follows since $5^2=1\ (\Mod\ 8)$.
		\item Since $o(A)=6$ and $A^3=-I_2$, Lemma \ref{lemma:nc} implies that 
		$ \langle A,B \rangle''$ is contained in the subgroup generated by $SL_2(\Z)$-conjugates of $[BA, A(BA)A^{-1}]$. Since $5^2=1\ (\Mod\ 8)$, the result follows from Claim \ref{claim1} with respect to
 $$ 
 X=BA= \left[\begin{array}{cc}  -1 & -2 \\ \ \ 0 & -1	\end{array}\right]\text{ and }
Y=A=\left[\begin{array}{cc} 0& -1\\	1 & \ \ 1 \end{array}\right] .
		$$
 \item Since $o(A)=6$ and $A^3=-I_2$, Lemma \ref{lemma:nc} implies that 
		$ \langle A,B \rangle''$ is contained in the subgroup generated by $SL_2(\Z)$-conjugates of $[BA, A(BA)A^{-1}]$.
Since $(17)^2=1\ (\Mod\ 32)$, the result follows from Claim \ref{claim1} with respect to
 $$ 
 X=BA= \left[\begin{array}{cc}  1& 4 \\ 0 & 1	\end{array}\right]\text{ and }
Y=A=\left[\begin{array}{cc} 0 & -1\\	1 & \ \ 1 \end{array}\right] .
		$$
\end{enumerate}
\end{proof}

\begin{lemma}\label{lemma:few_traces}  Under Notation \ref{nota:generators}, assume that $d \in G_{m,n}$ is the commutator of profinite generators of $\widehat{G}_{m,n}$. Then the possible values for $\trace(\cs(d))$ are given in Table \ref{table:pos}.
\begin{table}[!ht]
\begin{center}
\begin{tabular}{ | c | c| c |c|c| }
  \hline
  $(n,m)$ & $C$ & $\lambda$ & $l$ & $\trace(\cs(d))$\\ 
  \hline
  &&&&\\
  $(2,3)$  &   $\left[\begin{array}{cc} 
		 2& 1 \\
		1 & 1
	\end{array}\right]$& $1$ & $2$ & $1$ or $3$\\ 
	&&&&\\
  \hline
  &&&&\\
  $(2,\infty)$  & $\left[\begin{array}{cc} 
		 5& 2 \\
		2 & 1
	\end{array}\right]$ & $1$ or $5$& $8$ & $-2$ or $6$ \\
	&&&&\\
  \hline
  &&&&\\
  $(3,3)$  & $\left[\begin{array}{cc} 
		 \ \ 1& -2\\
		-2 & \ \ 5
	\end{array}\right]$
	&$1$ or $5$& $8$ & $-2$ or $6$\\
	&&&&\\
  \hline
  &&&&\\
  $(3,\infty)$  & $\left[\begin{array}{cc} 
		 \ \ 1 & -4\\
		-4 & \ 17
	\end{array}\right]$ &$1$ or $17$ & $32$ & $-14$ or $18$\\
	&&&&\\
  \hline
\end{tabular}
\captionof{table}{\label{table:pos} Possible values of $\trace(\cs(d))$}
\end{center}
\end{table}
\end{lemma}
\begin{proof}

Lemma \ref{lemma:conj} implies that the images of $c=[a,b]$ and $d$ in $G_{m,n}/G_{m,n}''$ are conjugates so there exist $r \in G_{m,n}$ and $s \in G_{m,n}''$ such that $d=rcr^{-1}s$. Let $A$, $B$ and $C$ be as in Notation  \ref{nota:generators}. Let $R$ be a lift of   $r$ to $\SL_2(\Z)$ so $\cs(c)=[A,B]=C$ and 
$$\cs(d)= RCR^{-1}\cs(s)\in \SL_2(\Z)'.$$ 

 Lemma \ref{lemma:G''}  implies that $\trace (\cs(d))=\lambda\trace(C) \ (\Mod\ l)$ where  $C$, the value of $l$ and the possible values of $\lambda$ are given in Table \ref{table:pos}. 

Lemma \ref{lemma:trace} implies that there exists $k \in \N$ such that $\trace(\cs(d))=2\pm 2^k$. Hence, the only possible values for $\trace(\cs(d))$ are the values written in the right column of Table \ref{table:pos}.

\end{proof}

\subsection{\ Proof of Proposition \ref{prop:commutator}}\ 

The goal of this section is to complete the proof of Proposition \ref{prop:commutator}. 

\begin{definition}
	The trace of an element $g \in \PSL_2(\Z)$, denoted by $|\trace g|$, is the absolute value of the trace of its lifts to $\SL_2(\Z)$. 
\end{definition}

Let $d \in G_{m,n}$ be  a commutator of a pair of topological generating set of $\widehat{G}_{m,n}$. Lemma  \ref{lemma:few_traces}  gives the possible values of the trace of $d$. We will describe an  algorithm such that for a given $t$, the algorithm computes a representatives set to the finitely many  $G_{m,n}$-conjugacy classes of trace $t$. By applying this algorithm to the finitely many possible values of the trace of $d$, we will get a finite list of elements such that $d$ is $G_{m,n}$-conjugate to one of them. Then we will show that for every element in this list which in not a commutator, there exists a finite quotient in which this element is not a commutator, 
 
\begin{definition}
	Let $G=\langle a \rangle * \langle b \rangle$ where $o(a)=m$ and $o(b)=n$. If $m<\infty$ define $L_a:=\{a\}$, otherwise, define $L_a:=\{a,a^{-1}\}$. Define $L_b$ similarly. An $L_{a,b}$-word  is a finite sequence of elements of $L_{a,b}=:L_a \cup L_b$ (including the empty sequence). The set of $L_{a,b}$-words is denote by $W_{a,b} $.  If $w$ is a
	 word then its length is denote by $l(w)$. We say that a  word $(x_1,\ldots,x_k)\in W_{a,b}$ represents the element $g \in  G_{m,n}$ if $g=x_1\cdots x_k$ (the empty word represents the identity element). If $w_1$ and $w_2$ are words, then $w_1w_2$ denotes their concatenation.

\end{definition}

We recall some standard definitions and facts about $L_{a,b}$-words.

\begin{enumerate}[label=\upshape(\arabic*)]
	\item An $L_{a,b}$-word  is called reduced if the following two conditions hold: 
		\begin{enumerate}[label=\upshape(\alph*)]
			\item If $m < \infty$ then there are no $m$ consecutive appearances of $a$ and if $n < \infty$ then there are no $n$ consecutive appearances of $b$.
			\item If $m=\infty$  then there are no  consecutive appearances  of the from $a,a^{-1}$ or $a^{-1},a$ and if $n=\infty$ then there are no  consecutive appearances of the from $b,b^{-1}$ or $b^{-1},b$.
		\end{enumerate}
	\item For every $g \in G$ there exists a unique reduced word  which represents $g$.  
	\item An $L_{a,b}$-word is called cyclically reduced if every cyclic permutation of it is reduced. 
	\item \begin{enumerate}[label=\upshape(\alph*)]
	\item  If $z$ is a reduced word which represents an element $g$ we denote by $z^{-1}$ the reduced word which represents $g^{-1}$.  
	\item  For every conjugacy class $C$ of $G$ there exists a cyclically reduced word  $w$, unique up to cyclic permutation, such that every element  of $C$ is represented  by a reduced word of the form $zw'z^{-1}$ where  $w'$ is a cyclic permutation of $w$ and $z$ is a reduced word. 
In particular, the only elements of $C$ which are represented by a cyclically reduced word are the elements which are represented by a cyclic permutation of $w$.
\end{enumerate}
\end{enumerate}

We recall the description of the conjugacy classes of $\PSL_2(\Z)$ with a given trace. The following two Lemmas are well known. We sketch the proofs only for the convenience of the reader.

\begin{lemma}\label{lemma:conjugacy_class_SL_12}
\begin{enumerate}[label=\upshape(\arabic*)]
	\item If $g \in \PSL_2(\Z)$ and $|\trace{g}|=0$ then $g$ is conjugate to $u$. 
	\item If $g \in \PSL_2(\Z)$ and $|\trace{g}|=1$ then $g$ is conjugate to $v$ or to $v^2$. 
	\item If $g \in \PSL_2(\Z)$ and $|\trace{g}|=2$ then $g$ is conjugate to $(uv)^r$ for some $r \in \Z$. 
\end{enumerate}	
\end{lemma}
\begin{proof}\ 

	\noindent1. Let  $g \in \PSL_2(\Z)$ with $|\trace{g}|=0$. Then $g$ can be lifted to an element  $\tilde{g} \in \SL_2(\Z)$. The characteristic polynomial of $\tilde{g}$ is $T^2+1$ so  $g^2=1$. Every element of order 2 in the free product  $\PSL_2(\Z)=\langle u \rangle*\langle v\rangle$ is conjugate  to $u$.
	\\\\
	2. Let  $g \in \PSL_2(\Z)$ with $|\trace{g}|=1$. Then $g$ can be lifted to an element  $\tilde{g} \in \SL_2(\Z)$  with $\trace{\tilde{g}}=-1$. The characteristic polynomial of $\tilde{g}$ is $T^2+T+1=\frac{T^3-1}{T-1}$ so $g^3=1$. Every element of order 3 in the free product  $\PSL_2(\Z)=\langle u \rangle*\langle v\rangle$ is conjugate either to $v$ or to $v^2$.	
	\\\\
	3. Let  $g \in \PSL_2(\Z)$ with $|\trace{g}|=2$. Then $g$ can be lifted to an element  $\tilde{g} \in \SL_2(\Z)$  with $\trace{\tilde{g}}=2$. The characteristic polynomial of $\tilde{g}$ is $(T-1)^2$ so  1 is the
		only eigenvalue of $\tilde{g}$. Let $ z \in \Z^2$ be a primitive eigenvector of $\tilde{g}$. Let $h \in \SL_2(\Z)$ be an element whose left column is $z$.  There exists $0 \ne r\in \Z$ such that  $h^{-1}\tilde{g}h=
		\begin{psmatrix}
		1 & r\\
		0 & 1		
		\end{psmatrix}
		=(-UV)^r$.   \end{proof}

\begin{lemma}  
\label{lemma:conjugacy_class_SL}
Let $C$ be a $\PSL_2(\Z)$-conjugacy class of trace $|t| \ge 3$. Then there exists 
$g \in C$ and $s_1,\ldots,s_k \in \{1,2\}$ such that $g=uv^{s_1}uv^{s_2}\ldots u v^{s_k}$.
\end{lemma}
\begin{proof}  The elements of $C$ have infinite order so $u,v,v^2 \notin C$. Take $g \in C$ to be any element such that $w_{u,v}(g)$ is cyclically reduced and starts with $u$. \end{proof}

\begin{lemma}\label{lemma:trace_bound} Under Notation \ref{nota:generators},
assume that $\{s_1,\ldots,s_k \}= \{1,2\}$. Then $|\trace(uv^{s_1}\ldots uv^{s_k})| \ge k$.	
\end{lemma}
\begin{proof}
It is enough to prove the claim when $s_1=1$. 
 Let $U,V$ be as in Notation \ref{nota:generators}.
	Denote $$R:=-UV=\left[\begin{array}{cc}
	1 & 1\\
	0 & 1
\end{array}
\right] S:=-UV^2=\left[\begin{array}{cc}
	1 & 0\\
	1 & 1
\end{array}\right].$$
Note that $$
R^{m_1}S^{m_2}=\left[\begin{array}{cc}
	1+m_1m_2 & m_1\\
	m_2 & 1
\end{array}
\right]
$$
so $\trace (R^{m_1}S^{m_2})=2+m_1m_2 \ge m_1+m_2$ and  all the entries of $R^{m_1}S^{m_2}$ are positive. 
A simple induction argument shows  that for every $k \ge 2$ and  every sequence $m_1,\ldots,m_k$ of positive integers all the entries of the product  $R^{m_1}S^{m_2}R^{m_3}S^{m_4}\cdots$ are positive and
$$\trace(R^{m_1}S^{m_2}R^{m_3}S^{m_4}\cdots)\ge m_1+\cdots+m_k.
$$
\end{proof}

\begin{remark} Under the assumption of Lemma \ref{lemma:trace_bound}
	it is possible to give a better bound to $|\trace(uv^{s_1}\ldots uv^{s_k})|$. Since the exact bound is not important to us, no attempt was made to make it optimal. 
\end{remark}


 The next step is to find a description of the conjugacy classes in $G_{m,n}$ of a given trace $t \ge 3$. 
 
 \begin{notation}\label{nota:ddd}
 	Under Notation \ref{nota:generators}, let $(m,n)\in \{(2,\infty),(3,3),(3,\infty)\}$. Then $\langle a\rangle * \langle b\rangle =G_{m,n}\le \PSL_2(\Z)=\langle u\rangle * \langle v\rangle $.  We can regard an element $g\in G_{m,n}$ as an element in the free product $\langle a\rangle * \langle b\rangle$ and as an element in the free product $\langle u\rangle * \langle u\rangle$. Thus, $g$ can be represented by a reduced $L_{a,b}$-word $w_{a,b}(g) \in W_{a,b}$ and also by a  reduced $L_{u,v}$-word $w_{u,v}(g) \in W_{u,v}$. 
 \end{notation}

\begin{lemma}\label{lemma:sub}  Under Notations \ref{nota:generators} and \ref{nota:ddd}, let $(m,n)\in \{(2,\infty),(3,3),(3,\infty)\}$. For every $g \in G_{m,n}$:
\begin{enumerate}[label=\upshape(\arabic*)]
	\item\label{item:sub1} $l(w_{u,v}(g)) \ge l(w_{a,b}(g))$.
	\item\label{item:sub2} If  $w_{a,b}(g)$ starts with $a^{\pm 1}$ and ends with $b^{\pm 1}$ then $w_{u,v}(g)$ is cyclically reduced.
	\item\label{item:sub3}  Let $C$ be a conjugacy class of $G_{m,n}$ which does not contain a power of $a$ nor a power of $b$. Then there exists $g \in C$ such that $w_{u,v}(g)$ is cyclic reduced and starts with $u$. 
\end{enumerate} 
\end{lemma}
\begin{proof}
	 If $o(a)<\infty$ denote $I_a:=\{0,\ldots,o(a)-1\}$, otherwise, denote $I_a:=\Z$. Define $I_b$ similarly.  The claims  easily follow from the following properties:
			\begin{enumerate}[label=\upshape(\alph*)]
				\item For every $r \in I_a$ and $s \in I_b$, $l(w_{u,v}(a^r))\ge |r|$ and $l(w_{u,v}(b^s))\ge |s|$.
				\item If $w_{v,u}(a)$ starts with $u$  then $w_{v,u}(b)$ starts with $v$.
				\item If $w_{v,u}(a)$ starts with $v$  then $w_{v,u}(b)$ starts with $u$.
				\item  For every $0 \ne r \in I_a$, the first and last letters of $w_{u,v}(a)$ and $w_{u,v}(a^r)$ are all equal. 
				\item  For every $0 \ne s \in I_b$, the first and last letters of $w_{u,v}(b)$ and $w_{u,v}(b^s)$ are all equal.				\item For every $r \in I_a$ and $s \in I_b$, the words $w_{u,v}(a^r)w_{u,v}(b^s)$ and $w_{u,v}(b^s)w_{u,v}(a^r)$ are reduced.
			\end{enumerate}
\end{proof}

The  proof of Lemma \ref{lemma:sub} also implies:
\begin{lemma}\label{lemma:is_free_product}
	Under Notation \ref{nota:generators}, $G_{m,n}$ is a free product of $\langle a\rangle $ and $\langle b\rangle $.
\end{lemma}

\begin{alg}\label{alg1} Fix $(m,n) \in \{(2,3), (2,\infty),(3,3),(3,\infty)\}$ and $3 \le t \in \N$. Use the following steps to find a finite subset of $G_{m,n}$ with a non-empty intersection with every $G_{m,n}$-conjugacy class of elements with trace $t$.
\begin{enumerate}
	\item Apply Lemmas \ref{lemma:conjugacy_class_SL_12}, \ref{lemma:conjugacy_class_SL} and \ref{lemma:trace_bound} to find a finite set $\{g_1,\ldots,g_k\} \subseteq \PSL_2(\Z)$ such that:
	\begin{enumerate}
	\item $\{g_1,\ldots,g_k\}$ contains a representative of every $\PSL_2(\Z)$-conjugacy class of trace $|t|$.
	\item $w_{u,v}(g_1),\ldots,w_{u,v}(g_k)$ are cyclically reduced.
	\end{enumerate}
	\item For every $1 \le i \le k$, let $r_i:=l(w_{u,v}(g_i))$ and let $g_{i,1},\ldots,g_{i,r_i}$ be such that  $w_{u,v}(g_{i,1}),\ldots,w_{u,v}(g_{i,r_i})$ are the cyclic permutations of $w_{u,v}(g_i)$. Item \ref{item:sub3} of Lemma \ref{lemma:sub} implies that every $G_{m,n}$-conjugacy class of elements of trace $t$ contains an element in $\{g_{i,j} \mid 1 \le i \le k, 1 \le j \le r_i\}$. 
	\item  Lemma \ref{lemma:sub} implies that for every $g \in G_{m,n}$, $l(w_{a,b}(g)) \le l(w_{u,v}(g))$. For every $1  \le i \le k$, compute $X_i:=\{g \in G_{m,n} \mid l(w_{a,b}(g))\le r_i\}$ and $Y_i:=\{g_{i,j} \mid 1 \le j \le r_i \text{ and } g_{i,j}\in X_i\}$.    Lemma \ref{lemma:sub} implies that $\cup_{1 \le i \le k}Y_i$ contains at least one element of every conjugacy class of $G_{m,n}$  of element of trace $t$.  \end{enumerate}
\end{alg}

\begin{lemma}\label{lemma:almost_done} Under Notation \ref{nota:generators}, in the following table, for the given values of   $(m,n)$ and  $t$, the set $R_t$  non-trivially intersects every $G_{m,n}$-conjugacy class consisting of elements of trace $t$:

\begin{center}
\begin{tabular}{ | c | c| c | c |} 
  \hline
  $(n,m)$ & $t$ & $R_t$& $R_t\cap G_{m,n}'$  \\ 
  \hline
  &&&\\
  $(2,3)$  & $3$ & $\big\{[a,b]\big\}$ & $\big\{[a,b]\big\}$ \\ 
	&&&\\
  \hline
  &&&\\
   $(2,\infty)$  & $6$ & $\big\{[a,b],ab^3,ab^{-3},abab^2,ab^{-1}ab^{-2}\big\}$ & $\big\{[a,b]\big\}$  \\ 
	&&&\\
  \hline
    &&&\\
   $(3,3)$  & $6$ & $\big\{[a,b],[a,b^2]\big\}$ & $\big\{[a,b],[a,b^2]\big\}$  \\ 
	&&&\\
  \hline
      &&&\\
   $(3,\infty)$  & $14$ & $\big\{b^2a^2,ab^{-2}\big\}$& $\emptyset$  \\ 
	&&&\\
  \hline
        &&&\\
   $(3,\infty)$  & $18$ & $\big\{[a,b],[a,b^{-1}]\big\}$& $\big\{[a,b],[a,b^{-1}]\big\}$  \\ 
	&&&\\
  \hline
\end{tabular}
\end{center}

\end{lemma}

\begin{proof}
	The set $R_t$ was found by using a computer program which implemented algorithm \ref{alg1}. \end{proof}

\begin{lemma}\label{lemma:not3} Under Notations \ref{nota:generators}. Let $d \in G_{m,n}$ be the commutator of topological generators of $\widehat{G_{m,n}}$. If $(m,n)\in \{(2,\infty),(3,3)\}$ then $|\trace d| \ne 2$. 
\end{lemma}
\begin{proof}
	Assume first that $(m,n)=(2,\infty)$. Lemma \ref{lemma:d_in_G'} implies that $d \in G_{2,\infty}'$.
	Lemmas \ref{lemma:conjugacy_class_SL_12} and \ref{lemma:sub} imply that every  $g\in G_{2,\infty}$ with trace $2$ is conjugate in $G_{2,\infty}$ to $b^r$ or $b^{-r}$ or $(ab)^r$ or $(ab^{-1})^r$ for some $r \ge 1$. In particular, such $g$ does not belong to $G_{2,\infty}'$. 
	
	Assume that  $(m,n)=(3,3)$.  Lemma \ref{lemma:d_in_G'} implies that $d \in G_{3,3}'$. Lemmas \ref{lemma:conjugacy_class_SL_12} and \ref{lemma:sub} imply  that every  $g\in G_{3,3}$ with trace $2$ is conjugate in $G_{3,3}$ to  $(ba)^r$ or $(b^2a^2)^r$ for some $r \ge 1$. If in addition, $g \in G_{3,3}'$ then $r$ is divisible by 3.   Corollary \ref{cor:(ab)^3} implies that such an element $g$ is not a commutator of topological generators of $\widehat{G}_{m,n}$.
 \end{proof}

\ 

\begin{proof}[Proof of Proposition \ref{prop:commutator} ] \ 

Lemma implies that $d \in G_{m,n}'$ and Lemma \ref{lemma:few_traces} gives the possible values for $|\trace d|$. Since $v,v^2 \not \in G_{2,3}'$, it follows from Lemma \ref{lemma:conjugacy_class_SL_12} that if  $(m,n)=(2,3)$ then $|\trace d| \ne 1$. It follows from Lemma \ref{lemma:not3} that if $(m,n)\in \{(2,\infty),(3,3)\}$ then $|\trace d| \ne 2$. In all the other cases, it follows from Lemma \ref{lemma:almost_done} that $d$ is a commutator. 
 
 \end{proof}

\vspace{30pt}

\newpage
\section{Lifting solutions}\label{liftings}\label{sec3}

\subsection{\ General $D$\nopunct}\label{general}\ 

The three equations that are central to our study are:

\
\begin{enumerate}[label=($\mathcal{E}$\textsubscript{\arabic*})\quad]\label{E1toE3}
\item $W(X,Y)=XYX^{-1}Y^{-1}=Z$, with $Z\in \Gamma_D=\SL_2(D)$ given and to be solved in $X$, $Y \in \Gamma_D$. Denote the set of solutions by $\mathcal{C}_Z(D)$.
\vspace{4pt}
\item $\Tr\left( W(X,Y)\right)=t$, where $t\in D$ is given, to be solved for $X,\ Y\in \Gamma_D$. Denote by $\mathcal{T}_t(D)$ the set of solutions $(X, Y)$.
\vspace{4pt}
\item $M(x_1,x_2,x_3) := x_1^2+x_2^2+x_3^2-x_1x_2x_3= t+2$ for $t\in D$ given, and to be solved for ${\bf x}=(x_1, x_2, x_3) \in D^3$. Denote the set of solutions by $\mathcal{M}_{t+2}(D)$.
\end{enumerate}

\

There are a number of infinite symmetries that the solutions to these equations satisfy and which facilitate their analysis. The first is the action by simultaneous conjugation by $\Gamma_D$ on the solutions $(X, Y)$ to $(\mathcal{E}_1)$ and $(\mathcal{E}_2)$.

If $Z' = gZg^{-1}$ for some $Z \in \G_D$, then
\begin{equation}\label{3.1}
\mathcal{C}_{Z'}(D) = g\mathcal{C}_Z(D)g^{-1},
\end{equation}
where
\begin{equation}\label{3.2}
g[(X,Y)]g^{-1} = (gXg^{-1}, gYg^{-1})\,.
\end{equation}

The relation between the solutions to $(\mathcal{E}_1)$ for conjugate $Z$'s suggests that we modify the solution set $\mathcal{C}_{Z}(D)$ of $(\mathcal{E}_1)$ to

\ 

\begin{enumerate}[label=($\hat{\mathcal{E}}$\textsubscript{\arabic*})\quad]
\item ${\rm COM}_{Z}(D) := \left\{(X,Y) \in \G_D\times\G_D \ : \ W(X,Y)\ \text{is\ a\ conjugate\ of} \ Z\right\}\ .$
\end{enumerate}

\ 

In this way $\text{COM}_Z(D)$ depends only on the conjugacy class of $Z$, $\G_D$ acts on $\text{COM}_Z(D)$, and we naturally divide by this action to get equivalence classes $\overline{(X,Y)}$ of solutions to $(\hat{\mathcal{E}}_1)$ and whose totality is denoted by
\begin{equation}\label{3.4}
\text{COM}_Z(D)/\G_D \ .
\end{equation}

If $\Tr(Z) = t$ then $\text{COM}_Z(D) \subset \mathcal{T}_t(D))$, and the action of $\G_D$ on pairs extends to $\mathcal{T}_t(D)$. Dividing the latter by $\G_D$ gives classes of solutions, the totality of which is denoted by
\begin{equation}\label{3.5}
\mathcal{T}_t(D)/\G_D \ .
\end{equation}
So for $\Tr(Z) = t$, we have $ \text{COM}_Z(D)/\G_D  \subset \mathcal{T}_t(D)/\G_D$. 

In what follows we make use of some well known properties of the $S$-arithmetic groups $\G_D$. These are the extensions of the reduction theory of Hermite and Minkowski as developed by Borel \cite{Borel} and algorithmically by Grunewald-Segal in \cite{GuSe1} and \cite{GuSe2}.

For a given $t \in D$ there are finitely many conjugacy classes in $\G_D$ represented by $Z_1, ..., Z_h$ say, of trace equal to $t$. It follows that we have the decomposition 
\begin{equation}\label{3.6}
\mathcal{T}_t(D) = \bigcup\limits_{j=1}^{h}\ \text{COM}_{Z_j}(D)\,.
\end{equation}
The decomposition respects the $\G_D$-action on these sets so that
\begin{equation}\label{3.7}
\mathcal{T}_t(D)/\G_D = \bigcup\limits_{j=1}^{h}\ \left(\text{COM}_{Z_j}(D)/\G_D \right).
\end{equation}

There are further symmetries that $\text{COM}_Z(D)$, $\mathcal{T}_t(D)$ and $\mathcal{M}_{t+2}(D)$ carry. These come from the action of outer automorphisms of the free group, that is the Nielsen automorphisms on pairs $(X, Y)$ and which descend to the Markoff group action on $\mathcal{M}_{t+2}(D)$. The Nielsen and Markoff groups are isomorphic (to $\PGL_2(\Z)$ in fact) and we denote them here by $G$ with the action on $(X, Y)$ or $(x_1, x_2, x_3)$ being implicit below. These are given explicitly in terms of generators of these groups in the tables in Lemma \ref{perms}.

Denote by $\pi$ the map from $\mathcal{T}_t(D)$  to $\mathcal{M}_{t+2}(D)$ given by
\begin{equation}\label{3.8}
\pi(\, (X,Y)\, ) = (\Tr(X), \Tr(Y), \Tr(XY))\,.
\end{equation}
$\pi$ preserves the $\G_D$-classes, and induces a map
\begin{equation}\label{3.9}
\pi\ :\ \mathcal{T}_t(D)/\G_D \rightarrow \mathcal{M}_{t+2}(D)\,.
\end{equation}
The $G$-actions on $\mathcal{T}_t(D)$  and $\mathcal{M}_{t+2}(D)$ are $\pi$ equivariant, that is for $g \in G$
\begin{equation}\label{3.10}
\pi(\, g(X,Y)\, ) = g(\pi(\, (X,Y)\, ))\,.
\end{equation}
Dividing $\mathcal{T}_t(D)/\G_D$ and $\mathcal{M}_{t+2}(D)$ by this $G$-action gives a further collapsing into $G$-equivalence classes $\mathcal{T}_t(D)/G$ and $\mathcal{M}_{t+2}(D)/G$. Note that $G$ preserves the decompositions \eqref{3.6} and \eqref{3.7} so that we retain the decomposition
\begin{equation}\label{3.11}
\mathcal{T}_t(D)/G = \bigcup\limits_{j=1}^{h}\ (\text{COM}_{Z_j}(D)/G) ,
\end{equation}
and
\begin{equation}\label{3.12}
\pi\ :\ \mathcal{T}_t(D)/G \rightarrow \mathcal{M}_{t+2}(D)/G.
\end{equation}

With the set up above we can analyse the relations between the solubility of $(\hat{\mathcal{E}}_1)$, $(\mathcal{E}_2)$  and $(\mathcal{E}_3)$ and decision procedures for them. From the decompositions \eqref{3.6}, \eqref{3.7} and \eqref{3.11}, and that for a given $t$, the $h$ classes $Z_1, ..., Z_h$ can be effectively computed (\cite{GuSe2}), it follows that a decision procedure for $(\mathcal{E}_1)$ gives one for $(\mathcal{E}_2)$. In the other direction, we can certainly decide if a given solution in $\mathcal{T}_t(D)$ lifts to a solution of $(\mathcal{E}_1)$ with a specific $Z$. This is asking whether in \eqref{3.6} the given solution $(X,Y)$ is in $\text{COM}_Z(D)$, which is the same as whether $W(X,Y)$ is a conjugate of $Z$. Again, by \cite{GuSe2} this can be determined effectively. This however does not allow us to decide $(\mathcal{E}_1)$ from simply having a solution to $(\mathcal{E}_2)$. If the solution at hand lifts to the desired $Z$ one succeeds, but in principle one has to examine infinitely many solutions in  $\mathcal{T}_t(D)/G$. If however the latter is finite as is the case for $D=\Z$ (see below), then one can go from a decision procedure for $\mathcal{T}_t(D)$ to $(\mathcal{E}_1)$. In any case, if $\mathcal{T}_t(D)$ is empty then so is $\text{COM}_Z(D)$ for any $Z$ with $\Tr(Z) = t$.

\ 

We turn to the question of lifting solutions from $\mathcal{M}_{t+2}(D)$ to $\mathcal{T}_t(D)$ in \eqref{3.9} and \eqref{3.12}.

In the forward direction, if $\mathcal{M}_{t+2}(D)$ is empty then so is $\mathcal{T}_t(D)$, which is one of the primary ways that we exploit equations $(\mathcal{E}_1)$, $(\mathcal{E}_2)$ and $(\mathcal{E}_3)$. In the other direction starting with $\mathcal{M}_{t+2}(D)$, the key feature is that the map $\pi$ in \eqref{3.9} is finite to one and effectively so. That is given ${\bf x}\in \mathcal{M}_{t+2}(D)$, $\pi^{-1}({\bf x}) \in \mathcal{T}_t(D)/\G_D$ is finite and there is an effective procedure to find each point in $\pi^{-1}({\bf x})$, including the case where this set is empty. To see this let $V$ be the 8-dimensional vector space over $F$ (the number field from which our $D$ is taken) of pairs $(A, B)$ of $2\times2$ matrices. Then $\SL_2(F)$ acts linearly via the diagonal conjugation representation $\rho(g):\ (A,B) \mapsto (gAg^{-1}, gBg^{-1})$. We are primarily interested in Zariski closed orbits of this action. $\rho$ preserves $\det(A)$, $\det(B)$, $\Tr(A)$, $\Tr(B)$ and $\Tr(AB)$, and these generate the ring of $\rho$-invariants.

Fixing the values $\det(A)=1$, $\det(B)=1$, $\Tr(A)=x_1$, $\Tr(B)=x_2$ and $\Tr(AB)=x_3$, with $x_1, x_2, x_3 \in D$, defines a homogeneous affine variety $X\subset V$ corresponding to an orbit of a $F$ point (see Prop. \ref{lift1}). $X$ is defined over $F$ and we seek the points in $X(D)$. By \cite{Borel}, $X(D)$ consists of finitely many $\rho(SL_2(D))$ orbits, and by \cite{GuSe2} these can be determined effectively. Thus, for any ${\bf x}\in \mathcal{M}_{t+2}(D)$ the lifts of ${\bf x}$ to  $\mathcal{T}_t(D)/\G_D$, that is $\pi^{-1}({\bf x})$ are finite in number and effectively so.

In the case that $\mathcal{M}_{t+2}(D)/G$ is finite (and effectively so), the above shows that $\mathcal{T}_t(D)/G$ is also finite and so is $\text{COM}_Z(D)/G$. So in this case $(\mathcal{E}_1)$, $(\mathcal{E}_2)$ and $(\mathcal{E}_3)$ are all effectively decidable. When $D$ has infinitely many units, $\mathcal{M}_{k}(D)/G$ is not necessarily finite and our analysis here falls short of giving a decision procedure for all three.

\subsubsection{\ \  An explicit analysis \nopunct }\label{genlift}\

For matrices $X$ and $Y$ in $\G_D=\SL_2(D)$, for any ring  $D$ (or field $F$), we consider the tuple $(X,Y)$ with commutator $W(X,Y)=Z$. Putting  $t=\Tr Z$, $x_1=\Tr X$, $x_2=\Tr Y$ and $x_3= \Tr XY$, for ${\bf x} = (x_1, x_2, x_3)$ on the Markoff surface $\mathcal{M}_{t+2}(D)$ in $(\mathcal{E}_3)$, we have double sign-changes, permutations and the (non-linear) Vieta maps as invariant actions. The double sign-changes are easily lifted from $(\mathcal{E}_3)$ to $(\mathcal{E}_1)$, corresponding to the three maps $(X,Y)\mapsto (-X,Y)$, $(X,Y)\mapsto (X,-Y)$ and $(X,Y)\mapsto (-X,-Y)$, preserving $W(*,*)=Z$. The six permutations and three Vieta maps of the $x_i$ are less trivial, but lift if we identify commutators with their inverses, as follows:

\begin{lemma}\label{perms} The first table gives the permutations of the subscripts of $(x_1,x_2,x_3)$ in $(\mathcal{E}_3)$, the corresponding action on $(\mathcal{E}_1)$, and the commutator:

 \begin{table}[h!]
\centering
\aboverulesep=0ex 
   \belowrulesep=0ex 
\begin{tabular}{|c|c|c|}

\midrule
\rule{0pt}{1.1EM}%
$(1,2,3)$  & $(X,Y)$ &$Z$\\
\midrule
\rule{0pt}{1.1EM}%
$(1,3,2)$& $(YXY^{-1},X^{-1}Y^{-1})$&$Z^{-1}$\\
\midrule
\rule{0pt}{1.1EM}%
$(2,1,3)$&$(Y,X)$&$Z^{-1}$ \\
\midrule
\rule{0pt}{1.1EM}%
$(2,3,1)$&$(XYX^{-1},Y^{-1}X^{-1})$&$Z$ \\
\midrule
\rule{0pt}{1.1EM}%
$(3,1,2)$&$(XY,X^{-1})$&$Z$ \\
\midrule
\rule{0pt}{1.1EM}%
$(3,2,1)$&$(YX,Y^{-1})$&$Z^{-1}$ \\
\midrule

\end{tabular}
\label{alpha}
\end{table}

This second table gives the lifting action of the three Vieta moves from $(\mathcal{E}_3)$ to $(\mathcal{E}_1)$ giving the Nielsen moves, with the resulting commutator indicated:

     \begin{table}[h!]
\centering
\aboverulesep=0ex 
   \belowrulesep=0ex 
\begin{tabular}{|c|c|c|}

\midrule
\rule{0pt}{1.1EM}%
$(x_1,x_2,x_3)\mapsto (x_1,x_2,x_1x_2-x_3)$  & $(X,Y)\mapsto (X^{-1},\ XYX^{-1})$ & $Z^{-1}$\\

\midrule
\rule{0pt}{1.1EM}%
$(x_1,x_2,x_3)\mapsto (x_2x_3-x_1,x_2,x_3)$  & $(X,Y)\mapsto (YXY,\ Y^{-1})$ & $Z^{-1}$\\

\midrule
\rule{0pt}{1.1EM}%
$(x_1,x_2,x_3)\mapsto (x_1,x_1x_3-x_2,x_3)$  & $(X,Y)\mapsto (X^{-1},\ XYX)$ & $Z^{-1}$\\
\hline
\end{tabular}
\label{alpha}
\end{table}
\end{lemma}

\begin{rem}

One notes that for $D=F$ a field, if $Z\in \Gamma_F$ with $\Tr Z=t$, then $Z$ is not conjugate to $Z^{-1}$ in $\Gamma_F$ if and only if $Z$ is not a diagonal matrix and if the equation $x^2-(t^2-4)y^2=-1$ has no solutions $(x,y)\in F^2$.

For $F=\RR$, $Z \not\sim Z^{-1}$ if and only if $Z$ is elliptic.

For $F=\mathbb{F}_{\!p}$ a prime field with $p>2$, using Lemma \ref{lemB1}, we see that $Z$ is not conjugate to $Z^{-1}$ in $\Gamma_F$ if and only if $Z$ is not a diagonal matrix, $p\equiv 3 \ (\text{mod 4})$ and $t\equiv \pm 2 \pmod p)$.
\end{rem}

We now state a (weak) lifting result from $(\mathcal{E}_3)$ to $(\mathcal{E}_1)$ over fields:

\begin{proposition}\label{lift1}\

 Let $F$ be a field with $\text{char}(F)=0$ , $\Gamma_F=\SL_2(F)$, $Z\in \Gamma_F$ with $\Tr Z=t \neq \pm 2$, and suppose  $(x_1,x_2,x_3)\in \mathcal{M}_{t+2}(F)$ is given. Suppose there is a matrix $Y\in \Gamma_F$ such that $\Tr ZY=\Tr Y=x_j$ for some $1\leq j\leq 3$. Then, either $Z$ or $Z^{-1}$ is $\Gamma_F$-conjugate to $W(X,Y)$ for some $X\in \Gamma_F$.
\end{proposition}

\begin{rem} 
The proposition also holds for fields with finite but sufficiently large characteristic.

The condition $Y\in \Gamma_F$ such that $\Tr ZY=\Tr Y=x_j$ for some $1\leq j\leq 3$ is obviously necessary for $(x_1,x_2,x_3)\in \mathcal{M}_{t+2}(F)$ to be a candidate to be lifted from $(\mathcal{E}_3)$ to $(\mathcal{E}_1)$.
\end{rem}
The proof will follow from the following

\begin{lemma}\label{lift2}With the notation as in the proposition above, suppose $\Delta=t+2 -x^2_2 \neq 0$. Then, given any $Y\in \Gamma_F$ with $\Tr ZY=\Tr Y=x_2$, there exists a unique $X\in \Gamma_F$  such that (i). $\Tr X=x_1$, (ii). $\Tr XY=x_3$, and (iii). $ W(X,Y)=Z$. 

The conclusion also holds if $F$ is replaced with  $D$  a ring but with the requirement $\Delta\in D^{\times}$.
\end{lemma}
\begin{proof} The proof is an application of Prop. \ref{identity}. 

Uniqueness follows readily from \eqref{Ap1.1}:  assume there are two matrices $X_1$ and $X_2$ satisfying all the conditions satisfied by $X$, so that on subtracting the two expressions obtained from \eqref{Ap1.1}, we have $(X_1-X_2)[x_3Y+(x_1-x_2x_3)I]=0$, from which we have $X_1=X_2$, since $\det[x_3Y+(x_1-x_2x_3)I]=\Delta$ is invertible.

For the existence, choose $X$ so that
\begin{equation}\label{matrix-X}
\begin{split}
\Delta X&=-x_3ZY +x_1Z +(x_3-x_1x_2)Y^{-1}+x_1I\ ,\\
&= \left(Z-Y^{-2}\right)\left(x_1I-x_3Y\right).
\end{split}
\end{equation}

To show $X\in \Gamma_F$, we note that $X(x_1Y-x_3I)= (ZY-Y^{-1})$. We have $\det(x_1X-x_3I)= x^2_1-x_1x_2x_3 + x_3^2=t+2-x_2^2=\Delta$. Similarly, $\det(ZY-Y^{-1})= 1-\Tr ZY^2 +1= 2-\Tr Z(x_2Y-I)=2 - x_2^2+t=\Delta$. Hence $\det X=1$.

For (i), since $\Tr ZY=x_2$, taking the trace in \eqref{matrix-X} gives 
\[\Delta \Tr X= \left[-x_3x_2 +x_1t+(x_3-x_1x_2)x_2 +2x_1\right]=\Delta x_1 .\]

For (ii), we have $\Delta XY=-x_3Z(x_2Y-I) +x_1ZY +(x_3-x_1x_2)I+x_1Y$, using $Y^2=x_2Y-I$. Expanding and taking the trace gives the result.

Finally, for (iii), we solve for $Z$ in $X(x_1Y-x_3I)= (ZY-Y^{-1})$, and using $Y^{-2}=x_2Y^{-1}-I$ and $Y^{-1}=-Y+x_2I$, together with Prop. \ref{identity} gives the result.

\end{proof}

\begin{proof}[Proof of Prop. \ref{lift1}]\

 Suppose $x_j^2 = t+2$ for at least two $j$'s. There are only a bounded number of such points on $\mathcal{M}_{t+2}(F)$  and by Zariski density we look at other points equivalent to the original under the Markoff group $\mathfrak{M}$ to get to one for which we may assume a coordinate does not satisfy $x_j^2 = t+2$. If the proposition is proved for this latter point, we will have a pair of matrices $X_0$ and $Y_0$ with $W(X_0,Y_0) = Z^{\pm 1}$.  Then, reversing the process and using Lemma \ref{perms}  gives a pair of matrices associated with the original point. Now assume $x_3^2 = t+2$, but not for $x_1$ and $x_2$. Then, the point obtained by a Vieta move $(x_1, x_2, x_1x_2 -x_3)$  is problematic only if $(x_1x_2 -x_3)^2 = t+2$, in which case we conclude that $x_1x_2 = 2x_3$ and also $(x_1^2-2)(x_2^2 -2)=4$, so that there are only a finite number of such exceptions. Then, we repeat the argument given above.

 So assume  $x_j^2 \neq t+2$ for all $j$. If $Y\in \Gamma_F$ can be found satisfying  $\Tr ZY=\Tr Y=x_j$ for some $j$ , we apply a permutation to move that $x_j$ to the middle coordinate, noting by Lemma \ref{perms} that the corresponding $Y$ can be found. Then Lemma \ref{lift2} above can be used to find $X$. The ambiguity of $Z$ or $Z^{-1}$ comes from the application of the permutation map.

\end{proof}

\begin{proposition} Suppose $F$ is a field with $\mathbf{card}(F)>5$. 
If $Z\in \Gamma_F$ with $\Tr Z\neq -2$, then $Z$ is a commutator.
\end{proposition}

\begin{proof} We consider two cases: $\Tr Z=2$ and $\Tr Z=t\neq \pm 2$.

When $\Tr Z=2$, we assume $Z\neq I$, in which case, $Z$ is $\Gamma_F$-conjugate to a matrix of the type $\left[\begin{smallmatrix} 1&b\\0&1 \end{smallmatrix}\right]$ for some $b\in F$; this is easily verified using matrices of the type $\left[\begin{smallmatrix} \ \ 0&1\\-1&x\end{smallmatrix}\right]$ and $\left[\begin{smallmatrix} 1&1\\0&1\end{smallmatrix}\right]$ to conjugate with. To verify that $Z$ is a commutator, it suffices to consider the case $Z=\left[\begin{smallmatrix}  1&b\\0&1\end{smallmatrix}\right]$ with $b\neq 0$. 

Put $X=\left[\begin{smallmatrix}  x_1&x_2\\0&x_1^{-1}\end{smallmatrix}\right]$ and $Y=\left[\begin{smallmatrix}  y_1&y_2\\0&y_1^{-1}\end{smallmatrix}\right]$. Substituting into $Z=W(X,Y)$ requires us to find $x_2$ and $y_2$ such that $(x_1-x_1^{-1})y_2-(y_1-y_1^{-1})x_2=bx_1^{-1}y_1^{-1}$. This is solvable if $x_1,\ y_1 \neq \pm1$. For this part of the proof, we need $\mathbf{card}(F)>3$.

For $\Tr Z=t\neq \pm 2$, we will use Lemma \ref{lift2}. To do so, we first create a suitable point $(x_1,x_2,x_3)\in \mathcal{M}_{t+2}(F)$; and then a $Y\in \Gamma_F$ satisfying all the conditions (these are easy to do on a field). The result then follows.

Let $\zeta \neq 0,\ \pm1$ and put $x_2 =\zeta+\zeta^{-1}$. Put $x_3= \frac{x_2^2 -t-1}{\zeta-\zeta^{-1}}$ and $x_1= 1+\zeta x_3$. Then, $(x_1,x_2,x_3)\in \mathcal{M}_{t+2}(F)$. Choose $\zeta$ such that $x_2^2\neq t+2$ (this requires $\mathbf{card}(F)>5$ if $t+2$ is a square). Now, put $Y=\left[\begin{smallmatrix} \zeta&\eta\\0&\zeta^{-1}\end{smallmatrix}\right]$, with $\eta\in F$ to be chosen. If $Z-I=\left[\begin{smallmatrix} a_1&a_2\\a_3&a_4\end{smallmatrix}\right]$, then $\Tr ZY=\Tr Y$ iff $a_1\zeta +a_3\eta+a_4\zeta^{-1}=0$. This is solvable in $F$ when $a_3\neq 0$. If $a_3=0$ but $a_2\neq0$, we conjugate $Z$ with $\left[\begin{smallmatrix} 0&1\\-1&0\end{smallmatrix}\right]$ to get a new matrix with the corresponding $a_3\neq 0$, giving a suitable $Y$. Finally, if $a_2=a_3=0$, we conjugate $Z$ with $\left[\begin{smallmatrix} 1&1\\0&1\end{smallmatrix}\right]$ to get a matrix with a non-zero off-diagonal entry $a_1-a_4\neq 0$ (since $Z\neq \pm I$), giving $Y$. Hence we conclude using Lemma \ref{lift2} that some conjugate of $Z$ or $Z^{-1}$ is a commutator, giving the result.
\end{proof}

\subsection{\ The case $D=\mathbb{Z}$\nopunct}\label{sec32}\

Since $\Gamma_D:=\G = \PSLZ$ is isomorphic to $\Z/2\Z \ast \Z/3\Z$, the purely group theoretic equation $(\mathcal{E}_1)$ can be solved explicitly by expressing $Z$ (or its conjugacy class in $\G$ denoted by $\{Z\}_\G$) in terms of its generators. Wicks' theorem \cite{wicks} describes explicitly which elements are commutators in such a free product, in terms of their spelling as words in the generators. As shown in \cite{park1}, if one orders the $\{Z\}_\G$'s by their minimal word length, then very few of these are commutators; roughly square-root of the total number. According to Theorem \ref{thm:main}, the failure of the typical $Z$ to be a commutator is witnessed in some finite quotient of $\Gamma$. One can ask about the local to global principle when restricting locally to congruence quotients of $\G$; which we call the ``congruence'' Hasse principle. 

For $D=\Z$ this congruence version for $(\mathcal{E}_1)$ has mostly failures. This version is also relevant for $(\mathcal{E}_2)$ and $(\mathcal{E}_3)$ which are no longer purely group theoretic equations. For each of $(\mathcal{E}_1')$, $(\mathcal{E}_2)$ and $(\mathcal{E}_3)$, the local congruence obstructions are passed for a positive proportion of the choices on the right hand sides, and we call such choices admissible. 

For $\mathcal{M}_k(\Z)$, $k=t+2$ is admissible if and only if
\begin{equation}\label{3.2.1}
t \not\equiv 1 \pmod 4 \quad \text{and} \quad t \not\equiv 1, 4 \pmod 9\ . 
\end{equation}
This and the corresponding Hasse principle for $(\mathcal{E}_3)$ were studied in depth in \cite{GhSa}.

The congruence obstructions for $\mathcal{M}_{t+2}(\Z)$ lead to ones for $\mathcal{T}_t(\Z)$ and there are further ones: $t$ is not admissible if 

\begin{equation*}
t \equiv 0, 1, 4, 5, 8, 9, 10, 12, 13  \pmod {16}
\end{equation*}
and
\begin{equation}\label{3.2.2}
 t \equiv 1, 4, 5, 8 \pmod 9\ . 
\end{equation}

We expect but have not verified that \eqref{3.2.2} gives a complete description of the admissible $t$'s, see Section \ref{HFC} (Prop. \ref{HFu2}).

For $(\mathcal{E}_1')$ an explicit description of the (congruence) admissible conjugacy classes $\{Z\}_\G$ is more complicated, but it consists of a positive proportion of the sets
\begin{equation}\label{3.2.3}
\left\{ \{Z\}_\G\ : |\Tr(Z) \leq T \right\}.
\end{equation}

Indeed if $A$ is the matrix in $\SL_2(\Z/q\Z)$, with $q=2^2.3^3.5$ in Theorem \ref{HFE1}, then according to that theorem if $Z\in \Gamma \pmod q$ is conjugate to $A \pmod q$, then $\{Z\}_\G$ is (congruence) admissible for $(\hat{\mathcal{E}}_1)$. By the Chebotarev Theorem for prime geodesics (\cite{sarnak_thesis}) it follows that

\[\frac{\displaystyle \#\left\{  \{Z\}_\G : |\Tr(Z)|\leq T,\ \{Z\}_{\SL_2(\Z/q\Z)}=\{A\}_{\SL_2(\Z/q\Z)}\right\}}{\displaystyle \#\left\{  \{Z\}_\G : |\Tr(Z)|\leq T\right\}} \rightarrow \frac{|\displaystyle \{A\}_{\SL_2(\Z/q\Z)}|}{|\displaystyle  {\SL_2(\Z/q\Z)}|} \ ,\]
as $T \to \infty$. This yields a positive proportion lower bound for the congruence admissible $Z$'s for $(\mathcal{E}_1')$.  Pushing this analysis a bit further, one can show that the proportion of such admissible conjugacy class tends to a limit.

\

The algorithm above to decide $(\mathcal{E}_1')$ can be used to decide $(\mathcal{E}_2)$. Given an admssible $t$ for $\mathcal{T}_t(\Z)$, one computes the $h(t)$ classes $\{Z\}_\G$ with trace equal to $t$. This is a classical calculation, in fact the number $h(t)$ is the Hurwitz class number of binary quadratic forms of discriminant $t^2-4$ (\cite{chowchow}) and there are approximately $t$ such classes for $t$ large. For each class one runs $(\mathcal{E}_1)$ and checks whether it is a commutator. This gives a list of which of these are commutators, and in particular whether $\mathcal{T}_t(\Z)$ has a solution or not. 

In \cite{park2} this procedure was implemented for $|t|<1000$. He finds that for most admissible $t$'s, $\mathcal{T}_t(\Z)$ is not empty,  and typically only a few of the $h(t)$ classes are commutators. He also found a number of Hasse failures, some examples being:

\ 
\begin{enumerate}[label=(\arabic*)]\label{ex}
\item $t=3$ : the one conjugacy class of trace 3 is not even in the commutator subgroup of $\G$.
\item $t=-21$ : there are two conjugacy classes of this trace which are in the commutator subgroup. However, neither is a commutator.
\item $t=15$ : this is not a Hasse failure but it is typical. There are four conjugacy classes which are in the commutator subgroup; two are not commutators and two of them $Z_1 = \begin{psmatrix} 2&5\\5&13\end{psmatrix}$ and $Z_2 = \begin{psmatrix} \ \ 2&-5\\-5&\ 13\end{psmatrix}$ are. Hence only $\{Z_1\}_\G$ and $\{Z_2\}_\G$ are lifts of solutions from $(\mathcal{E}_2)$ to $(\mathcal{E}_1')$ with $t=15$.
\end{enumerate}
\

The top down procedure from $(\mathcal{E}_1)$ to $(\mathcal{E}_2)$ gives an algorithm to decide both of these equations. However, it offers less in terms of proving theoretical results for $(\mathcal{E}_2)$. The bottom up approach, that starts with $(\mathcal{E}_3)$ and lifts solutions is useful in this regard, and since in this case $\mathcal{M}_{t+2}(\Z)/G$ is finite for generic points, it can also be used to give a decision procedure for all three of $(\mathcal{E}_1)$, $(\mathcal{E}_2)$ and $(\mathcal{E}_3)$.

If $\mathcal{M}_{t+2}(\Z)= \emptyset$ and $t$ is admissible (if it satisfies \eqref{3.2.2}), then $\mathcal{T}_t(\Z)=\emptyset$ and is a Hasse failure. An example of such is $t=-21$ ($\mathcal{M}_{-19}(\Z)$ is one of the Hasse failures from \cite{GhSa}). Moreover, infinitely many such Hasse failures for $\mathcal{T}_t(\Z)$ can be constructed using the theory in \cite{GhSa}, and we give some of these in Section 5 (see Prop. \ref{HFTrace0} and Prop. \ref{HFTrace}). In fact these are promoted to give infinitely many conjugacy classes $\{Z\}_\G$ for which  the congruence Hasse principle fails for $(\mathcal{E}_1')$. 

Our purpose in this section is the decision procedure for the $(\mathcal{E})$'s.  As detailed in Section \ref{general}, once $\mathcal{M}_{t+2}(D)/G$ is finite as it is here, we have only finitely many lifting problems to solving $\mathcal{T}_t(\Z)/G$, each of which is effective, and then a further finite number of effective liftings to $(\mathcal{E_1'})$, with $\Tr(Z)=t$. This gives a complete decision procedure for all the $(\mathcal{E})$'s over $\Z$. As the analysis from \cite{GhSa} yields an effective and feasible algorithm for $(\mathcal{E}_3)$, we apply it to do the same for $(\mathcal{E}_1)$ and $(\mathcal{E}_2)$ in the following subsection.

\subsubsection{\ \ An algorithmn for  $(\mathcal{E}_2(\Z))$  \nopunct}\

If $x_1=\Tr X$, $x_2=\Tr Y$ and $x_3=\Tr XY$, then $(\mathcal{E}_3)$ holds when $(\mathcal{E}_2)$ holds for a given $k$ with $t=k-2$. Given a vector ${\bf x}=(x_1,x_2,x_3)\in \mathcal{M}_k(\mathbb{Z})$, we say it is $(\mathcal{E}_2)$-good if we can find matrices $X$ and $Y$ such that $(\mathcal{E}_2)$ holds in this way,  and otherwise ${\bf x}$ is $(\mathcal{E}_2)$-bad.  We know (see  \cite{GhSa} for references), that when $D=\mathbb{Z}$, there is an algorithm to determine if $(\mathcal{E}_3)$ is solvable for generic $k>4$.

\begin{proposition}\label{m0}
Let $k\in \mathbb{Z}$ be generic such that the odd part of $k-4$ is squarefree. Let ${\bf x}_j$ be fundamental solutions to $(\mathcal{E}_3)$ with $1\leq j\leq \mathfrak{h}(k)$, where $  \mathfrak{h}(k)$ is the ``class number'' of $(\mathcal{E}_3)$. Denote the coordinates of ${\bf x}_j$ by $x_{ji}$ with $1\leq i\leq 3$. Suppose for each $j$ there is an odd prime factor $p$ of $k-4$ and a coordinate $x_{ji}$, such that $x_{ji}^2 -4$ is a quadratic non-residue modulo $p$. Then no solution $\bf{x}$ of $(\mathcal{E}_3)$ is $(\mathcal{E}_2)$-good.
\end{proposition}

\begin{example}  $k=108$ is generic with $k-4= 8*13$. One computes  $  \mathfrak{h}(108)=1$ with fundamental solution $(-3,3,6)$, using Theorem 1.1 of \cite{GhSa}. Since $3^2-4=5$ is a quadratic nonresidue modulo 13, the proposition states that $(\mathcal{E}_3)$ is solvable with $k=108$ but $(\mathcal{E}_2)$ is not solvable with $t=106$. Note that $t=106$ has no congruence obstructions in $(\mathcal{E}_2)$ (see Section 5.1) and so is a Hasse failure for $(\mathcal{E}_2)$, but not for $(\mathcal{E}_3)$.
\end{example}

The proof follows from the following lemmas.

\begin{lemma}\label{m1}
Let $(x_1,x_2,x_3)\in \mathcal{M}_k(\mathbb{Z})$ and let $p|(k-4)$ be odd. Then
\begin{enumerate}[label=\upshape(\arabic*).]
\item If $p\nmid x_i^2 -4$ for all $i$ , then $x_i^2-4$ are either all quadratic residues (QR) or all non-residues (QNR) modulo $p$.
\item If $p| x_i^2 -4$ for some $i$, then it does not do so for the other two. Moreover $x_j^2-4$ are then either both QR or both QNR modulo $p$ for $j\neq i$.
\end{enumerate}
\end{lemma}
\begin{proof}
Let $i,j,l$ be any permutation of $1,2,3$. Completing the square in $(\mathcal{E}_3)$ gives 
\[ (2x_i -x_jx_l)^2 -(x_j^2 -4)(x_l^2-4)= 4(k-4) \equiv 0 \pmod p\ .
\]
Since the products $(x_j^2 -4)(x_l^2-4)$ are all squares mod $p$ for each pair of subscripts, the first case of the lemma follows.

Next, it also follows that $p$ cannot divide two of $x_i^2-4$, as then $p^2|k-4$. The second case now follows.
\end{proof}

\begin{definition}
We say ${\bf x}=(x_1,x_2,x_3)$ is a QR mod $p$ if at least two of $x_i^2 -4$ are QR mod $p$. Otherwise we say it is a QNR
\end{definition}

\begin{lemma}\label{m2} Let $  \mathfrak{M}$ be the Markoff group and $  \mathfrak{m}\in  \mathfrak{M}$. If $\bf{x}$ satisfies $(\mathcal{E}_3)$, then $\bf{x}$ is a QR mod $p$ if and only if $  \mathfrak{m}{\bf x}$ is a QR mod $p$.
\end{lemma}
\begin{proof}
It suffices to verify this for a single Vieta move, say $(x_1,x_2,x_3)\rightarrow (x_4,x_2,x_3)$ with $x_4=x_2x_3-x_1$. From the assumption, it follows that either $x_2^2-4$ or $x_3^2-4$ is a QR. If $p$ does not divide either, it follows from Lemma \ref{m1} that both are QR and we are done. So assume $p|x_2^2-4$, in which case the assumption implies $x_3^2-4$ is a QR, so that again Lemma \ref{m1} implies $x_4^2-4$ is too.
\end{proof}

\begin{lemma}\label{m3}
Suppose ${\bf x}$ satisfying $(\mathcal{E}_3)$ is $(\mathcal{E}_2)$-good. Then so is $  \mathfrak{m}{\bf x}$ for all $  \mathfrak{m}\in   \mathfrak{M}$.
\end{lemma}
\begin{proof}
It suffices to check this for a Vieta move, but the trace identity $\Tr XY^{-1}=(\Tr X)(\Tr Y)-\Tr XY$ verifies this.
\end{proof}
Let $\mathcal{M}_k(\mathbb{Z})$ be decomposed into its $  \mathfrak{h}(k)$ connected components $V^{(j)}_k(\mathbb{Z})$. Lemma \ref{m3} implies that all elements of $V^{(j)}_k(\mathbb{Z})$ are either $(\mathcal{E}_2)$-good or all $(\mathcal{E}_2)$-bad. So to prove Prop. \ref{m0} it suffices to verify that all elements of a fundamental set are $(\mathcal{E}_2)$-bad.
\

\begin{proof}[Proof of Prop.\,\ref{m0}]\ 

Suppose ${\bf x}=(x_1,x_2,x_3)\in \mathcal{M}_k(\mathbb{Z})$ is $(\mathcal{E}_2)$-good. Write $x_1=\Tr X$, $x_2=\Tr Y$, $x_3=\Tr XY$ with $X$, $Y$ in $\SLZ$.
Write 
\[
X=\left[\begin{matrix} a_1&a_2\\a_3&x_1-a_1\end{matrix}\right],\quad \quad Y=\left[\begin{matrix} b_1&b_2\\b_3&x_2-b_1\end{matrix}\right]
\]
so that $x_3= a_1b_1 + a_2b_3 + a_3b_2 + (x_1-a-1)(x_2-b_1)$. Put $u=2a_1 -x_1$ and $v= x_2-2b_1$. We get three equations
\begin{enumerate}[label=(\roman*).]
\item $u^2 +4a_2a_3=x_1^2-4$,
\item $v^2 +4b_2b_3= x_2^2-4$,
\item $uv= (x_1x_2-2x_3)+2(a_2b_3+a_3b_2)$.
\end{enumerate}
These imply
\[
(x_2^2-4)(a_2v+b_2u)^2= [a_2(x_2^2-4)+(x_1x_2-2x_3)b_2]^2- 4b_2^2(k-4).
\]
Let $p|k-4$ be an odd prime factor. If $p\nmid (a_2v+b_2u)$, it follows that $x_2^2-4$ is zero or a QR mod $p$. Next, if $p|(a_2v+b_2u)$, since $p^2\nmid(k-4)$, it follows that $p|b_2$ so that (ii) implies again that $x_2^2-4$ is zero or a QR mod $p$. By symmetry, it follows that the same possibilities are true for $x_1^2-4$ and $x_3^2-4$. This is true for all odd prime factors of $k-4$. Thus if there is such a prime factor with one of $x_i^2-4$ is a QNR, we get that ${\bf x}$ is not $(\mathcal{E}_2)$-good. Then, by Lemmas \ref{m2} and \ref{m3}, if this property holds for each point of a fundamental set, it holds for all points of $\mathcal{M}_k(\mathbb{Z})$, as required.
\end{proof}
\begin{remark}\ 

 The analysis in this section corresponds to the an extension of $(\mathcal{E}_1)$ to subgroups $\Omega$ of $\Gamma_D$. Here, we choose $\Omega$  to be a  subgroup of $\Gamma_D$ invariant under the map $X \to -X$, so that it is closed under the Nielsen automorphisms. We then ask for solutions $(X,Y)\in \Omega\times \Omega$ satisfying $Z=W(X,Y)$ for a given $Z\in\Omega$; the analogue of $(\mathcal{E}_2)$ is described similarly. The projection of $(\mathcal{E}_1)$ and $(\mathcal{E}_2)$ then gives a subset of solutions to $(\mathcal{E}_3)$ with the coordinates of ${\bf x}$ restricted in some way, as illustrated below.

\begin{example} \ Let $p$ be an odd prime, and $D=\Z$.

(i). Let $\Omega =\left\{ U\in \G\ : U \equiv \pm I \pmod p\right\}$. Then $\Z=W(X,Y)$ being solvable in $\Omega$ implies $Z\equiv I \pmod p$, so that $\Tr Z =t\equiv 2 \pmod p$. The projection from $(\mathcal{E}_1)$ to $(\mathcal{E}_3)$ gives points Markoff-equivalent to $(2,2,2)$ modulo $p$ . 

One can generalize this to subgroups $\Omega_a =\left\{ U\in \G\ : U \equiv \pm R_a^m \pmod p , m\in \Z\right\}$ where $R_a = \begin{psmatrix} a&1\\-1&0 \end{psmatrix}$, so that a commutator is congruent to $I$. By choosing $a$ suitably, the corresponding points in $(\mathcal{E}_3)$ are Markoff-equivalent to $(2,b,b)$ modulo $p$ for all $b$ (such points $(2,b,b)$ represent the Markoff-inequivalent orbits on the Cayley surface $\mathcal{M}_4$ (see \cite{GhSa})).

(ii). Let $\Omega = \G_0(p)$. Then $\Z=W(X,Y)$ implies $Z \equiv \begin{psmatrix} 1&*\\0&1 \end{psmatrix} \pmod p$, so that $t\equiv 2 \pmod p$. Then, $x_1= \Tr X$ implies $x_1^2 -4$ is a QR modulo $p$; similarly for $x_2 = \Tr Y$. Then, $x_3^2 -4$ is a QR since $x_3=\Tr XY$ (this is reflected in Lemma \ref{m1}).

\end{example}
\end{remark}

\subsection{\ $D=\Z[\sqrt{-d}]$, Bianchi groups\nopunct}\

For $d <0$ squarefree, the groups $\SL_2(D)$ are known as the Bianchi groups. The commutator story for these is similar to that of $D=\Z$. There are only finitely many units in $D$ and the congruence subgroup property is known to fail (\cite{serre70}). In principle one can still proceed to give an effective procedure for equations $(\mathcal{E}_1)$ and  $(\mathcal{E}_2)$, thanks to the type of decidability results of  \cite{Sela}. While Sela's general results apply only to torsion-free Gromov hyperbolic groups, we understand from Ian Agol that one can decide the commutator problem for hyperbolic 3-manifold groups. However, one does not have an explicit description as with Wick's theorem so that the feasibility of implementing these procedures is less clear. On the other hand, the bottom up approach works equally well here as it does over $\Z$. The key is that $\mathcal{M}_{t+2}(D)/G$ is finite (for generic points) here as well. This was proved for $t=-2$ in \cite{Si90}, and using Markoff descent this extends to all $t\neq -2$. The geometric descent in \cite{whang} is carried out over these rings and also yields this finiteness. As explained in Section \ref{general} this effective finiteness together with the general lifting procedures, yields an effective and feasible means to decide all of $(\mathcal{E}_1)$, $(\mathcal{E}_2)$ and $(\mathcal{E}_3)$. While we have not implemented this procedure, we expect that doing so will yield results similar to that of $D=\Z$ in Section \ref{sec32}. Whether these $\G$'s satisfy a profinite local to global principle (the analogue of Theorem \ref{thm:main}) is less clear.

\section{Universal domains}\label{sec4}

\begin{definition}
For a domain $D$, ff any of $(\mathcal{E}_1)$, $(\mathcal{E}_2)$, $(\mathcal{E}_3)$ has a solution for all choices of the right hand side in $D$, we say that $D$ is {\it universal} for that $(\mathcal{E}_j)$. By abuse of language, we also say the corresponding $(\mathcal{E}_j)$ is universal, if the domain $D$ is implicit.
\end{definition}

We do not know of a choice of any of our domains $D$ for which $(\mathcal{E}_1)$ is universal. However for $(\mathcal{E}_2)$ and $(\mathcal{E}_3)$ we give some examples below.

\begin{proposition}\label{a1} Let $D$ be any ring containing a unit $\ve$ such that $\ve - \ve^{-1}$ is also a unit. Then, $(\mathcal{E}_2)$ and $(\mathcal{E}_3)$ are universal.
\end{proposition}

\begin{proof} For any $t\in D$, put $\eta =\ve -\ve^{-1} \in D^{\times}$ and  $\tau = (t-2)\eta^{-2} \in D$. 

For $(\mathcal{E}_2)$, we use  $X= \left[\begin{smallmatrix} \ 1&\tau\\-1\ &1-\tau \end{smallmatrix}\right]$ and $Y=\left[\begin{smallmatrix} \ve&0\\0& \ve^{-1}\end{smallmatrix}\right]$, giving $\Tr W(X,Y)=t$.

For $(\mathcal{E}_3)$, take $x_1=2-\tau$, $x_2= \ve + \ve^{-1}$, $x_3= \ve + (1-\tau)\ve^{-1}$ and $t=k-2$, for any $k\in D$.
\end{proof}

\begin{example}\label{examples_sec4}\
\begin{enumerate}
\item For any $q$ with $(6,q)=1$, the equations in $(\mathcal{E}_2)$ and $(\mathcal{E}_3)$ are solvable in $D=\Z/q\Z$. In particular, there are no congruence obstructions over $\Z$ in $(\mathcal{E}_2)$ and $(\mathcal{E}_3)$ modulo primes powers exceeding 3. Here, we take $\ve =2$.
\item For the $S$-integers, take $D=\mathbb{Z}[\frac{1}{6}]$  using the unit $\ve=2$. 
\item Let $F$ be the complex cubic field determined by the roots $\rho$  of the polynomial $x^3 - x^2+1$, and let $D=\mathcal{O}_K$, the ring of integers generated by $1,\rho$ and $\rho^2$.  It has discriminant -23 and has class number one. Since $\rho^2(1-\rho)=1$, $\rho$ is a unit, and so is $ \rho - \rho^{-1}=\rho^2 $. 
\end{enumerate}
Then Prop \ref{a1} holds in all cases, as does Prop. \ref{a2} below. 
\end{example}

Another version of universality for $(\mathcal{E}_3)$ is 

\begin{proposition}\label{a3}
If $\text{char}(D)\neq 2$ and $D$ contains elements $z$ and $w$ with $w$ a unit satisfying $z^2-4=w^2$, then $(\mathcal{E}_3)$ is universal for $D$.
\end{proposition}
\begin{proof}
The split binary form $f(x_1,x_2)=x_1x_2$ is universal for any $D$. Given $z$ as above, we put $x_3=z$ in $(\mathcal{E}_3)$. The restriction of $M(x_1,x_2,x_3)$ to this hyperplane section yields a binary quadratic form whose homogenous part is $g=x_1^2 + x_2^2 -zx_1x_2$, having discriminant $w^2$. Since $w$ is a unit, $g$ is equivalent over $D$ to $f$, giving universality. This construction is similar to the construction of the universal form $U_2$ in Section 5.1 of \cite{GhSa}.

Since  $\zeta=\frac{z+w}{2}$ is a unit, we obtain a point  in $\mathcal{M}_k(D)$ with coordinates $x_1=w^{-1}(1-k+z^2)$, $x_2=w^{-1}\left(\zeta -(k-z^2)\zeta^{-1}\right)$ and $x_3=z$, for any  $k\in D$. 
\end{proof}

\ 

For $D$ a PID, we can say a bit more. 

\begin{definition}\label{traceset}
Given a matrix $A$ defined over $D$, we let $\mathfrak{S}(A)$ be the set over $D$ given by
\[
\mathfrak{S}(A)=\{X : \Tr AX=\Tr X \ , |X|=1\}.
\]
\end{definition}

\begin{proposition}\label{a2} Let $D$ be a  PID containing a unit $\ve$ such that $\eta=\ve - \ve^{-1}$ is also a unit. Let $Z\in \Gamma_D=\SL_2(D)$,  and suppose there exists a $U\in \mathfrak{S}(Z)$ with $\Tr U=\ve +\ve^{-1}$. Then $Z$ is a commutator.
\end{proposition}
\begin{proof}(See Remark \ref{altproof} for another proof).

We first note that $\mathfrak{S}(\sigma Z\sigma^{-1})=\sigma^{-1}\mathfrak{S}(Z)\sigma$ for any $\sigma \in \SL_2(D)$. Moreover, $Z$ is a  commutator if any conjugate of it is. Thus, it suffices to verify the proposition by using a suitable conjugate of the matrix $U$. Since $U$ has $\ve$ for an eigenvalue and $D$ is a PID, there is a matrix $M\in\SL_2(D)$ such that $M^{-1}UM = U_0 :=  \big( \begin{smallmatrix} \ve&\alpha\\0&\ve^{-1} \end{smallmatrix}\big)$ for some $\alpha\in D$. Conjugating $U_0$ with $\left[\begin{smallmatrix} 1&\alpha\eta^{-1}\\0&1 \end{smallmatrix}\right]$ shows that we may take $\alpha=0$. So we will prove the proposition with $U=U_1:=\left[\begin{smallmatrix} \ve&0\\0& \ve^{-1} \end{smallmatrix}\right]$.

It suffices to show that $ZU_1$ and $U_1$ are conjugates in $\Gamma$, since then $(\mathcal{E}_1)$ follows with $Y=U_1$. Putting $B=ZU_1=\left[\begin{smallmatrix} b_1&b_2\\b_3&b_4 \end{smallmatrix}\right]$, we seek $X=\left[\begin{smallmatrix} x_1&x_2\\x_3&x_4 \end{smallmatrix}\right]\in\Gamma$ such that $BX=XU_1$. Since $U_1 \in   \mathfrak{S}(Z)$, we have $b_1 +b_4 = \ve + \ve^{-1}$, so that $(b_1 -\ve)(b_4 -\ve) = b_2b_3$. Assume $b_2 \neq 0$ and put $\delta = \text{gcd}(b_1 - \ve,b_2)$. If $\delta =0$, then $b_1=\ve$, $b_4 =\ve^{-1}$ and $b_3=0$. Then put $X=\left[\begin{smallmatrix} -\eta^{-1}&b_2\\0&-\eta \end{smallmatrix}\right]$. If $\delta \neq 0$, then we have $\frac{b_2}{\delta}\mid (b_4 -\ve)$ and so we put $X=\left[\begin{smallmatrix} \frac{b_2}{\delta}&-\delta\eta^{-1}\\\frac{\ve -b_1}{\delta}&-\frac{(b_4 - \ve)\delta}{b_{2}\eta} \end{smallmatrix}\right]\in \Gamma$. The argument is similar if $b_3 \neq 0$. Finally, if $b_2$ and $b_3$ are both zero, then $B$ is necessarily $U_1$ or its inverse, for which $X$ is either the identity or $\left[\begin{smallmatrix} 0&-1\\1&\ \ 0 \end{smallmatrix}\right]$.

\end{proof}

\begin{remark}\label{tracezero_1} It is not true that if $D$ is a  PID containing a unit $\ve$ such that $\eta=\ve - \ve^{-1}$ is also a unit, that then $(\mathcal{E}_1)$ is universal. This can be seen with $D=\mathbb{Z}[\frac{1}{6}]$, for which $-I$ is not a commutator. 

One can show that if $D$ is a PID, then $-I$ is a commutator if and only if the equation $r_1^2+r_2^2+r_3^2=0$ has a non-trivial solution in $D$ (see Remark \ref{tracezero}).
\end{remark}

\begin{remark}\label{altproof}\

 We give here an alternative proof of Prop. \ref{a2} using lifting as in the proof of Lemma \ref{lift2}. All notation is that used in Prop. \ref{a2}, and Lemma \ref{lift2}. The idea is to first construct a suitable point ${\bf x}$, a matrix $Y$ and then verify that the matrix $X$ in \eqref{matrix-X} is in $\Gamma_D$. We put $Y=U_1=\left[\begin{smallmatrix} \ve&0\\0&\ve^{-1}\end{smallmatrix}\right]$ and assume $Y\in \mathfrak{S}(Z)$;  use $x_2=\Tr Y=\ve+\ve^{-1}$ and $\Delta=t+2-x_2^2$.

Assume first that $\Delta\neq 0$. Write $Z-Y^{-2}=\left[\begin{smallmatrix} a&b\\c&d\end{smallmatrix}\right]$, with trace and determinant equaling $\Delta$. Then, $Y\in \mathfrak{S}(Z)$ implies $\Tr (Z-Y^{-2})Y=0$, so that $a\ve +d\ve^{-1}=0$. If $c=0$, a small calculation shows that $Z=\left[\begin{smallmatrix} 1&*\\0&1\end{smallmatrix}\right]$, and this is a commutator (since $ZY$ and $Y$ are conjugates, using $\eta$ is a unit). So we assume $c\neq 0$, and put $\mu=\gcd(c,\Delta)$. We then choose
\[
x_1 = \eta^{-1}\left(\ve\mu -\frac{\Delta}{\ve\mu}\right)\quad \text{and} \quad x_3 = \eta^{-1}\left(\mu -\frac{\Delta}{\mu}\right).
\]
The point ${\bf x}=(x_1,x_2,x_3)\in\mathcal{M}_{t+2}(D)$. Then, $X$ in \eqref{matrix-X} satisfies $\det X=1$, with 
\[
X= \Delta^{-1}\left[\begin{matrix} \Delta a\mu^{-1}&b\mu\\ \Delta c \mu^{-1}&d\mu\end{matrix}\right].
\]
 We need to show that $\Delta$ divides all entries in the matrix. We have $0=a\ve +d\ve^{-1}=a\ve +(\Delta -a)\ve^{-1}\equiv a\eta $ (mod $\Delta$), so that  $\Delta\vert a$ and $\Delta\vert d$. Using the determinant $-a^2 -bc \equiv 0$ (mod $\Delta$), we have $b \equiv 0$ (mod $\frac{\Delta}{\mu}$) so that $X\in \Gamma_D$.

Next, if $\Delta=0$, one concludes without much difficulty that 
\[
Z= \left[\begin{matrix} \ve^{-2}&* \\ 0&\ve^{2}\end{matrix}\right]\quad \text{or}\quad \left[\begin{matrix} z_1&z_2\\ z_3&t-z_1\end{matrix}\right],
\]
with $z_3$ free, $z_1= -\alpha\eta^{-1}z_3+\mu^{-2}$, and $z_2\in D$ determined. In either case, one shows that $Z$ is a commutator (showing $ZY$ is conjugate to $Y$, using elementary matrices and $\left[\begin{smallmatrix} 0&1 \\ -1&0\end{smallmatrix}\right]$).
\end{remark}

\section{Hasse failures}\label{sec5}

\subsection{\ Hasse failures for Markoff surfaces\nopunct}\label{HFM}

\subsubsection{\ \ Markoff surfaces and ternary quadratic forms, $\SLZ$-revisited\nopunct}

\begin{remark} We write this section mostly over $\mathbb{Z}$ but much carries over more generally to rings over number fields.
\end{remark}

For $D$ a ring, and $U(D)$ a subset of $D^3$, we set \[ \widetilde{U}(D) =\left\{{\bf X}=X({\bf x}):=\left[\begin{matrix} 2&x_1&x_2\\ x_1& 2&x_3\\x_2&x_3&2 \end{matrix}\right]\quad  | \quad{\bf x}=(x_1,x_2,x_3)\in U(D) \right\},\] and call the $x_j$'s coordinates of ${\bf X}$.

We will be primarily concerned with $D=\mathbb{Z}$ and $U= \mathcal{M}_k(\mathbb{Z})$, the Markoff surface. Then, there is a one-one correspondence ${\bf x} \leftrightarrow {\bf X}$ between points in $\mathcal{M}_k(\mathbb{Z})$ and  matrices in $\widetilde{\mathcal{M}}_k(\mathbb{Z})$ of determinant $-2(k-4)$. To each such point ${\bf x}$ and so matrix ${\bf X}$, we can then associate the ternary quadratic form \[f_{\bf X}({\bf u}):= \frac{1}{2}{\bf u}^\top {\bf Xu}=u_1^2 + u_2^2+u_3^2 + x_1u_1u_2 + x_2u_1u_3 + x_3u_2u_3,\] where ${\bf u}=(u_1,u_2,u_3)$.

For $k>4$ generic (see \cite{GhSa}, pg.2) and ${\bf X}\in \widetilde{\mathcal{M}}_k(\mathbb{Z})$, all such  quadratic forms are indefinite, with signature (2,1). For each prime $p$, we consider the Hasse-invariant $c_p(f_{\bf X})$ (see below for details). The Hasse invariants satisfy the product formula $\prod_{p \leq \infty} c_p =1$. One way to construct Hasse-failures $k$ is to take any matrix  ${\bf X}\in \widetilde{\mathcal{M}}_k(\mathbb{Z})$, and show that the product formula is violated for the corresponding quadratic form. This can be done for all the examples  of Hasse failures (not using descent) we have constructed in \cite{GhSa}. Since the method is quite general and simple, we present it here, and also extend it to $S$-integers.

We will give some details below for completeness sake. Throughout, $k$ will be generic, $k>4$ and fixed. For any ${\bf X}\in \widetilde{\mathcal{M}}_k(\mathbb{Z})$ as above, and $\gamma\in GL(3,\mathbb{Z})$, we write $\gamma\cdot {\bf X} = \gamma^\top {\bf X}\gamma$, and  we say $\gamma$ is good if $\gamma\cdot {\bf X}\in \widetilde{\mathcal{M}}_k(\mathbb{Z})$.

Let
\[ D_1= \left[\begin{matrix} 1&0&0\\ 0& 1&0\\0&0&-1 \end{matrix}\right]\ ,\quad D_2=\left[\begin{matrix} 1&0&0\\ 0& -1&0\\0&0&1 \end{matrix}\right]\ ,\quad D_3 =\left[\begin{matrix} -1&0&0\\ 0& 1&0\\0&0&1 \end{matrix}\right]\ .\] 

These matrices are all good for any ${\bf X}\in \widetilde{\mathcal{M}}_k(\mathbb{Z})$, with $D_j\cdot {\bf X}$ corresponding to a double sign-change on ${\bf x}\in \mathcal{M}_k(\mathbb{Z})$. The standard permutation matrices $P_j$ are also all good, corresponding to permutations of the coordinates of ${\bf x}$. For the Vieta moves, being non-linear, the associated matrix action depends on the coordinates of  ${\bf X}$. We put
\[
V_1= \left[\begin{matrix} 1&0&0\\ 0& -1&0\\0&x_3&1 \end{matrix}\right]\ ,\quad V_2 = \left[\begin{matrix} -1&0&0\\ x_1& 1&0\\0&0&1 \end{matrix}\right]\ ,\quad V_3= \left[\begin{matrix} 1&0&x_2\\ 0& 1&0\\0&0&-1 \end{matrix}\right]\ ,\]
where $V_j\cdot {\bf X}$ is the matrix corresponding to the  Vieta move applied to the $j$-th coordinate of ${\bf x}$, keeping the other two fixed. Obviously, these matrices are not unique, since they can be multiplied on the left by any automorph of ${\bf X}$. Since the Markoff group is generated by the three types of involutions, the action of any member of the Markoff group on ${\bf x}$ gives rise to ${\bf y}$ whose corresponding matrix satisfies $X({\bf y})= \gamma\cdot X({\bf x})$, for some $\gamma\in GL(3,\mathbb{Z})$; the entries of $\gamma$ obviously depending on both the Markoff action and $X({\bf x})$. We will say a $\gamma\in GL(3,\mathbb{Z})$ of the type discussed here, coming from a Markoff group action on ${\bf x}$ is of $M$-type; thus, given ${\bf X}$, a $M$-type matrix is a word generated by the matrices $D_j$, $P_j$ and specific $V_j$ depending on the coordinates of ${\bf X}$. 

By descent (see \cite{GhSa} for example), since $\mathcal{M}_k(\mathbb{Z})$ is a finite union of Markoff-orbits, there exists a fundamental set of matrices $\{ {\bf X}_1, ..., {\bf X}_h\}$ such that  any ${\bf X}\in \widetilde{\mathcal{M}}_k(\mathbb{Z})$ is $M$-type equivalent to a unique ${\bf X}_j$ (here ${\bf X}_j = X({\bf z}_j)$ with $\{{\bf z}_1, ... , {\bf z}_h\}$ a fundamental set in $\mathcal{M}_k(\mathbb{Z})$ with $h=\mathfrak{h}(k)$ of \cite{GhSa}). Thus, the same statement holds for the corresponding ternary quadratic forms: there are ternary quadratic forms $f^{(1)}, ..., f^{(h)}$ that are inequivalent under $M$-type transformations, and any $f_{\bf X}$ is $M$-type equivalent to a unique $f^{(j)}$. Since $M$-type equivalence implies $GL(3,\mathbb{Z})$-equivalence, all matrices coming from ${\bf x}$ in the same Markoff-orbit give rise to ternary forms in the same $GL(3,\mathbb{Z})$-orbit, and as such, will all have the same invariants (this being important in our discussion of Hasse failures).

We first look at some examples regarding fundamental sets. One can ask various questions about the ternary quadratics we get via the Markoff fundamental sets: isotropic, anisotropic, genera, $GL(3,\mathbb{Z})$-equivalence etc.. Obviously, if $\mathfrak{h}(k)=1$, all the associated quadratic forms are $GL(3,\mathbb{Z})$-equivalent. We begin with  Legendre's theorem and  consequences, together with some other verified facts:
\begin{lemma}\ 
\begin{enumerate}[label=\upshape(\arabic*).]
\item Legendre's theorem: let $a>0$, $b, c<0$ be integers with $abc$ squarefree. Then, $ax^2  +by^2+cz^2$ is isotropic if and only if $-ab$ is a square mod $c$, $-ac$ is a square mod $b$ and $-bc$ is a square mod $a$.
\item Suppose  $m$ is the squarefree part of $x_j^2-4$ for some $j$, and $n$ is the squarefree part of $k-4$ with $(m,n)=1$. Then $f= u_1^2 + u_2^2+u_3^2 x_1u_1u_2 + x_2u_1u_3 + x_3u_2u_3$ is isotropic if and only if $m$ is a square mod $n$ and $n$ is a square mod $m$.
\item Suppose there exists an odd prime $p$ such that $p\parallel (k-4)$ and $p\nmid (x_j^2 -4)$ for some $j$. If $x_j^2 -4$ is a quadratic non-residue mod $p$, then the associated quadratic form is anisotropic.
\end{enumerate}
\end{lemma}
\begin{proof}
 We may consider only the case $j=1$. If $f= u_1^2 + u_2^2+u_3^2 x_1u_1u_2 + x_2u_1u_3 + x_3u_2u_3$, then $4(x_1^2-4)f=(x_1^2 -4)A^2 - B^2 + 4(k-4)u_3^2$, where $A=2u_1+x_1u_2+x_2u_3$ and $B=(x_1^2-4)u_2 - (2x_3-x_1x_2)u_3$. Hence, $(4-x_1^2)f$ is $GL(3,\mathbb{Q})$ equivalent to $ax^2 +by^2+cz^2$ with $a=1$, $b=-m$  and $c=-n$.  If $m$ and $n$ are coprime,  the conditions of Legendre's theorem are satisfied. It is  isotropic if and only if $m$ is a square (mod $n$) and $n$ is a square (mod $m$). Part (3) is a consequence.
\end{proof}
\begin{examples}\
\begin{enumerate}[label=(\arabic*)]
\item $\mathfrak{h}(70)=1$ with $f^{(1)}= u_1^2 + u_2^2+u_3^2 -3u_1u_2 + 3u_1u_3 + 4u_2u_3$. By the lemma above, this is anisotropic with $p=3$ and $x_1=-3$..
\item $\mathfrak{h}(3780)=1$ with $f^{(1)}= u_1^2 + u_2^2+u_3^2 -3u_1u_2 + 3u_1u_3 + 57u_2u_3$. This is isotropic with  $f^{(1)}(409,251,50)=0$. Here, $m=5$ and $n=59$, satisfying the conditions of the lemma.
\item $\mathfrak{h}(460)=2$,  with anisotropic forms (the lemma holds with $p=3$ and $x_1=-3$)
\begin{align*} f^{(1)} &= u_1^2 + u_2^2+u_3^2 -3u_1u_2 + 3u_1u_3 + 17u_2u_3\, , \\
f^{(2)} &= u_1^2 + u_2^2+u_3^2 -3u_1u_2 + 9u_1u_3 + 10u_2u_3\, .
\end{align*}
However, $f^{(1)}(\gamma {\bf u})= f^{(2)}({\bf u})$ with $\gamma = \left[\begin{smallmatrix} 1&0&-18\\ 0& 1&-16\\0&0&-1 \end{smallmatrix}\right]$ an involution. The fact that $\gamma$ is not $M$-type is a non-trivial consequence of reduction theory and uses the fundamental sets on $\mathcal{M}_{460}(\mathbb{Z})$. Using $\gamma$ in combination with the generators of the Markoff group, one can send any matrix in $\widetilde{\mathcal{M}}_{460}(\mathbb{Z})$ to any other.
\vspace{4pt}
\item $\mathfrak{h}(329)=2$, with 
\begin{align*}
 f^{(1)} &= u_1^2 + u_2^2+u_3^2 -3u_1u_2 + 8u_1u_3 + 8u_2u_3\, , \\
 f^{(2)} &= u_1^2 + u_2^2+u_3^2 -4u_1u_2 + 4u_1u_3 + 11u_2u_3\, .
 \end{align*}
  $f^{(1)}$ is anisotropic with $p=13$ and $x_1=-3$, while $f^{(2})$ is isotropic with $m=3$ and $n=13$ with $f^{(2)}(9,1,-1)=0$. The two forms lie in different genera; not only are they $M$-type inequivalent, they are $GL(3,\mathbb{Q})$-inequivalent. The fact that they are $GL(3,\mathbb{Q})$-inequivalent gives a proof that $\mathfrak{h}(329)\geq2$ without reduction theory.
\vspace{4pt}
\end{enumerate}
\end{examples}

\subsubsection{\ \ Some Hasse failures revisited\nopunct}\ 

We use the notation in \cite{Pall1945}: for non-zero rational numbers $a$ and $b$, and any prime $p$, $(a,b)_p$ has values $\pm 1$ depending on the solvability over $\mathbb{Q}$ of the equation $ax^2 + by^2 \equiv 1$ (mod $p^r$) for all $r$. The case $p=\infty$ corresponds to solvability over the reals and in all cases consider hereon, we will have $(a,b)_{\infty}=1$. This is a reformulation of the Hilbert symbol over the $p$-adics. We state various properties that we will use repeatedly:

\begin{enumerate}[label=(\arabic*).]
\item $(a,b)_p = (a',b')_p$ if $a,\ a'$ and $b,\ b'$ differ multiplicatively by squares modulo prime powers of $p$;
\item $(a,b)_p=(b,a)_p$;
\item $(a,1)_p=(a,1-a)_p=(a,-a)_p=1$;
\item $(a,a)_p=(a,-1)_p$;
\item $(a,b_1b_2)_p=(a,b_1)_p(a,b_2)_p$;
\item for odd $p$, $(a,b)_p=1$ if $a$ and $b$ are units in $\mathbb{Q}_p$;
\item for odd $p$, $(p^\alpha m,p^\beta n)_p = (-1|p)_J^{\alpha\beta}(m|p)_J^\beta ((n|p)_J^\alpha$;
\item $(2^\alpha m,2^\beta n)_2 = (2|m)_J^{\beta}(2|n)_J^\alpha (-1)^{(m-1)(n -1)/4}$,
\end{enumerate}
with integers $m,\ n$ coprime to $p$, with $\alpha$ and $\beta$ being 0 or 1, and $(x|y)_J$ the Jacobi symbol. One also has the validity of the product formula $\prod_{p\leq \infty}(a,b)_p =1$.

The Hasse invariant $c_p(f)$ of the quadratic form $f= u_1^2 + u_2^2+u_3^2 + x_1u_1u_2 + x_2u_1u_3 + x_3u_2u_3$ is given by $c_p(f)=C_p(x_1)= (x_1^2 -4,k-4)_p$, for any $(x_1,x_2,x_3)\in \mathcal{M}_k(\mathbb{Z})$. Note that $C_p(x_1)=C_p(x_2)=C_p(x_3)$ since permutations correspond to $GL(3,\mathbb{Z})$-equivalence for the forms, and the same being true for the coordinates of any Markoff-equivalent points. These invariants are well defined since we are assuming $k>4$ is generic so that $x_j^2-4$ is never zero.  We now derive proofs of Propositions 8.1(i,ii), 8.2 and 8.3 of Hasse failures in \cite{GhSa}. 

\begin{proposition}\label{HFz} If $(x_1,x_2,x_3)\in \mathcal{M}_k(\mathbb{Z})$, for $k$ generic, the product formula fails for the following choices of $k$:
\begin{enumerate}[label=(\roman*).]
\item $k= 4 + 2\nu^2$ with $\nu$ having all of its prime factors in the congruence classes $\{\pm 1\}$ modulo 8;
\item $k= 4 + 12\nu^2$ with $\nu$ satisfying $\nu^2 \equiv 25$ (mod 32), and having all of its prime factors in the congruence classes $\{\pm 1\}$ modulo 12;
\item $k= 4 + 20\nu^2$ with $\nu$ having all of its prime factors in the congruence classes $\{\pm 1\}$ modulo 20.
\end{enumerate}
\end{proposition}
\begin{proof}\ \par
\noindent(i). We have $C_p(x_1)=(x_1^2-4,2)_p$, so that $C_p=1$ if $p>2$ and $p\nmid(x_1^2-4)$. We note that $C_2(x_1)=(2|x_1^2-4)_J=-1$ if $x_1$ is odd (since $x_1^2-4\equiv -3$ modulo 8). Moreover, since $4\nmid k$, at least one coordinate must be odd, so that $C_2=-1$ always holds. To complete our proof, we need to ensure that $C_p=1$ for all odd primes dividing $x_1^2 -4$.  If  $p^{\alpha}\parallel(x_1^2-4)$ then $C_p(x_1)=(x_1^2-4,2)_p= (2|p)_J^\alpha$. Using $A^2 - (x_1^2-4)(x_2^2-4)= 8\nu^2$, for some $A$, we see that if $p\nmid \nu$ and $p|(x_1^2-4)$, then 2 is a quadratic residue mod $p$, so that $C_p(x_1)=1$ as needed.  On the other hand if $p$ divides both $x_1^2-4$ and $\nu$, we cannot deduce that $(2|p)_J^\alpha=1$, unless we force $(2|p)_J$ to equal one, which can be done if all primes factors of $\nu$ are restricted to $\pm 1$ modulo 8. Since $C_p=-1$ for exactly one prime, the product formula is violated.\newline

\noindent(ii , iii). We consider $k=4+4q\nu^2$, where $q$ is an odd prime not dividing $\nu$, which is odd ($q$ will be chosen later). Then, $C_p(x_1)=(x_1^2-4,q)_p$. We consider various cases:
\begin{enumerate}[label=(\alph*).]
\item If $p\nmid 4q\nu^2$, suppose $p^\alpha\parallel(x_1^2-4)$ with $\alpha\geq 0$.then, $C_p(x_1)=(q|p)^\alpha_J=1$ since $A^2 - (x_1^2-4)(x_2^2-4)= 4(k-4)=16q\nu^2$ implies that $q$ is a quadratic residue modulo $p$ when $\alpha\geq 1$.
\item If $p|\nu$, then $p\nmid q$, so using the same notation as in (a), we have $C_p(x_1)=(q|p)^\alpha_J$. We force $(q|p)_J$ to equal one. So, if $q\equiv 1$ (mod 4), we insist all prime factors of $\nu$ are congruent to quadratic residues modulo $q$. For $q\equiv 3$ (mod 4), we insist all prime factors of $\nu$ that are congruent to 1 (mod 4) are also quadratic residues modulo $q$; while those congruent to 3 (mod 4) are  quadratic non-residues modulo $q$. This ensures that $C_p=1$. 

Example: when $q=3$, the conditions needed are all prime factors of $\nu$ are congruent to $\pm 1$ (mod 12), while the condition for $q=5$ are the classes $\pm1$ (mod 20).
\item The case $p=q$. If $q^\alpha\parallel(x_1^2-4)$, then from $A^2 - (x_1^2-4)(x_2^2-4)= 16q\nu^2$, we conclude that $\alpha=0$ or $1$. But if $\alpha=1$ then $q\nmid ((x_2^2-4)(x_3^2-4)$, so that using $x_2$ or $x_3$ instead of $x_1$ leads us to consider only the case $\alpha=0$. In that case we have $C_q(x_1)=(w^2-4,q)_q= (w^2-4|q)_J$ with $w=x_1$ or $x_2$ satisfying $q\nmid(w^2-4)$. we cannot say more without choosing special values of $q$, 

When $q=3$, $3\nmid(w^2-4)$ implies $3|w$, so that $C_3=(-1|3)_J=-1$.

When $q=5$, $5\nmid(w^2-4)$ if and only if $5|w$ or $w^2-4 \equiv 2$ (mod 5). In the latter case, $C_5=(2|5)_J=-1$ . In the former case, $C_5=(-4|5)_J=1$, which is problematic ((the same is true if we had used the third  coordinate and 5 divided it). We now show that this case cannot occur. Now if $w$ is odd, then $w^2-4$ must have an odd prime factor $p\equiv \pm2$ (mod 5), and $p\nmid \nu$ from (b). Then, $A^2 - (x_1^2-4)(x_2^2-4)= 80\nu^2$ implies $(5|p)_J=1$, a contradiction. Since $w$ is either of two coordinates, we may now assume they are both even, so that $x_1$ is too. Putting $x_j=2y_j$ gives us the equation $B^2 - (y_1^2-1)(y_2^2-1)= 5\nu^2$. If either $y_1$ or $y_2$ is odd, we have $5$ is a quadratic residue (mod 8) while if both $y_1$ and $y_2$ are even, we conclude $B^2 \equiv 2$ (mod 4). Hence these cases do not occur and $C_5=-1$.

\item When $p=2$, suppose $2^\alpha\parallel(x_1^2-4)$. We have three cases (a). $\alpha=0 \iff 2\nmid x_1$, (b). $\alpha =2 \iff 4\| x_1$ and (c). $\alpha\geq 5 \iff 2\parallel x_1$.

 For case (a), $C_2(x_1)=(-1)^{(x_1^2-5)(q-1)/4}=1$, since $x_1$ is odd as is $q$.

 For case (b), writing $x_1=4t$, we have \[C_2(x_1)=\left(\frac{x_1^2-4}{4},q\right)_2 = (-1)^{(q-1)(2t^2 -1)/2}.\] Then $C_2=1$ if $q\equiv 1$ (mod 4). For $q\equiv 3$ (mod 4), this gives $C_2=-1$, and we have to make sure this does not happen. This case happens when $4\| x_j$ for all $j$, and so writing $x_j=4y_j$, the Markoff equation gives $y_1^2 +y_2^2 +y_3^2 -4y_1y_2y_3 = \frac{1+q\nu^2}{4}$. If the right side is odd, then either all $y_j$'s are odd, or exactly one is. The left is then congruent to 7 (mod 8) or 1 (mod 4). So we restrict $\nu$ so that $\frac{1+q\nu^2}{4}\equiv 3$ (mod 8), so that these cases of $x_j$ do not occur. For $q=3$, this corresponds to restricting $\nu$ to satisfy $\nu^2 \equiv 25$ (mod 32). Hence $C_2=1$.

For case (c), put $x_1^2-4 = 2^\alpha\lambda$, with odd $\lambda$ and $\alpha\geq 5$. Then we have $C_2(x_1)= (2|q)_J^\alpha(-1)^{(\lambda -1)(q-1)/4}$. Now we have $B^2 - 2^{\alpha -4}\lambda(x_2^2-4)= q\nu^2$. If $q \equiv 3$ (mod 4), we must have $\alpha =5$ and $x_2$ odd, in which case we get $\lambda \equiv 3$ (mod 4). Then $C_2= (-1)^{(q^2-1)/8}(-1)^{(q-1)/2}= 1$ if $q\equiv 3$ (mod 8). Hence $C_2=1$ if $q\equiv 3$ (mod 8). 

Finally, if $q\equiv 1$ (mod 4), $C_2=(2|q)_J^\alpha$. Then $B^2 - 2^{\alpha -4}\lambda(x_2^2-4)= q\nu^2$ implies $q\equiv 1$ (mod 8) if $\alpha\geq 7$. we assume $q\equiv 5$ (mod 8) to avoid this case. the case $\alpha =6$ needs no work. Lastly if $\alpha =5$, we get $1-2\lambda(x_2^2-4) \equiv 5$ (mod 8), which is impossible.
\end{enumerate}
For $q=3$ and 5, we have shown $C_q=-1$ and $C_p=1$ for all $p\neq q$, violating the product formula.
\end{proof}

\subsubsection{\ \ Hasse failures over $S$-integers\nopunct}\ 

 We construct Hasse failures on the surface $\mathcal{M}_k(\Z[\frac{1}{\ell}])$ when $\ell$  is an odd prime satisfying $\ell \geq 5$ and some $k\in \Z$ . By Prop. 6.1 of \cite{GhSa}, the Markoff surface has solutions in $\Z/p^n\Z$ for all primes $p$ and integers $n\geq 1$ except for those $k$'s with congruence obstructions, coming from $p =2$ or 3. Hence the congruence  is solvable over $\mathcal{M}_k(\Z[\frac{1}{\ell}])$ when $p \neq 2, \ 3$. Since $\ell \neq 2,\ 3$, the congruence modulo $p^n$  over $\Z$ is equivalent to that over $\Z[\frac{1}{\ell}]$ when $p = 2$ or 3. Hence, there are no congruence obstructions if $k$ is not congruent to 3 (mod 4) and $\pm 3$ (mod 9).

\begin{lemma}\label{nonintegral} Suppose $\mathcal{M}_k(\Z[\frac{1}{\ell}])$ is non-empty. Then either there is an integer lattice point in $\mathcal{M}_k(\Z[\frac{1}{\ell}])$ or if not, there is a point of the form $(x_1, \frac{x_2}{\ell^a}, \frac{x_3}{\ell^a})$ with integers $x_j$ where $\ell\nmid x_2x_3$, and $a\geq 1$.
\begin{proof}
Consider the equation
\begin{equation}\label{surface}\left(\frac{x_1}{\ell^a}\right)^2 + \left(\frac{x_2}{\ell^b}\right)^2+\left(\frac{x_3}{\ell^c}\right)^2 - \left(\frac{x_1x_2x_3}{\ell^{a+b+c}}\right)=k,\end{equation}

with integers $a$, $b$, $c$  and with $\ell\nmid x_1x_2x_3$.

If $\max{(a,b,c)}\geq 1$, then by checking denominators, we may assume that either   $\min(a,b,c)\geq 1$ or  $a\leq 0$, with $b,\ c\geq 1$.

Suppose there is a solution with $\min(a,b,c)\geq 1$. Rearranging the coordinates, we assume $a\leq b\leq c$, and among such triples, we choose one with minimal $c$. Comparing denominators implies that $a+b=c$ so that \eqref{surface} becomes
\[
\left(\frac{x_1}{\ell^a}\right)^2 + \left(\frac{x_2}{\ell^b}\right)^2+ \frac{x_3(x_3-x_1x_2)}{\ell^{2(a+b)}}=k .\]

If $a\leq b$, we conclude that $\ell^{2a}\mid (x_3-x_1x_2)$. Then, the point $\left(\frac{x_1}{\ell^a} , \frac{x_2}{\ell^b},\frac{(x_1x_2 -x_3)/\ell^{2a}}{\ell^{b-a}}\right)$ has exponents in $\ell$ less than $c$. This contradicts the minimal choice of $c$ unless $a=b$, in which case we have a point with a coordinate over $\Z$. So we can assume that in \eqref{surface} we have a point with  $a\leq 0$, and $b,\ c\geq 1$, and checking denominators, conclude that $b=c$.  
\end{proof}

\end{lemma}

\begin{proposition}\label{HF1} 
Suppose $\ell$ is a prime satisfying $\ell \equiv \pm 1\  \text{(mod 8)}$.  Let $k= 4 + 2\nu^2$ with $\nu$ having all of its prime factors in the congruence classes $\{\pm 1\}$ modulo 8, and in addition with  $\nu \in \{0,\ \pm3,\ \pm4\}$ modulo 9. Then $k$ is a Hasse-failure in $\mathbb{Z}[\frac{1}{\ell}]$.
\end{proposition}
\begin{proof}
The conditions on $\nu$ ensures that there are no congruence obstructions in $\mathbb{Z}[\frac{1}{\ell}]$. By Prop. 8.1(b) of \cite{GhSa} (or Prop.\ref{HFz}\ (i)), there are no integral points. Applying Lemma \ref{nonintegral} and clearing denominators gives us 
\[
(2x_2 - x_1x_3)^2 - 8\nu^2\ell^{2a} = (x_1^2 -4)(x_3^2 - 4\ell^{2a}).\]
Since at least two of the $x_j$'s must be odd, we have chosen $x_3$ to be odd. Then, $x_3^2 - 4\ell^{2a} \equiv -3$ (mod 8), so that there must be a prime $q \equiv \pm 3$ (mod 8) dividing the right hand side. Since $q \nmid \nu\ell$, we get a contradiction as the equation implies that 2 is a quadratic residue modulo $q$.
\end{proof}

\begin{proposition}\label{HFTrace0}
Suppose $\ell$ is an odd prime satisfying $\ell \equiv \pm 1\  \text{(mod 5)}$.  Let $\nu \equiv \pm4\ \text{(mod 9)}$, with all prime factors congruent to $\pm 1\ \text{(mod 20)}$. Then, $k=4 + 20\nu^2 \in \mathbb{Z}$ is a Hasse-failure in $\mathbb{Z}[\frac{1}{\ell}]$.
\end{proposition}
\begin{proof}

By the preamble above, these $k$'s have no congruence obstructions. So it suffices to show that there are no lattice points in $\mathcal{M}_k(\Z[\frac{1}{\ell}])$. By Prop. 8.3 of \cite{GhSa}  (or Prop.\ref{HFz}\ (iii)), there are no integeral lattice points, so that by Lemma \ref{nonintegral} we have  \eqref{surface} in the form 
\begin{equation}\label{nonintegral2}
x_1^2 + \left(\frac{x_2}{\ell^a}\right)^2 + \left(\frac{x_3}{\ell^a}\right)^2 - \frac{x_1x_2x_3}{\ell^{2a}}=k  = 4 + 20\nu^2 ,
\end{equation}
with $x_1, x_2, x_3 \in \mathbb{Z}$, $a\geq 1$, and $\ell\nmid x_2x_3$. We see that $x_1$, $\frac{x_2}{\ell^a}$ and $\frac{x_3}{\ell^a}$ are not equal to $\pm 2$, so that we can apply the method considered above involving Hasse invariants. Exactly as we considered before, the Hasse invariants $c_p(f)$ are of the form $C_p(u_i) = (u_i^2 -4,5)_p = (x_i^2 -4\ell^{2a_i},5)_p$, where we have written $u_i= \frac{x_i}{\ell^{a_i}}$ with $a_i =0$ or $a$ as given above in \eqref{nonintegral2} (but not necessarily in that order). Completing the square gives $A^2 - (x_1^2 -4\ell^{2a_1})(x_2^2 -4\ell^{2a_2}) = 5(4\ell^{a_1+a_2}\nu)^2$, all integers.  We now follow the same procedure as in Prop. \ref{HFz} (iii) above, where we give the details for completeness.

\begin{enumerate}[label=(\alph*).]
\item If $p\nmid 10\ell\nu$, then 5 is a quadratic residue modulo $p$ when  $p^\al\parallel(x_j^2 -4\ell^{2a_j})$ for any $j$ with $\al \geq 1$. For such a $j$ we have $C_p(u_j) = (5,p)^\al_J =1$. On the other hand, if $p\nmid(x_j^2 -4\ell^{2a_j})$, then $\al=0$ so that $C_p(u_j)=1$.
\item If $p|\nu$, using the same notation as in (a), we have $C_p(u_j)=(5|p)^\alpha_J=(p|5)^\alpha_J =1$. 
\item When $p=\ell$, there are two cases to consider namely when $a_j=0$ and when $a_j \geq 1$. For the latter $C_\ell(u_j)= (x_j^2 - 4\ell^{2a_j},5)_\ell = (x_j^2\ , 5)_\ell =1$. Since at least one of the $a_j$'s is non-zero, we may permute the coefficients and get this value for $C_\ell$.
\item If $p=5$, $(x_j^2 -4\ell^{2a_j}) \equiv x_j^2 -4$ (mod 5), so that on replacing $\nu$ with $\nu_1=\ell^{a_i+a_j}\nu$, noting that $\nu_1 \equiv \pm\nu \pmod 5$ and $\nu_1^2 \equiv \nu^2 \pmod 8$, we have exactly the same equations as in the case considered in Prop. \ref{HFz} (iii, c) with $q=5$, so that $C_5 = -1$.
\item When $p=2$, suppose $2^\alpha\parallel(x_1^2-4\ell^a_1)$. We have three cases: (a). $\alpha=0 \iff 2\nmid x_1$, (b). $\alpha =2 \iff 4\| x_1$ and (c). $\alpha\geq 5 \iff 2\parallel x_1$.

 For case (a), $C_2(u_1)=(-1)^{(x_1^2-4\ell^{2a_1}-1)(5-1)/4}=1$, since $x_1$ is odd.

 For case (b), writing $x_1=4t$, we have \[C_2(u_1)=\left(\frac{x_1^2-4\ell^{2a_1}}{4},5\right)_2 = (-1)^{(5-1)(4t^2 - \ell^{2a_1}-1)/4}=1.\] 

For case (c), put $x_1^2-4\ell^{2a_1} = 2^\alpha\lambda$, with odd $\lambda$ and $\alpha\geq 5$. Then we have $C_2(u_1)= (2|5)_J^\alpha(-1)^{(\lambda -1)(5-1)/4} = (-1)^\alpha$. 

Since $A^2 - (x_1^2 -4\ell^{2a_1})(x_2^2 -4\ell^{2a_2}) = 80\nu_1^2$, we have $A_1^2 - 2^{\al-4}\lambda(x_2^2 -4\ell^{2a_2})= 5\nu_1^2$, with $A_1$ odd. Then, if $\al \geq 7$, we have $A_1^2 \equiv 5\nu_1^2 \equiv 5$ (mod 8), a contradiction. If $\al =6$, clearly $C_2(u_1)=1$. Finally if $\al =5$, we get $1-2\lambda(x_2^2-4\ell^{2a_2}) \equiv 5$ (mod 8), which is impossible. Hence $C_2(u_1)=1$.
\end{enumerate}
We thus see that the product formula is violated, proving the proposition.

\end{proof}
 Matrices that are commutators give rise to points on the Markoff surface, and checking admissibility in \eqref{3.2.2} (see Prop. \ref{HFu2}), we have 
 
\begin{proposition}\label{HFTrace}
Let $\ell$ be an odd prime satisfying $\ell\equiv \pm 1\  \text{(mod 5)}$. Then, there are infinitely many Hasse failures for $(\mathcal{E}_2)$ over $D= \Z[\frac{1}{\ell}]$. In particular, if  $t=2 +20\nu^2$ with $\nu$ as above, then $t$ is admissible but is a Hasse failure.
\end{proposition}

\subsection{\ Hasse failures for commutators\nopunct}\label{HFC}\ 

Our purpose here is to construct failures to the Hasse principle for the commutator problem $(\mathcal{E}_1)$: construct matrices $Z$ in an arithmetic group $\G_D$ over a ring $D$ that are commutators in congruence subgroups, that lie in the commutator subgroup but that are not commutators (see the discussion in Section 3.2). Since we have constructed Hasse failures  $\mathcal{M}_k(D)$  for certain infinite family of  $k$'s,  for  $D = \mathbb{Z}$ and $\Z[\frac{1}{\ell}]$ for some primes $\ell$, it is natural to want to extend these to matrices with trace $t=k-2$. It turns out that this does not always produce Hasse failures for the matrix problem as sometimes a Hasse failure $\mathcal{M}_k(D)$ corresponds to a congruence obstruction in $(\mathcal{E}_1)$ and $(\mathcal{E}_2)$ for matrices with trace $t=k-2$. We first give some exampes of these in Prop. \ref{HFu}, and then construct some Hasse failures for $(\mathcal{E}_1)$.

\begin{lemma}\label{hfc1}\ 
\begin{enumerate}[label=\upshape(\alph*).]
\item Suppose $t\equiv 2\ve \pmod q$ with $\ve =\pm 1$ and $Z = \left[\begin{smallmatrix} \al& \bt\\\g&t-\al \end{smallmatrix}\right] \in \SL_2(\Z/q\Z)$. Then, if $(\bt, q)=1$, there is a $M\in \SL_2(\Z/q\Z)$ with
\[
Z \equiv M\left[\begin{matrix} \ve& \bt^{-1}\\0&\ve \end{matrix}\right]M^{-1} \ ({\rm mod}\ q).
\]
\item For $q =2,\ 3,\ 4$, if $Z\in \SL_2(\Z/q\Z)$ is congruent to $\left[\begin{smallmatrix} 1& s\\0&1 \end{smallmatrix}\right]$  (or its transpose) modulo $q$, for some $s$ coprime to $q$, then $Z$ cannot be a commutator.
\end{enumerate}
\end{lemma}
\begin{proof}
For (a), take $M= \left[\begin{smallmatrix} \bt& 0\\\ve -\al&\bt^{-1} \end{smallmatrix}\right]$.

For (b), it suffices to check one case; otherwise consider the transpose of $Z$. Write $Z=XYX^{-1}Y^{-1}$ with $X, Y \in \SL_2(\Z/q\Z)$.  Computing the traces $\Tr(ZY)\equiv \Tr(Y) \pmod q$ and $\Tr(ZX^{-1})\equiv \Tr(X) \pmod q$, we conclude that $X$ and $Y$ are upper-triangular (mod $q$). With our choices of $q$, the diagonal entries must be congruent to $\ve =\pm 1$ so that $X$ and $Y$ commute, giving $Z \equiv I_2 \pmod q$, a contradiction.
\end{proof}

\begin{remark}\
The congruence obstructions above over $\Z$ also lead to congruence obstructions over $\Z[\frac{1}{\ell}]$ for $\ell \equiv \pm 1 \pmod 8$ for the first case, and $\ell \equiv \pm 1 \pmod {12}$ for the second.
\end{remark}

\begin{lemma}\label{HFu} For $Z \in \SLZ$, let $t=\Tr(Z)$. Then $Z$ is not a commutator if \upshape(a\upshape). $4|t$ or \upshape(b\upshape). $t \equiv \pm 1,  \pm 4 \pmod{9}$.
\end{lemma}
\begin{proof}\ 
Suppose there exist matrices $X,\ Y \in \SLZ$ such that $Z=XYX^{-1}Y^{-1}$. We write $Z= \left[\begin{smallmatrix} \al& \bt\\\g&t-\al \end{smallmatrix}\right]$. 

(a). If $\bt$ is odd, then by  Lemma \ref{hfc1}(a), with $q=2$, we may replace $Z$ with $\left[\begin{smallmatrix} 1& \bt\\0&1 \end{smallmatrix}\right]$ (mod 2), and then with Lemma \ref{hfc1}(b) conclude that $Z$ is not a commutator. The same reasoning holds if $\g$ was odd by transposing. So we can assume that both $\bt$ and $\g$ are even, so that $4|t$ implies that  the determinant is $-\al^2 \equiv 1$ (mod 4), a contradiction.

(b). We use the same argument as above with $q=3$, $t = 2\ve +3s$, and $\ve = \pm 1$. Then, 3 dividing both $\bt$ and $\g$ implies that $\al(t-\al) \equiv 1 \pmod 9$. Then, $(\al -\ve - 6s)^2 \equiv 12\ve s \pmod 9$, giving a contradiction when $3\nmid s$.
\end{proof}

The admissible $t$'s for $\mathcal{T}_t(\Z)$ may be described as follows:

let $\mathcal{V}$ be the affine variety in $\mathbb{A}^{\!8}$ given by the equations
\[
X=\left[\begin{matrix} x_1&x_2\\x_3&x_4\end{matrix}\right]\ , \quad Y= \left[\begin{matrix} y_1&y_2\\y_3&y_4\end{matrix}\right]\ , \quad \det X=\det Y=1 \ ,
\]
and let $f({\bf x},{\bf y})=\Tr (XYX^{-1}Y^{-1})$, with ${\bf x}=(x_1,x_2,x_3,x_4)$ and ${\bf y}=(y_1,y_2,y_3,y_4)$. The obstructions $\pmod {p^l}$, with $p$ a prime and $l\geq 1$ an integer, to $t\equiv  f({\bf x},{\bf y}) \pmod {p^l}$ having a solution is equivalent to describing the image in the $p$-adic integers $\Z_p$ of $f$, that is $f(\mathcal{V}(\Z_p))$.

By general principles of quantifier elimination for the $p$-adics (\cite{Cohen}), $f(\mathcal{V}(\Z_p))$ has an effective and explicit description for each $p\geq 2$.

\vspace{20pt}
\begin{proposition}\label{HFu2} \ 
\begin{enumerate}[label=\upshape(\alph*).]
\item For $p>3$, $f(\mathcal{V}(\Z_p)) =\Z_p$, that is there are no local obstructions for $t$ in $\Z_p$. 
\item For $p=3$, there are obstructions for $t \equiv 1, 4, 5, 8 \pmod {3^2}$.
\item For $p=2$, there are obstructions for $t \equiv 0, 1, 4, 5, 8, 9, 10, 12, 13  \pmod {2^4}$.
\end{enumerate}  
\end{proposition}

\begin{remark}\ 
Preliminary calculations indicate that the obstructions $\pmod {3^2}$ and $\pmod {2^4}$ are the only ones for $\Z_3$ and $\Z_2$ respectively, that is the admissible $t$'s for $\mathcal{T}_t(\Z)$ are as in \eqref{3.2.2}.
\end{remark}

\begin{proof}[Proof of Prop. \ref{HFu2}]
By Example \ref{examples_sec4}(1) in Section 4,  or by Lemma 6.2 of \cite{GhSa}, there are no congruence obstructions on $t$ modulo prime powers $p >3$, which proves the first statement. For $p=3$, the obstructions listed are those in Lemma \ref{HFu}. For $p=2$, we have the obstruction $t\equiv 1 \pmod 4$ coming from the Markoff equation with $t=k-2$, while the $t\equiv 0 \pmod 4$ obstruction follows from Lemma \ref{HFu}.  That leaves the case $t\equiv 10 \pmod {16}$. A matrix $Z \in \SL_2(\Z_2)$ (or its transpose) with trace $t$ is either equivalent to $A=\left[\begin{smallmatrix} 0&1\\1&0\end{smallmatrix}\right] \pmod 2$ or equivalent to  $B=\left[\begin{smallmatrix} 1&2\\0&1\end{smallmatrix}\right] \pmod 4$ or satisfies $Z \equiv I \pmod 4$. One checks directly that $A$ is not a commutator $\pmod 2$, and neither is $B \pmod 4$. Hence if $Z$ is a commutator $\pmod {16}$, then $Z \equiv I \pmod 4$, from which it follows that $t \equiv 2 \pmod{16}$.

\end{proof}

\begin{remark}
The Hasse failures for $\mathcal{M}_k(\Z)$ constructed in Prop. \ref{HFz} of the type $k=4+2\nu^2$ and $4+12\nu^2$ do not lift to Hasse failures for $(\mathcal{E}_1)$ and $(\mathcal{E}_2)$, since the corresponding $t$'s  are not admissible (with $4|t$ in the former, and $t \equiv 5 \pmod 9$ in the latter).
\end{remark}

We now extend Prop. \ref{HFTrace} and construct matrices that are Hasse failures to the commutator problem in the Ihara groups $\SLZl$ for primes $\ell \equiv \pm 1 \pmod 5$.

\begin{theorem}\label{HFE1}
Suppose $\ell$ is an odd prime satisfying $\ell\equiv \pm 1\  \text{(mod 5)}$ and put $\mathbb{D}=\Z[\frac{1}{\ell}]$. Put  $t=2 +20\nu^2$, where $\nu \equiv 4\ \text{(mod 27)}$  and with $\nu$ having  all prime factors congruent to $\pm 1 \pmod {20}$. Let
 \[ A=\left[\begin{matrix} t-5& \frac{2}{3}(5\nu  -2)\\6(5\nu  +2)&5 \end{matrix}\right]\in \SL_2(\mathbb{D}).\] Then, $A$  is a commutator in every finite (congruence) quotient of $\SL_2(\mathbb{D})$, but $A$ is not a commutator in $\SL_2(\mathbb{D})$.
 
  Consequently, there are infinitely many Ihara groups, each containing  infinitely many Hasse failures. 
\end{theorem}
\begin{proof}
 By Prop. \ref{HFTrace}, we know that $A$ is not a commutator in $\SL_2(\mathbb{D})$. Since all the finite quotients of $\Gamma_{\mathbb{D}}$ are congruence quotients (\cite{menn},\cite{serre70}), to show that $A$ is a Hasse-failure it remains to verify that $A$ is a commutator locally for all primes $p\geq 2$ (note that since $A\in \SLZ$, we do not have to consider the case  $p=\ell$ separately). We have $A\equiv I \pmod 2$, $A\equiv -I \pmod 3$ and $A \not\equiv \pm I \pmod p$ for all primes $p\geq 5$. Thus, by Theorem 3.5 of \cite{agks}, our conclusion holds for $p\geq 5$, so that it remains to check the cases $p=2$ and 3.

We write $A=I+2B$, where $B=\left[\begin{smallmatrix} b_1& b_2\\b_3&b_4 \end{smallmatrix}\right]\in \mathbb{M}_2(\mathbb{Z})$ with $b_1=2(5\nu^2 -1)$, $b_2=\frac{1}{3}(5\nu -2)$, $b_3=3(5\nu +2)$ and $b_4=2$. We use the following properties of the coefficients extensively: modulo 2, we have $2|b_1$, $2\| b_4$, $2\| (b_1-b_4)$ with $b_2,\ b_3$ odd; and modulo 3 we have $3\nmid b_1b_4$, $3\|b_2$, $3\| b_3$ and $3\|(b_1-b_4)$.

For the powers $2^n$ with $n\geq 2$, we appeal to Lemma \ref{lift2} with $D=\mathbb{Z}/2^n\mathbb{Z}$. Since $t+2 \equiv 0 \pmod 4$, $(1,1,1)\in \mathcal{M}_{t+2}(\mathbb{Z})$ is a non-singular solution modulo 2 and Hensel's lemma lifts this solution modulo $2^n$ (see Sec. 6 of \cite{GhSa}). Thus, we may choose $x_2$ to be odd in Lemma \ref{lift2} so that $\Delta$ is invertible in $D$, giving $X\in \Gamma_D$. So we need to find $Y\in \Gamma_D$ satisfying $\Tr AY\equiv \Tr Y\equiv x_2 \pmod{2^n}$. Writing $Y=\left[\begin{smallmatrix} x_2-y_1& y_2\\y_3&y_1 \end{smallmatrix}\right]$, we seek $y_j$ satisfying the two equations $y_1(x_2-y_1)-y_2y_3\equiv 1$ and $b_1(x_2-y_1)+b_3y_2+b_2y_3+b_4y_1\equiv 0 \pmod{2^{n-1}}$. Since $x_2$ is odd, we have $y_2$ and $y_3$ odd, so that the two equations combine with a need to solve $b_3y_2^2+(b_4-b_1)y_1y_2-b_2y_1^2+b_1x_2y_2+b_2x_2y_1\equiv b_2 \pmod{2^{n-1}}$ with $y_1$ free and $y_2$ odd. Since $b_2$ and $b_3$ are odd, there are no solutions here for even $y_2$'s and so we may drop the restriction that $y_2$ should be odd if $n\geq 2$. Completing the square, and using the fact that $b_1$ and $b_4$ are even, the equation to solve is equivalent to the equation $u^2 -\delta v^2 +\sigma v\equiv b_2b_3 + \left(\frac{b_1x_2}{2}\right)^2 \pmod{2^{n-1}}$, where $\delta=5\nu^2(5\nu^2 +1)$ is divisible exactly by 2, and $\sigma$ is odd. Using Gauss sums, one shows that this equation has $2^{n-1}$ solutions, giving the existence of $Y$ (alternatively one can appeal to Hensel's lemma, since the equation has trivially a non-singular solution modulo 2 that lifts to modulo $2^{n-1}$). When $n=1$, $A\equiv I$ is a commutator. 

For $l=3$, we are unable to use Lemma \ref{lift2} directly since $\Delta$ is always divisible by 3, but we are able to use the proof of the lemma to deduce our result. We start with a non-singular solution $(x_1,x_2,x_3)\equiv (3,3,3) \pmod 9$ and extend this using Hensel's lemma to modulo $3^{n+2}$, with $n\geq 0$ (see \cite{GhSa} Sec.6). Then, $\Delta=t+2-x_2^2\equiv -9 \pmod{27}$, so that $9\| \Delta$. The construction of $X$ is then valid with $|X|\equiv 1\pmod{3^{n-2}}$, this being so since in \eqref{matrix-X} we have 9 dividing the right-hand-side using $A\equiv -I \pmod 3$ and $3|x_j$ for all $j$. Thus, if $n\geq 2$, $A$ will be a commutator modulo $3^{n-2}$ provided we can find a $Y$ with $\Tr AY\equiv \Tr Y \pmod{3^{n+2}}$ with $|Y|\equiv 1 \pmod{3^{n+2}}$. To find $Y$, we follow the same strategy as the case above: we seek $y_i$ such that $y_1(x_2-y_1)-y_2y_3\equiv 1$ and $b_1(x_2-y_1)+b_3y_2+b_2y_3+b_4y_1\equiv 0 \pmod{3^{n+2}}$. Since $3|x_2$, we must have $3\nmid y_2y_3$ so that we are reduced to finding $y_1$ and $y_2$ satisfying  $\frac{b_3}{3}y_2^2+\frac{b_4-b_1}{3}y_1y_2-\frac{b_2}{3}y_1^2+b_1\frac{x_2}{3}y_2+\frac{b_2}{3}x_2y_1\equiv \frac{b_2}{3} \pmod{3^{n+1}}$ with $3\nmid y_2$. We may drop this latter condition as there are no solutions with $3|y_2$. Since $3\|b_3$, we complete the square and obtain an equation of the type $u^2 -\delta v^2 +\sigma v\equiv \tau \pmod{3^{n+1}}$, where $3|\delta$ and $3\nmid \sigma$. This is exactly as in the case above, and Hensel's lemma gives a solution lifted from a non-singular one modulo 3. Thus, the required $Y$ exists modulo $3^{n+2}$, so that $A$ is a commutator modulo $3^{n-2}$ for all $n\geq 2$.

\end{proof}

\begin{appendices}
\section*{Appendix}

\section{Preliminaries}\label{prelim}

We put here some general information that will be used repeatedly without comment.

\begin{notation} $D$ will be an integral domain and $F$ a field (finite or not). 

$\mathbb{M}_2(S)$ will denote $2\times2$ matrices with entries in the set $S$, and $\Gamma_S =\SL_2(S)$.

For $A= \left[\begin{smallmatrix} a&b\\c&d\end{smallmatrix}\right]\in \mathbb{M}_2(D)$, we put $A'= \left[\begin{smallmatrix} d&-b\\-c&a\end{smallmatrix}\right]$, and $I$ to denote the identity matrix.

$W(A,B)=ABA^{-1}B^{-1}$ for $A,\ B\ \in \Gamma_D$.\end{notation}

\begin{lemma}\label{ele1}
For $A,\ B \in \mathbb{M}_2(D)$, we have
\begin{enumerate}[label=\upshape(\roman*).]
\item $A+A'=(\Tr A) I$,
\item $AA'=|A|I$,
\item $A^2=(\Tr A)A-|A|I$\ \ and\ \  $A'^2=(\Tr A)A' -|A|I$,
\item $|A+B|= |A|+|B|+\Tr AB'$.
\end{enumerate}
\end{lemma}

\begin{proposition}\label{identity}
Let $X,\ Y\in\Gamma_D$, with $x_1=\Tr X$, $x_2=\Tr Y$ and $x_3=\Tr XY$. Then
\begin{equation}\label{Ap1.1}
W(X,Y) =x_3XY + (x_1 -x_2x_3)X + x_2Y^{-1} -I.
\end{equation}
\end{proposition}
\begin{proof} We have
\begin{align*}
W(X,Y)=XYX^{-1}Y^{-1}&=XY(-X+ x_1I)Y^{-1},\\&=-(XY)^{2}Y^{-2} +x_1X,\\&=-(x_3XY -I)(x_2Y^{-1} -I)+x_1X,
\end{align*}
from which the result follows on expanding.
\end{proof}

\begin{remark} Taking trace of \eqref{Ap1.1} gives the Fricke identity, and so the Markoff equation $M(x_1,x_2,x_3)=x_1^2+x_2^2+x_3^2-x_1x_2x_3=k$. We say ${\bf x}=(x_1,x_2,x_3)\in \mathcal{M}_k(D)$ with $k=\Tr W(X,Y)+2$.
\end{remark}

\begin{remark}
If $W(X,Y)\neq I$, then the coefficients in \eqref{Ap1.1} are uniquely determined in the following sense: if  $W(X,Y)=\alpha XY+\beta X+\delta Y^{-1}-I$, then $\alpha =\Tr XY$, $\beta=\Tr X -\Tr Y\Tr XY$ and $\gamma=\Tr Y$. This is because if we suppose not,  subtracting from \eqref{Ap1.1} gives $(x_3-\alpha)XY= s_1X + t_1Y^{-1}$, for some $s_1,\ t_1\in D$ not both zero. Since $W(X,Y)\neq I$, we have $\alpha \neq x_3$, so that $XY= s_2X + t_2Y^{-1}$ for some $s_2$ and $t_2$, so that $X(Y- s_2I)=t_2Y^{-1}$. If $|Y-s_2I|=0$, we have $t_2=0$, so that $Y=s_2I$, giving $W(X,Y)=I$. Hence, $X=t_2Y^{-1}(Y-s_2I)^{-1}$. But $(Y-s_2I)(Y^{-1}-s_2I)=(s_2^2-x_2s_2+1)I\neq 0$. Hence $X$ is a linear combination of $Y^{-1}$ and $Y^{-2}$, and so commutes with $Y$, giving $W(X,Y)=I$. This provides the contradiction.
\end{remark}

\begin{definition}\label{traceset}
For a matrix $A$ defined over $D$, we let $\mathfrak{S}(A)$ denote the set over $D$ given by
\[
\mathfrak{S}(A)=\{X : \Tr AX=\Tr X \ , |X|=1\}.
\]
\end{definition}

\begin{lemma}\label{traceq}
Let $X,\ Y,\ Z\in \Gamma_D$. Suppose $D$ is an integral domain and $\Tr Z\neq 2$ or $D$ is a ring with $\Tr Z -2\in D^{\times}$. Then
\[
Z=W(X,Y)\quad \text{if and only if}\quad Y\in \mathfrak{S}(Z),\ \text{and}\  \ X,\  XY\in \mathfrak{S}(Z^{-1}).
\]
\end{lemma}
\begin{proof} That the left implies the right is immediate. 

We now assume the right and put $\Tr Z=t$. We convert the trace equations on the right to matrix form using Lemma \ref{ele1}, giving

\begin{enumerate}[label=(\roman*).]
\item $ZY= Y+Y^{-1}-Y^{-1}Z^{-1}$\ ,
\item $Z^{-1}X=X+X^{-1}-X^{-1}Z$\ ,\  and
\item $XY + Y^{-1}X^{-1}= Z^{-1}XY + Y^{-1}X^{-1}Z$\ .
\end{enumerate}
Substituting the first two into the third, and using Lemma \ref{ele1}(i) to replace $Z^{-1}$ with $Z$, and simplifying, gives 
\[
(Y^{-1}X^{-1}-X^{-1}Y^{-1})Z = Y^{-1}X^{-1} +(1-t)X^{-1}Y^{-1}.
\]
Multiplying with $XY$ and putting $W(X,Y)=W$ gives
\[
(I-W)Z=I+(1-t)W\ .
\]
All of the above is true for any $t$. In what follows, we assume $t\neq 2$. 

We take determinants of both sides: $|I-W|=2-\Tr W$, and $|I+(1-t)W|= 1+(1-t)^2 +(1-t)\Tr W$. Equating the two gives $\Tr W=t$ . Then multiplying both sides of the equation with $(I-W')$ and simplifying gives $(2-t)Z=(2-t)W$, from which the result follows.
\end{proof}

\begin{remark}\label{tracezero} Combining Lemmas \ref{identity} and \ref{traceq} shows that $-I$ is a commutator in $D$ if and only if there exist matrices $X$ and $Y\in \Gamma_D$ such that $\Tr X=\Tr Y= \Tr XY=0$, in which case $-I = W(X,Y)$. Finding such $X$ and $Y$ satisfying the trace equations is problematic in general as it requires solving simultaneously three quadratic equations over $D$, coming from $\Tr XY=0$ and $|X|=|Y|=1$. It is straightforward to see that $-I$ is not a commutator in $\SLR$, and hence not in any subring.

For a PID, one can say more. Suppose $r_1^2+r_2^2+r_3^2=0$ with $r_j\in D$ not all zero. Choose $r_1$ and $r_2$ so that $\gcd(r_1,r_2)=1$. Pick $u$ and $v\in D$ such that $r_1u-r_2v=1$. Put $X= \left[\begin{smallmatrix} ur_3&u^2r_2+v(r_1u+1)\\r_2&-ur_3\end{smallmatrix}\right]$ and $Y= \left[\begin{smallmatrix} vr_3&v^2r_1+u(r_2v-1)\\r_1&-vr_3\end{smallmatrix}\right]$, both in $\Gamma_D$. Then $W(X,Y)=-I$. Conversely, given the latter, by equating the traces and determinants, we get a non-trivial solution to the sum of three squares being zero. 
\end{remark}

\section{Local analysis}\label{App.B}\

\begin{lemma}[Elementary]\label{lemB1} Let $p$ be an odd prime and $\chi(.)$ the Legendre symbol mod $p$. 
\begin{enumerate}[label=\upshape(\roman*).]
\item The number  of solutions to $x^2-\Delta y^2\equiv n$ (mod $p$) is  
\[
= \left\{ \begin{array}{ll}
 [1+\chi(n)]p & \text{if}\ \ p\nmid n\ \text{and}\ p|\Delta \ , \\
 p-\chi(\Delta) & \text{if}\ \ p\nmid n\ \text{and}\ p\nmid\Delta \ , \\
 p[1+\chi(\Delta)] -\chi(\Delta) & \text{if}\ \ p|n\ . \\
\end{array} \right. 
\]
\item Let $\Delta=b^2-4ac$. Then
\[
\sum_{m,n\ ({\rm mod} p)}e_p(am^2+bmn+cn^2)= \left\{ \begin{array}{ll}
 \chi(\Delta)p & \text{if}\ \ p\nmid \Delta \ , \\
 \chi(a)pS_p & \text{if}\ \ p\nmid a\ \text{or}\ c, \ \text{and}\ p|\Delta \ , \\
 p^2 & \text{if}\ p|a\ ,b\ \text{and}\ c\ , \\
\end{array} \right. 
\]
\end{enumerate}
where $S_p = \sum_{x}e_p(x^2)$ is the Gauss sum.
\end{lemma}

\begin{lemma}\label{corB1} For $p\geq 3$ a prime, let $\Gamma=\SL_2(\mathbb{Z}/p\mathbb{Z})$.  For any $\zeta$ with $p\nmid \zeta$, let $A= \left[\begin{smallmatrix} a&b\\ c&d \end{smallmatrix}\right]$ and $B= \left[\begin{smallmatrix} a&b\zeta\\ c\zeta^{-1}&d \end{smallmatrix}\right]$ be in $\Gamma$. Then, if $A\neq \pm I$, $\Tr A\equiv \pm 2$ and $\zeta$ is a QNR (mod $p$), $A$ and $B$ are not conjugates in $\Gamma$, and otherwise  $A$ and $B$ are $\Gamma$-conjugates.
\end{lemma}

\begin{proof} We seek $\gamma=\left[\begin{smallmatrix} \gamma_1&\gamma_2\\ \gamma_3&\gamma_4 \end{smallmatrix}\right]$ in $\Gamma$ satisfying $\gamma A=B\gamma$. If $p|b$ and $p|c$ , take $\gamma=I$. So, we assume $p\nmid c$; otherwise we transpose everything. This leads to the equation 
\[c\gamma_1^2 -(d-a)\zeta\gamma_1\gamma_3 -b\zeta^2\gamma_3^2=c\zeta,
\]
with discriminant $\Delta=(\Tr A)^2-4$, with $\gamma_2$ and $\gamma_4$ determined. We apply Lemma \ref{lemB1} with $n=4c^2\zeta$.

 If $\Tr A \equiv \pm2$ (mod $p$), and if $\zeta$ is a quadratic nonresidue, there are no solutions so that $A$ and $B$ are not conjugates. Otherwise there are solutions and $A$ and $B$ are $\Gamma$-conjugates.
\end{proof}

\begin{lemma}\label{corB2} For $p\geq 3$ a prime and $\Gamma=\SL_2(\mathbb{Z}/p\mathbb{Z})$, if $t \not\equiv \pm 2$ (mod $p$), there is exactly one conjugacy class with trace $t$ in $\Gamma$, with representative $\left[\begin{smallmatrix} 0&1\\ -1&t \end{smallmatrix}\right]$.
\end{lemma}
\begin{proof} Let $A= \left[\begin{smallmatrix} a&b\\ c&d \end{smallmatrix}\right] \in \Gamma$. First, we verify that $A$, $A^{-1}$ and $A^\intercal$ are all $\Gamma$-conjugates. Conjugating with $S= \left[\begin{smallmatrix} 0&1\\ -1&0 \end{smallmatrix}\right]$ shows that $A^{-1}$ and $A^\intercal$ are conjugates. Then, if $p\nmid bc$, taking $\zeta=\frac{c}{b}$ in Lemma \ref{corB1} shows that $A$ is conjugate to them too. If $p\nmid b$ and $p|c$, we have $p\nmid(a-a^{-1})$ and conjugating $A$ with $ \left[\begin{smallmatrix} b&(1+b^2)/(a-a^{-1})\\ -(a-a^{-1}) &-b \end{smallmatrix}\right]$ gives $A^{-1}$.

Now if $p\nmid b$ or $p\nmid c$, conjugating $A$ (or its transpose) with an upper triangular matrix gives rise to a matrix of the form $\left[\begin{smallmatrix} 0&b\\ -b^{-1}&t \end{smallmatrix}\right] \in \Gamma$, which by Lemma \ref{corB1} is equivalent to $\left[\begin{smallmatrix} 0&1\\ -1&t \end{smallmatrix}\right]$. If $p|b$ and $p|c$, conjugating with $\left[\begin{smallmatrix} 1&1\\ 0&1 \end{smallmatrix}\right] \in \Gamma$ gives $\left[\begin{smallmatrix} a&a-d\\ 0&d \end{smallmatrix}\right]$. Since $p\nmid a-d$, as otherwise $\Tr A = \pm 2$, the argument above again leads to the matrix $\left[\begin{smallmatrix} 0&1\\ -1&t \end{smallmatrix}\right]$.
\end{proof}

\begin{remark}This lemma is related to the following. 

Let $K$ be a field of characteristic not 2, $Z\in \Gamma_F$, $\Tr Z=t$ with $t\neq \pm 2$. Then, $Z \ \text{or}\ Z^{-1} \sim \left[\begin{smallmatrix} 0&\alpha\\ -\alpha^{-1}&t \end{smallmatrix}\right]$ for some $\alpha \neq 0$. 

If $F$ is a field with the following property: there exists an $\omega\in F^{\times}/\square_F$ such that $F=\square_F\cup w\square_F$, where $\square_F$ denotes the squares in $F$. Then $Z \ \text{or}\ Z^{-1} \sim \left[\begin{smallmatrix} 0&1\\ -1&t \end{smallmatrix}\right]$ if $x^2-(t^2-4)y^2=\omega$ is solvable with $(x,y)\in F^2$. Otherwise $Z \ \text{or}\ Z^{-1} \sim \left[\begin{smallmatrix} 0&1\\ -1&t \end{smallmatrix}\right] \ \text{or}\  \left[\begin{smallmatrix} 0&\omega\\ -\omega^{-1}&t \end{smallmatrix}\right]$, the latter two being inequivalent.

\end{remark}

\end{appendices}

\footnotesize
\setlength{\parskip=0.0pt}
\setlength{\lineskip=0.0pt} 
\bibliographystyle{plain}

\bibliography{references}

\end{document}